\documentclass[11pt]{article}
\usepackage[dvipdfmx]{graphicx}
\usepackage{amsmath}
\usepackage{amssymb}
\usepackage{amsthm}
\usepackage{bm}
\usepackage{url}
\usepackage{mathrsfs}
\usepackage[usenames]{color}

\newtheorem{dfn}{Definition}[section]
\newtheorem{thm}[dfn]{Theorem}
\newtheorem{prop}[dfn]{Proposition}
\newtheorem{lem}[dfn]{Lemma}
\newtheorem{cor}[dfn]{Corollary}
\newtheorem{rem}[dfn]{Remark}

\newcommand{\expig}{e^{\mathrm{i}\theta}}
\newcommand{\expmig}{e^{-\mathrm{i}\theta}}

\newcommand{\R}{\mathbb{R}}
\newcommand{\Z}{\mathbb{Z}}
\newcommand{\N}{\mathbb{N}}
\newcommand{\C}{\mathbb{C}}

\newcommand{\ba}{\bar{a}}
\newcommand{\bU}{\bar{U}}

\newcommand{\cB}{\mathcal{B}}
\newcommand{\cE}{\mathcal{E}}
\newcommand{\cI}{\mathcal{I}}

\newcommand{\cN}{\mathcal{N}}
\newcommand{\cP}{\mathcal{P}}
\newcommand{\cT}{\mathcal{T}}
\newcommand{\cX}{\mathcal{X}}

\newcommand{\ta}{\tilde{a}}
\newcommand{\bydef}{\,\stackrel{\mbox{\tiny\textnormal{\raisebox{0ex}[0ex][0ex]{def}}}}{=}\,} 

\newcommand{\im}{\mathrm{i}}
\newcommand{\xx}{\mathrm{x}}
\newcommand{\bfj}{{\boldsymbol j }}
\newcommand{\bfi}{{\boldsymbol i }}

\newcommand{\fL}{\mathfrak{L}}

\numberwithin{equation}{section}
\usepackage[top=1in, bottom=1in, left=1in, right=1in]{geometry}

\begin{document}
\title{
Rigorous numerics for nonlinear heat equations in the complex plane of time
}
\author{%
Akitoshi~Takayasu\thanks{%
	Faculty of Engineering, Information and Systems, University of Tsukuba, 1-1-1 Tennodai, Tsukuba, Ibaraki 305-8573, Japan (\texttt{takitoshi@risk.tsukuba.ac.jp})
}\and
Jean-Philippe~Lessard\thanks{%
	Department of Mathematics and Statistics, McGill University,
%	805 Sherbrooke West,
	Montreal, QC H3A 0B9, Canada
}\and
Jonathan~Jaquette\thanks{%
	Department of Mathematics, Brandeis University,
	Waltham, MA 02453, US
}\and
Hisashi~Okamoto\thanks{%
	Department of Mathematics, Gakushuin University, Tokyo 171-8588, Japan
}
}
\date{\today}
\maketitle
\begin{abstract} % Abstract Limit 150 words
% In this paper, we introduce a method for computing rigorous local inclusions of solutions of Cauchy problems for the nonlinear heat equation for complex time values. The proof is constructive and provides explicit bounds for the inclusion of the solution of the Cauchy problem, which is rewritten as a zero-finding problem on a certain Banach space. Using a solution map operator, we construct a simplified Newton operator and show that it has a unique fixed point. The fixed point together with its rigorous bounds provide the local inclusion of the solution of the Cauchy problem. The existence of the solution map operator is guaranteed by computing rigorously with Chebyshev series the solutions of finite dimensional linearized problems and by proving the existence of the evolution operator on the tail. The local inclusion technique is then applied iteratively to compute solutions over long time intervals. This technique is used to prove the existence of a branching singularity in the nonlinear heat equation. Finally, we introduce an approach based on the Lyapunov-Perron method for calculating part of a center-stable manifold, which is then used to prove that an open set of solutions of the Cauchy problem converge to zero, hence yielding the global existence of the solutions in the complex plane of time.
In this paper, we introduce a method for computing rigorous local inclusions of solutions of Cauchy problems for nonlinear heat equations for complex time values. Using a solution map operator, we construct a simplified Newton operator and show that it has a unique fixed point. The fixed point together with its rigorous bounds provides the local inclusion of the solution of the Cauchy problem. The local inclusion technique is then applied iteratively to compute solutions over long time intervals. This technique is used to prove the existence of a branching singularity in the nonlinear heat equation. Finally, we introduce an approach based on the Lyapunov-Perron method to  calculate part of a center-stable manifold and prove that an open set of solutions of the Cauchy problem converge to zero, hence yielding the global existence of the solutions in the complex plane of time.
\end{abstract}

{\bf Keywords : } blow-up solutions for nonlinear heat equations, branching singularity, global existence of solution, Lyapunov-Perron method, contraction mapping, rigorous numerics
%\par
%\bigskip
%{\bf AMS subject classifications : } 35A20, 35B40, 35B44, 35K55, 65G40, 65M15, 65M70

%%%%%%%%%%%%%%%%
%\tableofcontents

%%%%%%%%%%%%%%%
%%% INTRODUCTION %%%
%%%%%%%%%%%%%%%

\section{Introduction}
%%%%%%%%%%%%%%%%
%Blow-up phenomena appears everywhere in nature.
%For example, 
In this paper, we consider solutions of a nonlinear heat equation with the following setting:
\begin{align}\label{eq:nheq}
\begin{cases}
	u_t=u_{xx}+u^2, & x\in(0,1), ~t\ge 0,\\
	u(0,x)=u_0(x), & x\in(0,1)
\end{cases}
\end{align}
with the periodic boundary condition in $x$ variable.
%Here, $u_t$ and $u_x$ denote the derivative with respect to $t$ and $x$, respectively.
Here the subscripts denote the derivatives in the respective variables and $u_0(x)$ is a given initial data.
%In particular, we set specific initial data $u_0(x)=50(1-\cos(2\pi x))$.
It is well-known that solutions of the nonlinear heat equation \eqref{eq:nheq} may blow up in finite time, that is the $L^{\infty}$-norm of the solution tends to $\infty$ as $t$ tends to a certain $B<\infty$.
Such a $B$ is called the \emph{blow-up time} of \eqref{eq:nheq}.
There are plenty of studies for blow-up problems of nonlinear heat equations.
One can consult previous studies by, e.g., 
%\cite{DENG200085,bib:FM2002,doi:10.1137/1032046,zbMATH05152729} 
\cite{bib:FM2002,zbMATH05152729} 
and references therein.

We extend the time variable of \eqref{eq:nheq} into the complex plane.
Changing the $t$ variable into $z$, the aim of this paper is to study dynamics of the complex-valued nonlinear heat equation
\begin{align}\label{eq:cnheq}
%\begin{cases}
u_z=u_{xx}+u^2,\quad x\in(0,1), \quad\mathrm{Re}(z)\ge 0,
%u(0,x)=u_0(x), & x\in(0,1)
%\end{cases}
\end{align}
under the periodic boundary condition in $x$ with initial data $u(0,x)=u_0(x)$.
Here, the subscript $z$ denotes the complex-derivative with respect to $z$ and $\mathrm{Re}$/$\mathrm{Im}$ denote the real/imaginary part of complex values, respectively.

%As a pioneering work of this problem, Masuda \cite{MASUDA1983119,Masuda1984} has
In his pioneering works \cite{MASUDA1983119,Masuda1984} for this problem, Masuda considered the solution of \eqref{eq:cnheq} under the Neumann boundary condition and proved global existence of the solution in the shaded domain of Fig.~\ref{fig:fig1}\,$(a)$ if the initial data $u_0(x)$ is close to a constant.
Masuda has also shown that the solution is analytic in both the shaded domain and its mirror-image about the real axis (see Fig.~\ref{fig:fig1}\,$(b)$). 
Furthermore, it is proved that the initial data is a constant if the solution agrees in the intersection of the two domains, which is shown in Fig.~\ref{fig:fig1}\,$(b)$.
This result implies that a non-constant solution is not a single-valued function in $\mathrm{Re}(z)> z_B$, where $z_B\equiv B$ denotes the blow-up time.
It indicates that the singularity is a branching point.

\begin{figure}[htbp]\em
	\centering
	\begin{minipage}{0.49\hsize}
		\centering
		\includegraphics{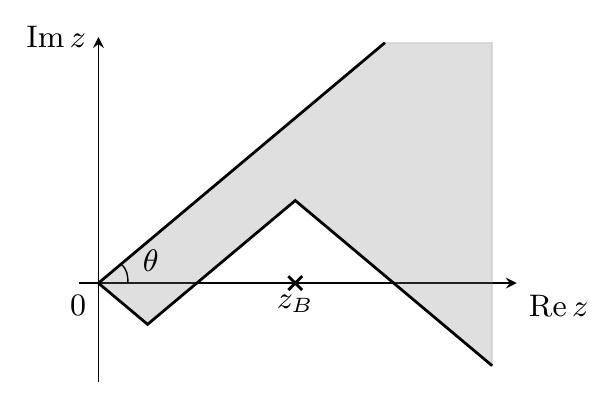}\\
		$(a)$
	\end{minipage}
	\begin{minipage}{0.49\hsize}
		\centering
		\includegraphics{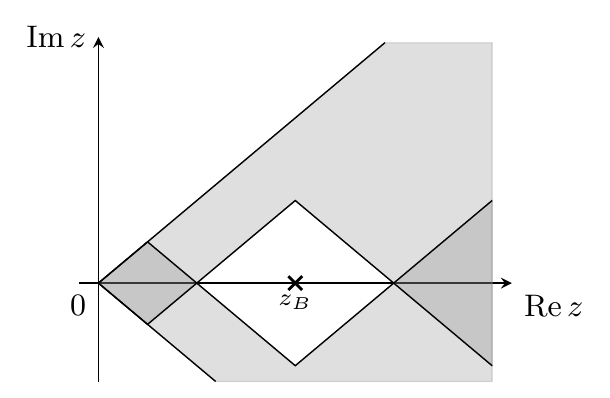}\\
		$(b)$
	\end{minipage}
	\caption{%
		$(a)$~The shaded domain in which global existence of the solution is proved by Masuda \cite{MASUDA1983119,Masuda1984}. Here, $0<\theta<\pi/2$ and $z_B$ denotes the blow-up time.
%		The real singularity is bypasseed and the domain extends to infinity.
		$(b)$~The shaded domain and its mirror-image about the real axis: intersection of two domains is plotted in dark gray.
	}
	\label{fig:fig1}
\end{figure}

%\corrc Add more of nonlinear heat equations. Akitoshi <<>>
Following Masuda, Cho et al.~\cite{Cho2016} numerically tested dynamics of the solution of \eqref{eq:cnheq} under the periodic boundary condition.
They have shown that the solution may converge to the zero function on a straight path 
$\Gamma_\theta \bydef \{ z \in \mathbb{C} :  z=t\expig, \quad t \ge 0 \}$ where $\theta \in \left( -\frac{\pi}{2},\frac{\pi}{2} \right)$ and $\im =\sqrt{-1}$.
They have also shown that the solution of \eqref{eq:cnheq} may have only one singularity on the real axis, which branches the Riemann surface of complex-valued solutions.
As conclusions of their study, two conjectures have been proposed.
\begin{itemize}
	\item The analytic function defined by the nonlinear heat equation \eqref{eq:cnheq} has branching singularities and only branching singularities, unless it is constant in $x$,
	\item The nonlinear Schr\"odinger equation, which is the case of $\theta = \pm\pi/2$, is globally well-posed for any real initial data, small or large.
\end{itemize}
These results were based on numerical observations and no mathematical proof was presented.
The authors in \cite{Cho2016} said that
\begin{quote}\em
	``in fact if we wish to add anything new in a rigorous way, we encounter a serious difficulty. For instance, we tried unsuccessfully to prove global existence without assuming closeness to a constant.''
\end{quote}
This remark and the above conjectures consist of our main motivation for the current paper.

The contribution of the current paper is to give two computer-assisted proofs for the complex-valued nonlinear heat equation \eqref{eq:cnheq}, which extends the results by Masuda \cite{MASUDA1983119,Masuda1984} and mathematically proves the results of numerical observations by Cho et al.~\cite{Cho2016}.
Our first result shows the existence of a branching singularity, which appears at the blow-up point Cho et al.~\cite{Cho2016} calculated to be approximately   $z_B \approx 0.0119$.  

\begin{thm}[Existence of branching singularity]\label{thm:branching}
For the complex-valued nonlinear heat equation \eqref{eq:cnheq} under the periodic boundary condition with the specific initial data $u_0(x)=50(1-\cos(2\pi x))$, there exists a branching singularity at $z_B\in (0.0116,0.0145)$.  
\end{thm}
%For comparison,  observe numerically that the branching singularity occurs at approximately $z_B \approx 0.0119$.  
%We are able to more closely approximate this value from below, as this only requires rigorous integration of a real scalar PDE, whereas the upper bound requires the rigorous integration of a complex scalar PDE on a contour of twice the length. 
%This result partially answers the first conjecture of Cho et al.

Our second result agrees with Masuda's work for the case of periodic boundary conditions without assumption of closeness to a constant for the initial data.
\begin{thm}[Global existence]\label{thm:GE}
	For $\theta=\pi/3$, $\pi/4$, $\pi/6$ and $\pi/12$, setting a straight path $\Gamma_\theta: z=t\expig$ $(t\ge 0)$ in the complex plane of time, the solution of the complex-valued nonlinear heat equation \eqref{eq:cnheq} under the periodic boundary condition with the specific initial data $u_0(x)=50(1-\cos(2\pi x))$ exists globally in $t$ and converges to the zero function as $t\to\infty$.
\end{thm}

The proofs of these results are obtained in Section \ref{sec:Numerical_results} by using {\em rigorous numerics}, via a careful blend of functional analysis, semi-group theory, numerical analysis, fixed point theory, the Lyapunov-Perron method and interval arithmetic. It is worth mentioning that the study of blow-ups and global existence of solutions in differential equations is beginning to be studied with the tools of rigorous numerics (e.g.\ see \cite{MR3339076,MR3348927,matsue_takayasu_blowup,MIZUGUCHI20171,MR3575546,BergPREPApproximating}).
Furthermore, several approaches of rigorous numerics for time-dependent PDEs have been studied in \cite{Cyranka2014,Kinoshita2014,doi:10.1137/141001664,zgliczy2009covering}, etc.
We expect that these rigorous integrators will be applied to understand dynamics of time-dependent PDEs soon.

The present paper is organized as follows.
In Section \ref{sec:setting} we set up the fixed point operator to solve the Cauchy problem \eqref{eq:cnheq} on a straight path in the complex plane.
Using the Fourier series, we derive an infinite-dimensional system of differential equations and define a formulation of the fixed point operator to rigorously provide an enclosure of the solution on a short time interval.
Section \ref{sec:solution_operator} is devoted to defining a solution map operator, which provides a solution of the linearized Cauchy problem with arbitrary forcing term.
Such a solution map operator is defined by what is called the evolution operator.
We show how we verify the existence of such an evolution operator by using rigorous numerics and give a computable estimate for a uniform bound of the evolution operator.
Using the solution map operator, we derive in Section \ref{sec:Localinclusion} a sufficient condition to have a fixed point for the fixed point operator.
Theorem \ref{thm:main_theorem} provides rigorous enclosure of the solution of the Cauchy problem. In Section~\ref{sec:timestepping}, we introduce a method for applying iteratively the local inclusion approach to compute solutions over long time intervals. This technique is then used to prove Theorem~\ref{thm:branching}. In Section~\ref{sec:center-stable-manifold}, we introduce an approach based on the Lyapunov-Perron method in order to calculate part of a center-stable manifold, which is used to prove that an open set of solutions of the given Cauchy problem converge to zero, hence yielding the global existence in time of the solutions. This proves our second main result, namely Theorem~\ref{thm:GE}. The numerical results are presented in details in Section~\ref{sec:Numerical_results}. 

%%%%%%%%%%%%%%%%%%%%%%%%%%%%
%%% fixed point OPERATOR FORMULATION %%%
%%%%%%%%%%%%%%%%%%%%%%%%%%%%

\section{Setting-up the fixed point operator to solve the Cauchy problem}\label{sec:setting}

In this section, we derive the formulation of the fixed point operator which is used to provide a rigorous enclosure for the solution of the Cauchy problem on a short time interval. 
Taking the straight path $\Gamma_\theta = \{ z \in \mathbb{C} :  z=t\expig, \quad t \ge 0 \}$, the complex-valued nonlinear heat equation \eqref{eq:cnheq} is transformed into the equation
\begin{align}\label{eq:CGL}%\tag{CGL}
u_t=\expig\left(u_{xx}+u^2\right),\quad x\in(0,1),\quad t\ge 0
\end{align}
with periodic boundary conditions. Here, we set $\theta\in(-\pi/2,\pi/2)$ and the initial data $u_0(x)=50(1-\cos(2\pi x))$. We expand the unknown function into the Fourier series:
\begin{align}\label{eq:Fourier_expansion}
u(t,x)=\sum_{k\in\Z }a_k(t)e^{\im k\omega x},
\end{align}
where we have set $\omega = 2\pi$.
Plugging \eqref{eq:Fourier_expansion} in the initial-boundary value problem \eqref{eq:CGL}, we have the following infinite-dimensional system of ODEs:
\begin{align}\label{eq:CGL_ode}
\frac{d}{dt}a_k(t)=\expig\left[-k^2\omega^2a_k(t)+\left(a\left(t\right)*a\left(t\right)\right)_k\right]\quad (k\in\Z ),\quad a(0)=\varphi,
\end{align}
where ``$*$'' denotes the discrete convolution product defined by
\[
(b*c)_k\bydef\sum_{m\in\mathbb{Z}}b_{k-m}c_m\qquad (k\in\mathbb{Z})
\]
for bi-infinite sequences $b=(b_k)_{k\in\mathbb{Z}}$ and $c=(c_k)_{k\in\mathbb{Z}}$,
and $\varphi$ is defined by
\[
	\varphi_k=\left\{\begin{array}{cl}
	-25, & k=\pm1\\
	50, & k=0\\
	0, & \mbox{otherwise}.
	\end{array}\right.
\]

%From the Cauchy-Kovalevskaya theorem (e.g. see \cite{evans}), the solution of \eqref{eq:CGL} is expected to be analytic until some $t>0$. We consider the Fourier coefficients of the solution which are expected to be in the sequence space $\ell^1$ based on the results of Paley and Wiener. %'s result \cite{PW}.

For a fixed time $h>0$, let us define $J\bydef(0,h)$ and the Banach space
\begin{equation}\label{eq:X_space}
X\bydef C(J;\ell^1),\quad \|a\|_X\bydef\sup_{t\in J}\|a(t)\|_{\ell^1},
\end{equation}
where
$\ell^1\bydef\left\{a=(a_k)_{k\in\Z }:\sum_{k\in\Z }|a_k|<\infty,~a_k\in\C \right\}$ with norm $\|a\|_{\ell^1}\bydef\sum_{k\in\Z }|a_k|$.
An important and useful feature of $\ell^1$ is that it is a Banach algebra under discrete convolution, namely
\begin{equation} \label{eq:banach_algebra}
\| a * b \|_{\ell^1} \le \| a \|_{\ell^1} \| b \|_{\ell^1}, \quad \text{for all } a,b \in \ell^1. 
\end{equation}
Our task of rigorous numerics is to determine the Fourier coefficients $(a_k(t))_{k\in\Z }$ of the solution of the Cauchy problem \eqref{eq:CGL_ode}.	
By numerical computation, one can  have $(\ba_k(t))_{|k|\le N}$, which is an approximation of the Fourier coefficients with maximal wave number $N$.
Let $$\ba(t)\bydef(\dots,0,0,\ba_{-N}(t),\dots,\ba_N(t),0,0,\dots)$$ be an approximation of $a(t)$.
We will rigorously include the Fourier coefficients in the neighborhood of numerical solution defined by
\begin{equation}\label{eq:Ball}
B_{J}(\ba,\varrho)\bydef\left\{a\in X:\|a-\ba\|_{X}\le \varrho,~a(0)=\varphi\right\}.
\end{equation}

Define the {\em Laplacian} operator $L$ acting on a sequence of Fourier coefficients as
\begin{equation}\label{eq:mul_op}
Lb\bydef\left(-k^2\omega^2b_k\right)_{k\in\Z },~b=(b_k)_{k\in\Z }.
\end{equation}
The domain of the operator $L$ is denoted by $D(L)\bydef\left\{a=(a_k)_{k\in\Z }:\sum_{k\in\Z }k^2|a_k|<\infty\right\}\subset\ell^1$. We define an operator acting on $a\in C^1(J;D(L))$ by
\begin{equation} \label{eq:F=0_Cauchy_problem}
(F(a))(t) \bydef \frac{d}{dt}a(t)-\expig\left(La(t)+a\left(t\right)*a\left(t\right)\right).
\end{equation}
Then we consider the Cauchy problem \eqref{eq:CGL_ode} as the zero-finding problem $F(a)=0$.
It is obvious that $a$ solves \eqref{eq:CGL_ode} if $F(a)=0$ with the initial condition $a(0)=\varphi$ holds.
Let us define the following operator:
\begin{align}\label{eq:simp_Newton_op}
(T(a))(t)\bydef \mathscr{A}\left[\expig\big(a(t)*a(t)-2\ba(t)*a(t)\big)\right], \quad T:X\to X.
\end{align}
As shown in Remark \ref{rem:simplified_newton_op}, this operator is alternative  form of the simplified Newton operator.
We expect that the simplified Newton operator \eqref{eq:simp_Newton_op} has a fixed point $\ta\in B_{J}(\ba,\varrho)$ such that $\ta=T(\ta)$.
Here, $\mathscr{A}$ is an operator to be explicitly defined in the next section.

\begin{rem}\label{rem:simplified_newton_op}
	Suppose $\ba\in C^1(J;D(L))$,
	we take $\mathscr{A}^{\dagger}$ as the Fr\'echet derivative of $F:C^1(J;D(L))\to X$ at $\ba$, say $\mathscr{A}^{\dagger}=DF[\ba]$. It is easy to see that the operator $Id$ satisfies $\mathscr{A}\mathscr{A}^{\dagger}={\rm Id}$ (the identity operator in $X$) with the initial condition.
	Then, the simplified Newton operator satisfies
\begin{align*}
	T(a)&=\mathscr{A}\left[\expig\left(a*a-2\ba*a\right)\right]\\
	&= \mathscr{A}\left[\frac{d}{dt}a-\expig\left(La+2\ba*a\right)-\frac{d}{dt}a+\expig\left(La+a*a\right)\right]\\
	&= \mathscr{A}\left(\mathscr{A} ^{\dagger}a-F(a)\right)= a-\mathscr{A}F(a)
\end{align*}
for $a\in C^1(J;D(L))$ with $a(0)=\varphi$.
We note that this form $a-\mathscr{A}F(a)$ is defined only on $C^1(J;D(L))$ but the simplified Newton operator \eqref{eq:simp_Newton_op} can be defined on $X$.
Moreover, from the property of the operator $\mathscr{A}:X\to C^1(J;D(L))\subset X$, which will be shown in Remark \ref{rem:bootstrap} in the next section, the fixed point satisfies $\ta\in C^1(J;D(L))$ if such a fixed point is obtained.
\end{rem}

%%%%%%%%%%%%%%%%%%%%%%%%%%%%
%%% EVOLUTION OPERATOR FORMULATION %%%
%%%%%%%%%%%%%%%%%%%%%%%%%%%%

%%%%%%%%%%%%%%%%
\section{The solution map operator \boldmath$\mathscr{A}$\unboldmath} \label{sec:solution_operator}
%%%%%%%%%%%%%%%%
The simplified Newton operator \eqref{eq:simp_Newton_op} is characterized by the solution map operator $\mathscr{A}$.
We define $\mathscr{A}$ using an evolution operator $U(t,s)$ of the following homogeneous initial value problem (IVP):
\begin{align}\label{eq:linearized_problem}
\frac{d}{dt}b_k(t)+\expig\left[k^2\omega^2b_k(t)-2\left(\ba\left(t\right)*b(t) \right)_k\right]=0\quad(k\in\Z )
\end{align}
for any initial sequence $b(s)=\phi\in\ell^1$ ($0\le s\le t$). More explicitly, $U(t,s)$ is an evolution operator generated on $\ell^1$ and provides the solution of \eqref{eq:linearized_problem} via the relation $b(t)=U(t,s)\phi$. %Moreover, it is an evolution operator on the Banach space $\ell^1$. 
Consider \eqref{eq:linearized_problem} with an arbitrary forcing term $g(t) = (g_k(t))_{k \in \Z}$:
\begin{align}\label{eq:linearized_problem_non_homogeneous}
\frac{d}{dt}b_k(t)+\expig\left[k^2\omega^2b_k(t)-2\left(\ba\left(t\right)*b(t) \right)_k\right]= g_k(t)\quad(k\in\Z ).
\end{align}
This leads to the definition of the solution map operator $\mathscr{A}:X\to C^1(J;D(L))\subset X$ given by  
%such that  we get a solution representation of such an inhomogeneous IVP as
\begin{equation} \label{eq:definition_of_scriptA}
	(\mathscr{A}g)(t) \bydef U(t,s)\phi+\int_s^tU(t,\tau)g(\tau)d\tau.
\end{equation}
In particular, $\mathscr{A}0 =  U(t,s)\phi$.

%We use this form in the followings.

The remaining tasks of this section consist of constructing the solution map operator \eqref{eq:definition_of_scriptA}. This requires showing the existence of the evolution operator $\left\{U(t,s)\right\}_{0\le s\le t\le h}$ by obtaining a {\em computable} constant $\bm{W_h} > 0$ satisfying
\[
%\|b\|_{X}=
\sup_{0\le s\le t\le h}\|U(t,s)\phi\|_{\ell^1}\le \bm{W_h} \|\phi\|_{\ell^1},\quad\forall\phi\in\ell^1.
\]
More precisely, the constant $\bm{W_h}>0$ provides a uniform bound for the bounded linear operator norm of $U(\cdot,\cdot)$ over the simplex $\mathcal{S}_h \bydef \{(t,s)~:~0\le s\le t\le h\}$, that is 
\begin{equation}\label{eq:W_h_constant}
\| U(t,s) \|_{B(\ell^1)} \le \bm{W_h}, \quad \forall~ (t,s) \in \mathcal{S}_h.
\end{equation}
Here, we use the notation $B(\ell^1)$  to denote the set of all bounded linear operators on $\ell^1$ with corresponding bounded operator norm $\|\cdot \|_{B(\ell^1)}$.

\begin{rem}
It is important to notice that for a given $\ba(t) \in \ell^1$ and $h$ small enough, there exists $\bm{W_h}$ satisfying \eqref{eq:W_h_constant} (e.g.~see \cite{pazy1983semigroups}). This in turns yields the existence of the evolution operator $U(t,s)$ for the linearized problem \eqref{eq:linearized_problem} for all $(t,s) \in \mathcal{S}_h$ for a small enough step size $h>0$. However, once a fixed $h>0$ has been chosen, it is in general difficult to obtain explicitly $\bm{W_h}$ and therefore show the existence of the evolution operator $U(t,s)$ for  $(t,s) \in \mathcal{S}_h$. In the next section, we derive an explicit and computable formulation for the constant $\bm{W_h}$.
\end{rem}

\subsection{Deriving a formulation for \boldmath$W_h$\unboldmath~and existence of the evolution operator}

We begin this section with some notation. Given a finite dimensional (Fourier) projection number $m \in \N$ and a vector $\phi =\left(\phi_{k}\right)_{k \in \Z } \in \ell^1$, define the projection $\pi^{(m)}:\ell^1 \to \ell^1$ as follows:
\[
\left( \pi^{(m)} \phi \right)_k = \begin{cases} \phi_k, & |k| \le m \\ 0, & |k| > m, \end{cases}
\]
for $k \in \Z$. Given $\phi \in \ell^1$, denote $\phi^{(m)} \bydef \pi^{(m)} \phi \in \ell^1$ and $\phi^{(\infty)} \bydef \left( {\rm Id} - \pi^{(m)} \right) \phi \in \ell^1$.
Thus, $\phi$ is represented by $\phi=\phi^{(m)}+\phi^{(\infty)}$.
%
%\begin{align*}
%	\phi&=(\dots,0,\phi_{-m},\dots,\phi_{m},0\dots)+(\dots,\phi_{-m-1},0,\dots,0,\phi_{m+1},\dots)\\
%	&=\phi^{(m)}+\phi^{(\infty)}.
%\end{align*}
%
For example, with this notation, the discrete convolution in \eqref{eq:linearized_problem} can be expanded as
\begin{align*}
(\ba(t) * b(t))_{k} &=\left(\ba(t)*b^{(m)}(t)\right)_{k}+\left(\ba(t) * b^{(\infty)}(t)\right)_{k}
\\
&= \sum_{\substack{k_{1}+k_{2}=k\\\left|k_{1}\right| \leq N,~|k_{2}|\le m}} \ba_{k_{1}}(t) b_{k_{2}}(t)+\sum_{\substack{k_{1}+k_{2}=k\\\left|k_{1}\right| \leq N,~|k_{2}|>m}} \ba_{k_{1}}(t) b_{k_{2}}(t) \quad(k\in\Z ).
\end{align*}

As previously mentioned, given a step size $h>0$, our goal in this section is to show the existence of the evolution operator $U(t,s)$ of \eqref{eq:linearized_problem} by computing a constant $\bm{W_h}$ satisfying \eqref{eq:W_h_constant}. However, from a computational point of view, the formulation of the homogeneous non-autonomous linear IVP \eqref{eq:linearized_problem} is daunting due to the fact that for each fixed $k \in \Z$ the convolution term $\left(\ba\left(t\right)*b(t) \right)_k$ involves all Fourier modes $(b_j(t))_{j \in \Z}$.
In order to address this issue, we separate the equations by considering yet another homogeneous IVP  with respect to the sequence $c(t)=(c_k(t))_{k\in \Z }$
\begin{align}
\label{eq:linearizedeq_finite}
\frac{d}{dt}c_k(t)+\expig\left[k^2\omega^2c_k(t)-2\left(\ba\left(t\right)*c^{(m)}(t)\right)_k\right] & =0 \quad(|k|\le m)
\\[1mm]
\label{eq:linearizedeq_infinite}
\frac{d}{dt}c_k(t)+\expig\left[k^2\omega^2c_k(t)-2\left(\ba\left(t\right)*c^{(\infty)}(t)\right)_k\right] & =0\quad(|k|>m).
\end{align}
This new {\em decoupled} formulation, while not being equivalent to \eqref{eq:linearized_problem}, will be used to control the evolution operator associated to \eqref{eq:linearized_problem}. Denote by $C^{(m)}(t,s)$ and $C^{(\infty)}(t,s)$ the evolution operators of the $(2m+1)$-dimensional equation \eqref{eq:linearizedeq_finite} and the infinite dimensional equation \eqref{eq:linearizedeq_infinite}, respectively. At this point, the existence of the operator $C^{(\infty)}(t,s)$ is not guaranteed, but will be verified in Section~\ref{sec:evolution_operator}.
 We extend the action of the operator $C^{(m)}(t,s)$ (resp. $C^{(\infty)}(t,s)$) on $\ell^1$ by introducing the operator $\bar U^{(m)}(t,s)$ (resp. $\bar U^{(\infty)}(t,s)$) as follows. Given $\phi \in \ell^1$, define $\bar U^{(m)}(t,s):\ell^1 \to \ell^1$ and $\bar U^{(\infty)}(t,s):\ell^1 \to \ell^1$ by 
\begin{align}
\label{eq:bU_m_definition}
\left( \bar U^{(m)}(t,s) \phi \right)_k &= \begin{cases} \left( C^{(m)}(t,s) ( \phi_k)_{|k|\le m} \right)_k, & |k| \le m \\ 0, & |k| >m \end{cases}
\\
\label{eq:bU_infty_definition}
\left( \bar U^{(\infty)}(t,s) \phi \right)_k &= \begin{cases} 0 , & |k| \le m \\ \left( C^{(\infty)}(t,s)  ( \phi_k)_{|k| > m} \right)_k, & |k| > m.   \end{cases}
\end{align}
%
%Using that notation, we can define an operator $\bU(t,s):\ell^1 \to \ell^1$ by 
%%
%\[
%	\bU(t,s) \phi = \bU^{(m)}(t,s)\phi^{(m)} + \bU^{(\infty)}(t,s)\phi^{(\infty)}.
%\]

The proof of existence of the evolution operator $U(t,s)$ of the original linearized problem \eqref{eq:linearized_problem} and the explicit bound $\bm{W_h}$ satisfying \eqref{eq:W_h_constant} is presented in the following theorem. 
\begin{thm}\label{thm:sol_map}
	Let $s,t\in J=(0,h)$ satisfying $0\le s\le t\le h$ and let $\ba$ be fixed.
	Assume that there exists a constant $W_m>0$ such that
	\begin{equation} \label{eq:bound_W_m}
	\sup_{0\le s\le t\le h}\left\| \bU^{(m)}(t,s)\right\|_{B(\ell^1)}\le W_m.
	\end{equation}
	Assume that $C^{(\infty)}(t,s)$ exists and that $\bU^{(\infty)}(t,s)$ defined in \eqref{eq:bU_infty_definition} satisfies
	\begin{equation} \label{eq:assumption_existence_U_infty}
	\left\| \bU^{(\infty)}(t,s)\right\|_{B (\ell^1)}\le W^{(\infty)}(t,s)\bydef e^{-\mu_{m+1}(t-s)+2\int_{s}^{t}\|\ba(\tau)\|_{\ell^1}d\tau},
	\end{equation}
	where $\mu_{m+1} \bydef (m+1)^2\omega^2\cos\theta$.
	Define the constants $W_{\infty}\ge0$, $\bar{W}_{\infty}\ge0$,  $W_{\infty}^{\sup}>0$ by
	\begin{align}
	\label{eq:def_W_infty}
	W_\infty &\bydef \frac{e^{\left(2 \|\ba\|_{X}-\mu_{m+1}\right)h}-1}{2\|\ba\|_{X}-\mu_{m+1}}
	\\[1mm]
	\label{eq:def_barW_infty}
	\bar{W}_{\infty} & \bydef
	\frac{W_\infty-h }{2\|\ba\|_{X}-\mu_{m+1}}\\
	\label{eq:def_W_infty_sup}
	W_{\infty}^{\sup} & \bydef \begin{cases}
	1, & \mu_{m+1} \ge 2\|\ba\|_X\\
	e^{(\mu_{m+1} - 2\|\ba\|_X)h}, & \mu_{m+1}< 2\|\ba\|_X,
	\end{cases}
	\end{align}
	respectively.
%	,\\
%	\nonumber
%	&\sup _{0 \leq s \leq t \leq h} \int_{s}^{t} W^{(\infty)}(r, s) d r=\sup _{0 \leq s \leq t \leq h} \int_{s}^{t} W^{(\infty)}(t, r) d r\le W_\infty,\\
%	\nonumber
%	&\sup _{0 \leq s \leq t \leq h}\int_{s}^{t} \int_{s}^{r} W^{(\infty)}(r, \tau) d \tau d r\le \bar{W}_{\infty}\\
%	\nonumber
%	&\sup_{0 \leq s \leq t \leq h}W^{(\infty)}(t,s)\le W_{\infty}^{\sup},
%	\end{align}
%	respectively.
%	and let $W_{\infty}^{\sup}>0$ be any constant satisfying 
%	\[
%	\sup_{0 \leq s \leq t \leq h}W^{(\infty)}(t,s)\le W_{\infty}^{\sup}.
%	\]
	%
	Define $\ba^{(\bm{s})}(t) \in \ell^1$ component-wise by
	\[
	\ba_k^{(\bm{s})}(t) = \begin{cases} 0,  & |k|\le N \text{ and } k = 0 \\ \ba_k(t), & |k|\le N \text{ and } k \ne 0 \\ 0, & |k|>N. \end{cases}
	\]
	If 
	\begin{equation} \label{eq:kappa_condition}
	\kappa\bydef 1-4W_m\bar{W}_\infty\|\ba^{(\bm{s})} \|_X^2>0,
	\end{equation}
	then the evolution operator $U(t,s)$ exists and the following estimate holds 
	\[
	\sup_{0\le s\le t\le h}\|U(t,s)\phi\|_{\ell^1}\le \bm{W_h} \|\phi\|_{\ell^1},\quad\forall\phi\in\ell^1,
	\]
	where
	\begin{equation} \label{eq:definition_of_W_h}
	\bm{W_h} \bydef\left\|\begin{bmatrix}
	W_m\kappa^{-1} & 2W_mW_\infty\|\ba^{(\bm{s})} \|_X\kappa^{-1}\\
	2W_mW_\infty\|\ba^{(\bm{s})} \|_X\kappa^{-1} & W_{\infty}^{\sup}+4 W_mW_\infty^{2}\|\ba^{(\bm{s})} \|_{X}^{2}\kappa^{-1}
	\end{bmatrix}\right\|_1.
	\end{equation}
\end{thm}

After the proof of Theorem~\ref{thm:sol_map}, we introduce in Section~\ref{sec:fundamental_solution} a rigorous computational method based on Chebyshev series to 
obtain the finite dimensional evolution operator $C^{(m)}(t,s)$, which will directly yield $\bU^{(m)}(t,s)$. We show how this helps computing $W_m>0$ satisfying \eqref{eq:bound_W_m}. Section \ref{sec:evolution_operator} is devoted to show the existence of $C^{(\infty)}(t,s)$ (and hence the existence of $\bU^{(\infty)}(t,s)$) and to show that the hypothesis \eqref{eq:assumption_existence_U_infty} holds. The proof of Theorem~\ref{thm:sol_map} uses the following elementary result.

\begin{lem} \label{lem:ineqW_infty}
Consider the constants $W_{\infty}\ge0$, $\bar{W}_{\infty}\ge0$ and  $W_{\infty}^{\sup}>0$ as defined in \eqref{eq:def_W_infty}, \eqref{eq:def_barW_infty} and \eqref{eq:def_W_infty_sup}, respectively. %Then, the following inequalities hold
Then $W^{(\infty)}$, defined in \eqref{eq:assumption_existence_U_infty}, obeys the following inequalities:

\begin{align}
\label{eq:ineqW_infty_sup}
\sup_{0 \leq s \leq t \leq h}	W^{(\infty)}(t,s)  & \le W_{\infty}^{\sup}
\\
\label{eq:ineqW_infty}
\sup _{0 \leq s \leq t \leq h} \int_{s}^{t} W^{(\infty)}(\tau, s) d \tau ~,~\sup _{0 \leq s \leq t \leq h} \int_{s}^{t} W^{(\infty)}(t, \tau) d \tau &\le W_\infty
\\
\label{eq:ineqbarW_infty}
\sup _{0 \leq s \leq t \leq h}\int_{s}^{t} \int_{s}^{\tau} W^{(\infty)}(\tau, \sigma) d \sigma d \tau & \le  \bar{W}_{\infty}.
\end{align}
\end{lem}

\begin{proof}
First, note that from \eqref{eq:assumption_existence_U_infty}
\begin{align*}
	\sup_{0 \leq s \leq t \leq h}	W^{(\infty)}(t,s) 
	&= \sup_{0 \leq s \leq t \leq h}e^{-\mu_{m+1}(t-s)+2\int_{s}^{t}\|\ba(\tau)\|_{\ell^1}d \tau}\\
	& \le \sup_{0 \leq s \leq t \leq h}e^{-(\mu_{m+1}-2\|\ba\|_{X})(t-s)}\\
	& \le \begin{cases}
	1, & \mu_{m+1} \ge 2\|\ba\|_X\\
	e^{(2\|\ba\|_X - \mu_{m+1})h}, & \mu_{m+1}< 2\|\ba\|_X,
	\end{cases}\\
	& = W_{\infty}^{\sup}.
\end{align*}
Second, note that
\begin{align*}
\sup _{0 \leq s \leq t \leq h} \int_{s}^{t} W^{(\infty)}(\tau, s) d \tau
&=\sup _{0 \leq s \leq t \leq h}\left(\int_{s}^{t} e^{-\mu_{m+1}(\tau-s)+2\int_{s}^{\tau}\|\ba(\sigma)\|_{\ell^1}d\sigma}dr\right)\\
&\le\sup _{0 \leq s \leq t \leq h}\left(\int_{s}^{t} e^{(2\|\ba\|_X-\mu_{m+1})(\tau-s)}dr\right)\\
&=\sup _{0 \leq s \leq t \leq h}\left(\frac{e^{\left(2 \|\ba\|_{X}-\mu_{m+1}\right)(t-s)}-1}{2\|\ba\|_{X}-\mu_{m+1}}\right)\\
&\le \frac{e^{\left(2 \|\ba\|_{X}-\mu_{m+1}\right)h}-1}{2\|\ba\|_{X}-\mu_{m+1}} = W_\infty
\end{align*}
and that
\begin{align*}
\sup _{0 \leq s \leq t \leq h} \int_{s}^{t} W^{(\infty)}(t, \tau) d \tau
&=\sup _{0 \leq s \leq t \leq h}\left(\int_{s}^{t} e^{-\mu_{m+1}(t-\tau)+2\int_{\tau}^{t}\|\ba(\sigma)\|_{\ell^1}d\sigma}dr\right)\\
&=\sup _{0 \leq s \leq t \leq h}\left(\int_{s}^{t} e^{(2\|\ba\|_X-\mu_{m+1})(t-\tau)}dr\right)\\
&=\sup _{0 \leq s \leq t \leq h}\left(\frac{e^{\left(2 \|\ba\|_{X}-\mu_{m+1}\right)(t-s)}-1}{2\|\ba\|_{X}-\mu_{m+1}}\right)\\
&\le \frac{e^{\left(2 \|\ba\|_{X}-\mu_{m+1}\right)h}-1}{2\|\ba\|_{X}-\mu_{m+1}}=W_\infty.
\end{align*}
Third, 
\begin{align*}
\int_{s}^{t} \int_{s}^{\tau} W^{(\infty)}(\tau, \sigma) d \sigma d \tau
&=\int_{s}^{t} \frac{e^{\left(2 \|\ba\|_{X}-\mu_{m+1}\right)(\tau-s)}-1}{2\|\ba\|_{X}-\mu_{m+1}}d \tau\\
&=\left(\frac{1}{2\|\ba\|_{X}-\mu_{m+1}}\right)\left[\frac{e^{\left(2 \|\ba\|_{X}-\mu_{m+1}\right)(\tau-s)}}{2\|\ba\|_{X}-\mu_{m+1}}-\tau\right]_{\tau=s}^{\tau=t}\\
&=\left(\frac{1}{2\|\ba\|_{X}-\mu_{m+1}}\right)\left[\frac{e^{\left(2 \|\ba\|_{X}-\mu_{m+1}\right)(t-s)}-1}{2\|\ba\|_{X}-\mu_{m+1}}-(t-s)\right],
\end{align*}
and hence, it follows that
\[
\sup _{0 \leq s \leq t \leq h}\int_{s}^{t} \int_{s}^{\tau} W^{(\infty)}(\tau, \sigma) d \sigma d \tau\le \left(\frac{1}{2\|\ba\|_{X}-\mu_{m+1}}\right)\left(\frac{e^{\left(2 \|\ba\|_{X}-\mu_{m+1}\right)h}-1}{2\|\ba\|_{X}-\mu_{m+1}}-h\right) = \bar{W}_{\infty}. \qedhere
\]

\end{proof}

\begin{proof}[Proof of Theorem~\ref{thm:sol_map}]
%The aim is to obtain that the evolution operator, say $\left\{U(t,s)\right\}_{0\le s\le t\le h}$, exists under a certain condition.
%To prove existence of the evolution operator, we get a positive constant $W$ satisfying
%\[
%\|b\|_{X}=\sup_{0\le s\le t\le h}\|U(t,s)\phi\|_{\ell^1}\le W\|\phi\|_{\ell^1},\quad\forall\phi\in\ell^1.
%\]
First note that for $|k|\le m$, the system of differential equations \eqref{eq:linearized_problem} is described by the non-homogeneous equation
\begin{equation} \label{eq:non-homogeneous-finite-system}
\frac{d}{dt}b_k(t)+\expig\left[k^2\omega^2b_k(t)-2\left(\ba\left(t\right)*b^{(m)}(t) \right)_k\right]=2\expig\left(\ba\left(t\right)*b^{(\infty)} \right)_k\quad(|k|\le m)
\end{equation}
with the initial condition $b_k(s)=\phi_k$ for $|k| \le m$.
Consider the homogeneous equation \eqref{eq:linearizedeq_finite} and denote by $C^{(m)}(t,s)$ the solution of the variational problem.
We show in Section~\ref{sec:fundamental_solution} how to rigorously compute $C^{(m)}(t,s)$.
As before, let $\bU^{(m)}(t,s)$ be the extension of the action of the operator $C^{(m)}(t,s)$ on $\ell^1$.
Using the evolution operator $\bU^{(m)}(t,s):\ell^1 \to \ell^1$, we can integrate the system \eqref{eq:non-homogeneous-finite-system} to obtain 
the integral equation
\begin{equation} \label{eq:b^{m}_integral_equation}
b^{(m)}(t)=\bU^{(m)}(t, s) \phi^{(m)} + 2 e^{\im \theta} \int_{s}^{t} \bU^{(m)}(t, \tau)  \pi^{(m)}\left(\ba(\tau) * b^{(\infty)}(\tau)\right) d \tau.
\end{equation}
Here, for $|k|\le m$,
\[
\left(\ba * b^{(\infty)}\right)_{k} = \sum_{\substack{k_{1}+k_{2}=k\\\left|k_{1}\right| \leq N, |k_{2}|>m}} \ba_{k_{1}} b_{k_{2}} = \sum_{\substack{k_{1} \neq 0\\\left|k_{1}\right| \leq N,~|k-k_1|>m}} \ba_{k_{1}} b_{k-k_{1}}= \sum_{\left|k_{1}\right| \leq N,~|k-k_1|>m} \ba^{(\bm{s})}_{k_{1}} b_{k-k_{1}}
\]
holds. Combining \eqref{eq:bound_W_m} and \eqref{eq:b^{m}_integral_equation}, and using the property \eqref{eq:banach_algebra}, it follows that
\begin{equation}\label{eq:zero-mode}
\left\|b^{(m)}(t)\right\|_{\ell^1} \leq W_m\left\|\phi^{(m)}\right\|_{\ell^1}+2 W_m\int_{s}^{t}\left\| \ba^{(\bm{s})}(\tau)\right\|_{\ell^1}\left\|b^{(\infty)}(\tau)\right\|_{\ell^1} d \tau.
\end{equation}

Next, for the case of $|k|>m$, we rewrite the system of differential equations \eqref{eq:linearized_problem} as
\begin{equation} \label{eq:non-homogeneous-infinite-system}
	\frac{d}{dt}b_k(t)+\expig\left[k^2\omega^2b_k(t)-2\left(\ba\left(t\right)*b^{(\infty)} \right)_k\right]= 2\expig\left(\ba\left(t\right)*b^{(m)} \right)_k\quad(|k|>m)
\end{equation}
with the initial condition $b_k(s)=\phi_k$ for $|k|>m$. Note that by assumption the evolution operator $C^{(\infty)}(t,s)$ associated to \eqref{eq:linearizedeq_infinite} exists. Define $\bU^{(\infty)}(t,s)$ as in \eqref{eq:bU_infty_definition}. 
%Let us consider the following homogeneous IVP:
%%
%\begin{equation}\label{eq:homo_tail}
%		\frac{d}{dt}c_k(t)+\expig\left[k^2\omega^2c_k(t)-2\left(\ba\left(t\right)*c^{(\infty)} \right)_k\right]=0,\quad c_k(s)=\phi_k\quad(|k|>m),
%\end{equation}
%where $c^{(\infty)}=\left(c_k\right)_{|k|>m}$.
%More explicitly,
%There exists the solution $c^{(\infty)}(t)$ of \eqref{eq:homo_tail} satisfying
%\[
%	\left\|c^{(\infty)}(t)\right\|_{\ell^1}\le e^{-\mu_{m+1}(t-s)+2\int_{s}^{t}\|\ba(r)\|_{\ell^1}dr}\|\phi^{(\infty)}\|_{\ell^1}.
%\]
%where $\mu_{m+1}=(m+1)^2\omega^2\cos\theta$.
%We will give a proof of this fact by Theorem \ref{thm:ev_op} in Section \ref{sec:evolution_operator}.
Using that operator, the system \eqref{eq:non-homogeneous-infinite-system} can be re-written as
\begin{equation} \label{eq:b^{infty}_integral_equation}
b^{(\infty)}(t)=\bU^{(\infty)}(t, s) \phi^{(\infty)} + 2 \expig \int_{s}^{t} \bU^{(\infty)}(t, \tau) ({\rm Id} -\pi^{(m)})\left(\ba(\tau)*b^{(m)}(\tau)\right) d \tau.
\end{equation}
For $|k|>m$, the following holds
\[
\left(\ba * b^{(m)}\right)_{k} = \sum_{\substack{k_{1}+k_{2}=k\\\left|k_{1}\right| \leq N, |k_{2}|\le m}} \ba_{k_{1}} b_{k_{2}} = \sum_{\substack{k_{1} \neq 0\\\left|k_{1}\right| \leq N,~|k-k_1|\le m}} \ba_{k_{1}} b_{k-k_{1}}= \sum_{\left|k_{1}\right| \leq N,~|k-k_1|\le m} \ba^{(\bm{s})}_{k_{1}} b_{k-k_{1}}.
\]
Combining \eqref{eq:assumption_existence_U_infty}, \eqref{eq:b^{infty}_integral_equation} and using the property \eqref{eq:banach_algebra}, it follows that
\begin{equation}\label{eq:tail-mode}
\left\|b^{(\infty)}(t)\right\|_{\ell^1}\le W^{(\infty)}(t,s)\left\|\phi^{(\infty)}\right\|_{\ell^1}+2\int_{s}^{t} W^{(\infty)}(t, \tau)\left\|\ba^{(\bm{s})}(\tau)\right\|_{\ell^1}\left\|b^{(m)}(\tau)\right\|_{\ell^1} d \tau.
\end{equation}
Plugging \eqref{eq:tail-mode} into \eqref{eq:zero-mode}, and using the inequalities \eqref{eq:ineqW_infty} and \eqref{eq:ineqbarW_infty} from Lemma~\ref{lem:ineqW_infty},
we have
\begin{align}\nonumber
\left\|b^{(m)}(t)\right\|_{\ell^1}
&\le W_m\left\|\phi^{(m)}\right\|_{\ell^1}+2 W_m\int_{s}^{t}\|\ba^{(\bm{s})}(\tau)\|_{\ell^1}\Bigg\{
W^{(\infty)}(\tau, s)\left\|\phi^{(\infty)}\right\|_{\ell^1}\\\nonumber
&\hphantom{=}\quad +2 \int_{s}^{\tau} W^{(\infty)}(\tau, \sigma)\|\ba^{(\bm{s})}(\sigma)\|_{\ell^1}\|b^{(m)}(\sigma)\|_{\ell^1}d\sigma
\Bigg\}d \tau\\\nonumber
&=W_m\left\|\phi^{(m)}\right\|_{\ell^1}+\left(2 W_m\int_{s}^{t}\|\ba^{(\bm{s})}(\tau)\|_{\ell^1} W^{(\infty)}(\tau, s) d \tau\right)\left\|\phi^{(\infty)}\right\|_{\ell^1}\\\nonumber
&\hphantom{=}\quad +4 W_m\int_{s}^{t}\|\ba^{(\bm{s})}(\tau)\|_{\ell^{1}}\left(\int_{s}^{\tau}\|\ba^{(\bm{s})}(\sigma)\|_{\ell^1} W^{(\infty)}(\tau, \sigma)\left\|b^{(m)}(\sigma)\right\|_{\ell^1} d \sigma\right)d \tau\\
&\le W_m\left\|\phi^{(m)}\right\|_{\ell^1}+2 W_mW_\infty\|\ba^{(\bm{s})}\|_{X}\left\|\phi^{(\infty)}\right\|_{\ell^1}
+4 W_m\bar{W}_{\infty}\|\ba^{(\bm{s})}\|_{X}^{2}
%\left(\int_{s}^{t} \int_{s}^{\tau} W^{(\infty)}(\tau, \tau) d \tau d \tau\right)
\left\|b^{(m)}\right\|_{X}.
\label{eq:zero-mode_estimate}
\end{align}

By assumption \eqref{eq:kappa_condition}, $\kappa = 1-4W_m\bar{W}_\infty\|\ba^{(\bm{s})} \|_X^2>0$ and using \eqref{eq:zero-mode_estimate} yields that
\begin{equation}\label{eq:b_m_norm}
	\|b^{(m)}\|_X\le \frac{W_m\left\|\phi^{(m)}\right\|_{\ell^1}+2 W_mW_\infty\|\ba^{(\bm{s})} \|_{X}\left\|\phi^{(\infty)}\right\|_{\ell^1}}{\kappa},
\end{equation}
which guarantees the existence of the solution of the finite part of \eqref{eq:linearized_problem}. 
Using the inequalities \eqref{eq:ineqW_infty_sup} and \eqref{eq:ineqW_infty} in Lemma~\ref{lem:ineqW_infty} and \eqref{eq:b_m_norm}, the tail mode \eqref{eq:tail-mode} is bounded by
\begin{align}\nonumber
	\left\|b^{(\infty)}\right\|_{X} &\le W_{\infty}^{\sup}\left\|\phi^{(\infty)}\right\|_{\ell^1}+2W_\infty\left\|\ba^{(\bm{s})} \right\|_X\left\|b^{(m)}\right\|_X\\
	&\le 2W_mW_\infty\|\ba^{(\bm{s})} \|_X\kappa^{-1}\left\|\phi^{(m)}\right\|_{\ell^1}+\left(W_{\infty}^{\sup}+4 W_mW_\infty^{2}\|\ba^{(\bm{s})} \|_{X}^{2}\kappa^{-1}\right)\left\|\phi^{(\infty)}\right\|_{\ell^{1}}.\label{eq:b_inf_norm}
\end{align}
%where $W_{\infty}^{\sup}$ denotes a constant satisfying $\sup_{0 \leq s \leq t \leq h}W^{(\infty)}(t,s)\le W_{\infty}^{\sup}$.

Finally, \eqref{eq:b_m_norm} and \eqref{eq:b_inf_norm} yield
\begin{align*}
\|b\|_X
&\le\left\|b^{(m)}\right\|_{X}+\left\|b^{(\infty)}\right\|_{X}\\
&=\left\|\begin{bmatrix}
\left\|b^{(m)}\right\|_{X}\\[2mm]\left\|b^{(\infty)}\right\|_{X}
\end{bmatrix}\right\|_1\\
&\le\left\|\begin{bmatrix}
W_m\kappa^{-1} & 2W_mW_\infty\|\ba^{(\bm{s})} \|_X\kappa^{-1}\\
2W_mW_\infty\|\ba^{(\bm{s})} \|_X\kappa^{-1} & W_{\infty}^{\sup}+4 W_mW_\infty^{2}\|\ba^{(\bm{s})} \|_{X}^{2}\kappa^{-1}
\end{bmatrix}\begin{bmatrix}
\|\phi^{(m)}\|_{\ell^1}\\\|\phi^{(\infty)}\|_{\ell^1}
\end{bmatrix}\right\|_1\\
&\le \bm{W_h} \|\phi\|_{\ell^1}. \qedhere
\end{align*}
%where $W$ is defined by the following 1-norm of matrix:
%%$\|\cdot\|_1$ denotes the usual 1-norm of vectors and 
%\[
%W\bydef\left\|\begin{bmatrix}
%W_m\kappa^{-1} & 2W_mW_\infty\|\ba^{(\bm{s})} \|_X\kappa^{-1}\\
%2W_mW_\infty\|\ba^{(\bm{s})} \|_X\kappa^{-1} & W_{\infty}^{\sup}+4 W_mW_\infty^{2}\|\ba^{(\bm{s})} \|_{X}^{2}\kappa^{-1}
%\end{bmatrix}\right\|_1. \qedhere
%\]
\end{proof}
\begin{rem}\label{rem:bootstrap}
	For arbitrary forcing term $g\in X$,  $b=\mathscr{A}g$ defined in \eqref{eq:definition_of_scriptA} solves \eqref{eq:linearized_problem_non_homogeneous} in the sense of classical solution (see, e.g., \cite{pazy1983semigroups}). This implies $b\in C^1(J;D(L))$. Then, the solution map operator $\mathscr{A}$ has a bootstrap property from $X$ to $C^1(J;D(L))$.
\end{rem}

Having derived a computable formulation in \eqref{eq:definition_of_W_h} for the constant $\bm{W_h}>0$ which provides a uniform bound for the norm of evolution operator $U(\cdot,\cdot)$ over $\mathcal{S}_h$, we now turn to the question of computing the bound $W_m$, which is required to defined $\bm{W_h}$.

%%%%%%%%%%%%%%%%
\subsection{Computation of the bound \boldmath$W_m$\unboldmath}\label{sec:fundamental_solution}
%%%%%%%%%%%%%%%%

The goal of this section is to present a rigorous computational approach to obtain the constant $W_m$ satisfying 
\eqref{eq:bound_W_m}, that is a uniform bound for the operator norm of $\bU^{(m)}$ over the simplex $\mathcal{S}_h$. From \eqref{eq:bU_m_definition}, defining $\bU^{(m)}(t,s)$ boils down to the computation of $C^{(m)}(t,s)$, that is the evolution operators of the $(2m+1)$-dimensional equation \eqref{eq:linearizedeq_finite}, which can be written component-wise as 
\begin{equation}\label{eq:variational_prob}
\frac{d}{d t} c_{k,j}(t) = \expig\left[-k^2\omega^2c_{k,j}(t)+2\left(\ba\left(t\right)*c^{(m)}_{j}(t) \right)_k\right], \quad c_{k,j}(s)=\delta_{k,j}\quad(|k|,|j|\le m), 
\end{equation}
for $0\le s\le t\le h$, and where $\delta_{k,j}$ denotes the Kronecker's delta.
Denote by $A(t)$ the matrix representation of the right-hand side of \eqref{eq:variational_prob} acting on the coefficients $\left(c_{k,j}(t,s)\right)_{|k|,|j|\le m}$.

The evolution matrix $C^{(m)}(t,s)$ can be decomposed as 
\begin{equation} \label{eq:Cm_decomposition}
C^{(m)}(t,s)=\Phi(t)\Psi(s) \bydef \Phi(t)\Phi(s)^{-1},
\end{equation}
where $\Phi(t) \in M_{2m+1}(\C )$ is the principal fundamental solution and solves 
\begin{equation} \label{eq:finite_dim_variational_problem}
\frac{d}{d t} \Phi(t)=A(t){\Phi}(t),\quad\Phi(0)={\rm Id}, \quad t \in [0,h],
\end{equation}
and where $\Psi(s)=\Phi(s)^{-1}$ is the solution of the adjoint problem
\begin{equation} \label{eq:finite_dim_adjoint_variational_problem}
\frac{d}{d s} \Psi(s)=-{\Psi}(s)A(s),\quad\Psi(0)={\rm Id}, \quad s \in [0,h].
\end{equation}

Having introduce the set-up, the computation of $W_m$ is twofold. First, using the approach of \cite{MR3148084}, we use the tools of rigorous numerics and Chebyshev series expansion to compute the fundamental matrix solutions $\Phi(t)$ and $\Psi(t)$ satisfying \eqref{eq:finite_dim_variational_problem} and \eqref{eq:finite_dim_adjoint_variational_problem}, respectively. Second, using interval arithmetic and the Chebyshev series representations of $\Phi(t)$ and $\Psi(t)$, we obtain $W_m$ such that 
\begin{equation} \label{eq:evaluating_W_m}
\sup _{t\in [0,h]} \|\Phi(t)\|_1 \cdot \sup _{s\in [0,h]} \|\Psi(s)\|_1\le W_m,
\end{equation}
where $\| \cdot \|_1$ represents the matrix norm induced by the $\ell^1$ norm on $\C^{2m+1}$. By construction, the constant $W_m$ obtained computationally in \eqref{eq:evaluating_W_m} satisfies \eqref{eq:bound_W_m}.
The remainder of this section is devoted to the computational method for obtaining the matrix solutions $\Phi(t)$ and $\Psi(t)$. 

\subsubsection{Rigorously computing \boldmath$\Phi(t)$\unboldmath~and~\boldmath$\Psi(t)$\unboldmath~via Chebyshev series}

The Chebyshev polynomials $T_k:[-1,1] \to \R$ ($k \ge 0$) are orthogonal polynomials defined by $T_0=1$, $T_1(t)=t$ and $T_{k+1}(t)=2tT_k(t) - T_{k-1}(t)$ for $k \ge 1$.  
Every Lipschitz continuous function $v:[-1,1]\to \R$ has a unique representation as an absolutely and uniformly convergent series $v(t) = \sum_{k=0}^{\infty}a_{k}T_{k}(t)$ (e.g. see \cite{MR3012510}). The more regular the function $v$ is, the faster the decay rate of its Chebyshev coefficients. 

Fixing a Fourier projection number $N$ and a Chebyshev projection number $n$, let us assume that we have numerically computed an approximate solution of the initial-boundary value problem \eqref{eq:CGL} in the form
\begin{equation}\label{eq:app_sol}
\bar u(t,x) = \sum_{|k| \le N} \ba_k(t) e^{\im k\omega x}, 
\qquad
\ba_k(t) = \ba_{0,k} + 2 \sum_{\ell=1}^{n-1} \ba_{\ell,k} T_{\ell}(t), 
\end{equation}
with $\omega=2\pi$. In practice, we perform this task by considering a Galerkin approximation of \eqref{eq:CGL_ode} of size $2N+1$ and using MATLAB software system Chebfun (e.g. see \cite{MR2767023}) to compute a solution of the IVP on the time interval $[0,h]$. Given this approximate solution $\ba(t) = \left( \ba_k(t) \right)_{k=-N}^N$, the variational problem \eqref{eq:finite_dim_variational_problem} consists of solving the homogeneous initial value problem (IVP) with respect to sequence $\left( c_k(t) \right)_{k=-m}^m$ satisfying (after rescaling the time interval $[0,h]$ to $[-1,1]$)
\begin{equation} \label{eq:rescale_ODEs}
\dot c_k = G_k(c) \bydef - \frac{h \expig}{2} \left[k^2\omega^2c_k-2\left(\ba*c^{(m)}\right)_k\right], \quad |k|\le m.
\end{equation}
Assume that the initial condition of \eqref{eq:rescale_ODEs} is given by $c_k(-1) = b_k$.
In practice, we will solve rigorously $2m+1$ IVPs with initial conditions $(b_k)_{k=-m}^m = {\rm e}_j \in \mathbb C^{2m+1}$ the canonical basis vectors. 
Rewriting the system \eqref{eq:rescale_ODEs} as an integral equation results in
\begin{equation} \label{eq:integral_equation_Fisher}
c_k(t) = b_k +  \int_{-1}^t G_k(c(s))~ds, \qquad |k| \le m, \quad t \in [-1,1].
\end{equation}
For each $k$, we expand $c_k(t)$ using a Chebyshev series, that is 
\begin{equation} \label{eq:a_k_Chebyshev_expansion}
c_k(t) = c_{0,k} + 2 \sum_{\ell \ge 1} c_{\ell,k} T_\ell(t) = c_{0,k} + 2 \sum_{\ell \ge 1} c_{\ell,k} \cos(\ell \theta) 
= \sum_{\ell \in \Z} c_{\ell,k} e^{\im \ell \theta},
\end{equation}
where $c_{-\ell,k} = c_{\ell,k}$ and $t = \cos(\theta)$. For each $|k| \le m$, we expand $G_k(c(t))$ using a Chebyshev series, that is 
\begin{equation} \label{eq:f_k(a)_Chebyshev_expansion}
G_k(c(t)) = \psi_{0,k}(c)  + 2 \sum_{\ell \ge 1} \psi_{\ell,k}(c) \cos(\ell \theta) 
=  \sum_{\ell \in \Z} \psi_{\ell,k}(c) e^{\im \ell \theta}.
\end{equation}
where
\[
\psi_{\ell,k}(c) = \lambda_k c_{\ell,k} + N_{\ell,k}(c).
\]
Letting $N_k(c) \bydef (N_{\ell,k}(c))_{\ell \ge 0}$, $\psi_k(c) \bydef (\psi_{\ell,k}(c))_{\ell \ge 0}$ and noting that $(\lambda_k c_k)_\ell = \lambda_k c_{\ell,k}$, we get that
\begin{equation} \label{eq:phi_expansion_general}
\psi_k(c) = \lambda_k c_k + N_k(c).
\end{equation}

\begin{rem} 
	Note that for problem \eqref{eq:rescale_ODEs},
	\[
	\lambda_k = - \frac{h \expig}{2} \omega^2 k^2 
	\]
	and that 
	\begin{equation} \label{eq:explicit_Nk}
	N_{\ell,k}(c) = h \expig \left(\ba*c^{(m)}\right)_{\ell,k} = h \expig \sum_{{\ell_1+\ell_2 = \ell \atop k_1 + k_2 = k} \atop |\ell_2| <n, |k_2| \le m} \ba_{\ell_1,k_1} c_{\ell_2,k_2}.
	\end{equation}
\end{rem}

Combining expansions \eqref{eq:a_k_Chebyshev_expansion} and \eqref{eq:f_k(a)_Chebyshev_expansion} leads to %
\[
\sum_{\ell \in \Z} c_{\ell,k} e^{\im \ell \theta} = c_k(t) = b_k+\int_{-1}^t G_k(c(s))~ds=  b_k+\int_{-1}^t \sum_{\ell \in \Z} \psi_{\ell,k}(c) e^{\im \ell \theta}~ds 
\]
and this results (e.g. see in \cite{MR3148084}) in solving $f=0$, where 
$f = \left(f_{\ell,k} \right)_{\ell,k}$ is given component-wise by
\[
f_{\ell,k}(c) = 
\begin{cases}
\displaystyle
c_{0,k} + 2 \sum_{j=1}^\infty (-1)^j c_{j,k} - b_k, & \ell=0, |k| \le m \\
\displaystyle
2\ell c_{\ell,k} + ( \psi_{\ell+1,k}(c) - \psi_{\ell-1,k}(c)), & \ell>0, |k| \le m.
\end{cases}
\]
Hence, for $\ell>0$ and $|k| \le m$, we aim at solving
\[
f_{\ell,k}(c)  = 
2\ell c_{\ell,k} +  \lambda_k ( c_{\ell+1,k} - c_{\ell-1,k})
+ ( N_{\ell+1,k}(c) - N_{\ell-1,k}(c)) = 0.
\]
Finally, the problem that we solve is $f=0$, where $f = \left(f_{\ell,k} \right)_{\ell,k}$ is given component-wise by
\[
f_{\ell,k}(c) \bydef 
\begin{cases}
\displaystyle
c_{0,k} + 2 \sum_{j=1}^\infty (-1)^j c_{j,k} - b_k, & \ell=0, |k| \le m \\
\displaystyle
- \lambda_k c_{\ell-1,k} + 2\ell c_{\ell,k} + \lambda_k c_{\ell+1,k}
+ (N_{\ell+1,k}(c) - N_{\ell-1,k}(c)), & \ell > 0 , |k| \le m.
\end{cases}
\]
Define the operators (acting on Chebyshev sequences) by
\begin{equation} \label{eq:tridiagonal_T}
\cT \bydef
\begin{pmatrix} 
0&0&0&0&0&\cdots\\
-1&0&1&0&\cdots&\ \\
0&-1&0&1&0&\cdots\\
\ &\ddots&\ddots&\ddots&\ddots&\ddots\\
\ &\dots&0&-1&0&1 \\
\ &\ &\dots&\ddots&\ddots&\ddots
\end{pmatrix},
\end{equation}
and
\begin{equation} \label{eq:Lambda}
\Lambda  \bydef
\begin{pmatrix} 
0&0&0&0&0&\cdots\\
0&2&0&0&\cdots&\ \\
0&0&4&0&0&\cdots\\
\ &\ddots&\ddots&\ddots&\ddots&\ddots\\
\ &\dots&0&0&2\ell&0 \\
\ &\ &\dots&\ddots&\ddots&\ddots
\end{pmatrix}.
\end{equation}
%
%Similarly, define the operators $\mathbf{T}$ and $\mathbf{\Lambda}$ (acting on a Chebyshev-Fourier sequence $c=(c_{\ell,k})_{\ell \ge 0 \atop k=-N,\dots,N}$) by
%%
%\[
%\mathbf{T} c = \left( \cT c_{-N},\dots,\cT c_{-1},\cT c_{0},\cT c_{1},\dots,\cT c_{N} \right)
%\]
%%
%and
%%
%\[
%\mathbf{\Lambda} c = \left( \Lambda c_{-N},\dots,\Lambda c_{-1},\Lambda c_{0},\Lambda c_{1},\dots,\Lambda c_{N} \right)
%\]
%%
%where $c_k \bydef (c_{\ell,k})_{\ell \ge 0}$.
Using the operators $\cT$ and $\Lambda$, we may write more densely for the cases $\ell>0$ and $|k| \le m$
\[
f_{k}(c)  = \Lambda c_{k} + \cT( \lambda_k  c_{k} + N_k(c) ).
\]
Hence,
\begin{equation} \label{eq:f_{ell,k}}
f_{\ell,k}(c) =
\begin{cases}
\displaystyle
c_{0,k} + 2 \sum_{j=1}^\infty (-1)^j c_{j,k} - b_k, & \ell=0, |k| \le m \\
\displaystyle
\left( \Lambda c_{k} + \cT( \lambda_k c_{k} + N_k(c) ) \right)_\ell, & \ell > 0 , |k| \le m.
\end{cases}
\end{equation}
Denoting the set of indices $\cI = \{ (\ell,k) \in \Z^2 : \ell \ge 0 \text{ and } |k| \le m \}$ and denote $f=(f_{\bfj})_{\bfj \in \cI}$. Assume that using Newton's method, we computed $\bar c = (\bar c_{\ell,k})_{\ell=0,\dots,n-1 \atop k=-m,\dots,m}$ such that $f(\bar c) \approx 0$. Fix $\nu \ge 1$ the {\em Chebyshev decay rate} and define the weights $\omega_{\ell,k} = \nu^\ell$. Given a sequence 
$c = (c_{\bfj})_{\bfj \in \cI} = (c_{\ell,k})_{k=-m,\dots,m \atop \ell \ge 0}$, denote the Chebyshev-weighed $\ell^1$ norm by
%%
%\[
%\| c \|_{1,\nu} \bydef \sum_{|k| \le m \atop \ell \ge 0} |c_{\ell,k}| \nu^{\ell}.
%\]
%%
%Denoting the set of indices $\cI = \{ (\ell,k) \in \Z^2 : \ell \ge 0 \text{ and } |k| \le m \}$, and the weights $\omega_{\ell,k} = \nu^\ell$, we re-write the norm as
%
\[
\| c \|_{\cX^{m}_\nu} \bydef \sum_{\bfj \in \cI} |c_{\bfj}| \omega_{\bfj}.
\]
The Banach space in which we prove the existence of the solutions of $f=0$ is given by
\[
\cX^{m}_\nu \bydef
\left\{ c = (c_{\bfj})_{\bfj \in \cI} : 
\| c \|_{\cX^{m}_\nu}  < \infty
\right\}.
\]
The following result is useful to perform the nonlinear analysis when solving $f=0$ in $\cX^{m}_\nu$. We omit the elementary proof which can be mimicked from the one of Lemma~3 in \cite{MR3353132}.
\begin{lem} \label{lem:banach_algebra}
	For all $a,b \in \cX^{m}_\nu$, $\| a * b\|_{\cX^{m}_\nu} \le 4 \| a\|_{\cX^{m}_\nu} \| b\|_{\cX^{m}_\nu}$.
\end{lem}

The computer-assisted proof of existence of a solution of $f=0$ relies on showing that a certain Newton-like operator $c \mapsto c-Af(c)$ has a unique fixed point in the closed ball $B_r(\bar c) \subset \cX^{m}_\nu$, where $r$ is a radius to be determined. Let us now define the operator $A$. Given $n$, a finite number of Chebyshev coefficients used for the computation of $c$, denote by $f^{(n,m)}:\mathbb C^{n(2m+1)} \to \mathbb C^{n(2m+1)}$ the finite dimensional projection used to compute $\bar c \in \mathbb C^{n(2m+1)}$, that is 
\[
f^{(n,m)}_{\ell,k}(c^{(n,m)}) \bydef 
\begin{cases}
\displaystyle
c_{0,k} + 2 \sum_{j=1}^{n-1} (-1)^j c_{j,k} - b_k, & \ell=0, |k| \le m \\
\displaystyle
\left( \Lambda c_{k} + T( \lambda_k c_{k} + N_k(c^{(n,m)}) ) \right)_\ell, & 0<\ell<n , |k| \le m.
\end{cases}
\]

First consider $A^\dagger$ an approximation for the Fr\'echet derivative $Df(\bar c)$: 
\[
(A^\dagger c)_{\ell,k} = 
\begin{cases}
\left( Df^{(n,m)}(\bar c) c^{(n,m)} \right)_{\ell,k}, & 0 \le \ell < n, |k| \le m 
\\
2\ell c_{\ell,k}, & \ell \ge n, |k| \le m,
\end{cases}
\]
where $Df^{(n,m)}(\bar c) \in M_{n(2m+1)}(\mathbb C)$ denotes the Jacobian matrix. Consider now a numerical inverse $A^{(n,m)}$ of $Df^{(n,m)}(\bar c)$. We define the action of $A$ on a vector $c \in \cX^{m}_\nu$ as 
\[
(Ac)_{\ell,k} = 
\begin{cases}
\left( A^{(n,m)} c^{(n,m)} \right)_{\ell,k}, & 0 \le \ell < n, |k| \le m
\\
\frac{1}{2\ell} c_{\ell,k}, & \ell \ge n, |k| \le m.
\end{cases}
\]

The following Newton-Kantorovich type theorem (for linear problems posed on Banach spaces) is useful to show 
the existence of zeros of $f$.
\begin{thm} \label{thm:radii_polynomial}
Assume that there are constants $Y_0, Z_0, Z_1 \ge 0$ having that 
\begin{align}
\label{eq:Y0}
\| A f(\bar c) \|_{\cX^{m}_\nu} &\le Y_0,
\\
\label{eq:Z0}
\| {\rm Id} - A A^\dagger \|_{B(\cX^{m}_\nu)} &\le Z_0,
\\
\label{eq:Z1}
\| A (Df(\bar c) - A^\dagger) \|_{B(\cX^{m}_\nu)} &\le Z_1.
\end{align}
%
%If there exists an $r_0 > 0$ such that
%%
%\begin{equation} \label{eq:radii_polynomial}
%( Z_0 + Z_1 ) r_0 + Y_0 < r_0
%\end{equation}
%%
%then there exists a unique $\tilde c \in B_{r_0}(\bar c)$ such that $f(\tilde c) = 0$.
%
If 
\begin{equation} \label{eq:radii_polynomial}
Z_0 + Z_1 <1,
\end{equation}
then for all 
\[
r \in \left( \frac{Y_0}{1-Z_0-Z_1},\infty \right),
\]
there exists a unique $\tilde c \in B_r(\bar c)$ such that $f(\tilde c) = 0$. 
\end{thm}

\begin{proof}
We omit the details of this standard proof. Denote $\kappa \bydef Z_0 + Z_1 <1$. The idea is show that $T(c)\bydef c - A f(c)$ satisfies $T(B_r(\bar c)) \subset B_r(\bar c)$ and then that $\| T(c_1)-T(c_2) \|_{\cX^{m}_\nu} \le \kappa \| c_1-c_2 \|_{\cX^{m}_\nu}$ for all $c_1,c_2 \in B_r(\bar c)$. From the Banach fixed point theorem, there exists a unique $\tilde c \in B_r(\bar c)$ such that $T(\tilde c) = \tilde c$. The condition \eqref{eq:radii_polynomial} implies that $\| {\rm Id} - A A^\dagger \|_{B(\cX^{m}_\nu)} < 1$, and by construction of the operators $A$ and $A^\dagger$, it can be shown that $A$ is an injective operator. By injectivity of $A$, we conclude that there exists a unique $\tilde c \in B_r(\bar c)$ such that $f(\tilde c) = 0$.
\end{proof}

Given a Fourier projection dimension $m$, we apply Theorem~\ref{thm:radii_polynomial} to compute the solution of $2m+1$ problems of the form $f=0$ given in \eqref{eq:f_{ell,k}} with initial conditions 
\[
b = (b_k)_{k=-m}^m \in \{ {\rm e}_{1},{\rm e}_{2},\dots,{\rm e}_{2m+1} \}
\]
yielding a sequence of solutions $\tilde c^{(j)}:[-1,1] \to \C^{2m+1}$ ($j=-m,\dots,m$) with Chebyshev series representation 
\[
\tilde c^{(j)}(t) = \tilde c^{(j)}_{0,k} + 2 \sum_{\ell \ge 1} \tilde c^{(j)}_{\ell,k} T_\ell(t).
\]
Finally we can define the fundamental matrix solution $\Phi(t) \in M_{2m+1}(\C )$ satisfying \eqref{eq:finite_dim_variational_problem} as
\begin{equation} \label{eq:Psi}
\Phi(t) \bydef \begin{pmatrix} \vdots & \vdots &  & \vdots 
\\ \tilde c^{(1)}(t)  & \tilde c^{(2)}(t)  & \cdots & \tilde c^{(2m+1)}(t) 
\\
\vdots & \vdots &  & \vdots
\end{pmatrix}.
\end{equation}

Using a similar construction, we can construct $\Psi(s) \in M_{2m+1}(\C )$ satisfying \eqref{eq:finite_dim_adjoint_variational_problem}. The rest of this section is dedicated to the explicit construction of the bounds $Y_0$, $Z_0$ and $Z_1$ of Theorem~\ref{thm:radii_polynomial}. \\

\noindent{\bf The bound \boldmath$Y_0$\unboldmath.} Recalling the definition of the quadratic term $N_k$ of $f$ in \eqref{eq:explicit_Nk}, given the numerical solution $\bar c = (\bar c_{\ell,k})_{\ell=0,\dots,n-1 \atop k=-m,\dots,m}$, one gets that $N_{k,j}(\bar c)=0$ for all $j \ge 2n-1$. This implies that the term $f(\bar c)=(f_{\bfj}(\bar c))_{\bfj \in \cI}$ has only finitely many nonzero terms. Hence, the computation of $Y_0$ satisfying $\|Af(\bar c)\|_{\cX^{m}_\nu} \le Y_0$ is only a finite computation with interval arithmetic. \\

\noindent{\bf The bound \boldmath$Z_0$\unboldmath.} The computation of the bound $Z_0$ satisfying \eqref{eq:Z0} requires defining the operator
\[
B \bydef  {\rm Id} - A A^\dagger, 
\]
which action is given by
\[
(Bc)_{\ell,k} = 
\begin{cases}
\left( ({\rm Id} - A^{(n,m)} Df^{(n,m)}(\bar c) )c^{(n,m)} \right)_{\ell,k}, & 0 \le \ell < n, |k| \le m 
\\
0, & \ell \ge n, |k| \le m
\end{cases}
\]
Using interval arithmetic, compute $Z_0$ such that
\[
\| B \|_{B(\cX^{m}_\nu)} = \sup_{\bfj \in \cI} \frac{1}{\omega_{\bfj}} \sum_{\bfi \in \cI} | B_{\bfi,\bfj} | \omega_{\bfi}
= \max_{\ell_2=0,\dots,n-1 \atop |k_2| \le m} \frac{1}{\nu^{\ell_2}} \sum_{\ell_1=0,\dots,n-1 \atop |k_1| \le m} | B_{(\ell_1,k_1),(\ell_2,k_2)} | \nu^{\ell_1} \le Z_0.
\]

\noindent{\bf The bound \boldmath$Z_1$\unboldmath.} For any $c \in B_1(0)$, let
\[
z \bydef [Df(\bar c)-A^\dagger]c 
\]
which is given component-wise by 
\[
z_{\ell,k}
= z_{\ell,k}(\ba,c) \bydef 
\begin{cases}
\displaystyle
2 \sum_{j \ge n} (-1)^j c_{j,k}, & \ell=0, |k| \le m
\\
\displaystyle
h \expig \left( \cT (\ba*c^{(\infty,m)})_k \right)_{\ell}, & 0<\ell<n , |k| \le m
\\
\displaystyle
\lambda_k (\cT c_{k})_\ell + h \expig \left( \cT (\ba*c^{(m)})_k \right)_{\ell}, & \ell \ge n , |k| \le m.
\end{cases}
\]
Next, for the cases $0\le \ell<n$ and $|k| \le m$, we present component-wise uniform bounds $\hat z_{\ell,k} \ge 0$ such that $|z_{\ell,k}(\ba,c)| \le \hat z_{\ell,k}$ for all $c \in \cX^{m}_\nu$ with $\|c\|_{\cX^{m}_\nu} \le 1$. These bounds will then be used to define $Z_1$. First, given $c\in B_1(0)$ and $|k| \ge m$,
\begin{equation} \label{eq:hatz_{0,k}(c)}
|z_{0,k}(c)| = \left| 2 \sum_{j \ge n} (-1)^j c_{j,k} \right| \le 2 \sum_{j \ge n} |c_{j,k}| \frac{\nu^j}{\nu^j} \le
\frac{2}{\nu^n} \sum_{j \ge n} |c_{j,k}| \nu^j \le \frac{2}{\nu^n} \|c\|_{\cX^{m}_\nu} \le  \hat z_{0,k} \bydef \frac{2}{\nu^n}.
\end{equation} 

The next step is to obtain $\hat z_{\ell,k}$ for $\ell = 1,\dots, n-1$. This task involves understanding $\ba*c^{(\infty,m)}$ whose components can interpreted as 
%
%\begin{align*}
%(\ba*c^{(\infty,m)})_{\ell,k} &= \sum_{{{\ell_1+\ell_2 = \ell \atop k_1 + k_2 = k} \atop |\ell_1| <n \le |\ell_2|} \atop |k_1| \le N, |k_2| \le m} \ba_{\ell_1,k_1} c_{\ell_2,k_2} 
%\\
%&= \sum_{|\ell_2| \ge n \atop |k_2| \le m} \ba_{\ell-\ell_2,k-k_2} c_{\ell_2,k_2} 
%\\
%&= 
%\sum_{\ell_2 \ge n \atop |k_2| \le m} \ba_{\ell+\ell_2,k-k_2} c_{-\ell_2,k_2} 
%+
%\sum_{\ell_2 \ge n \atop |k_2| \le m} \ba_{\ell-\ell_2,k-k_2} c_{\ell_2,k_2} 
%\\
%&= 
%\sum_{\ell_2 \ge n \atop |k_2| \le m} \ba_{\ell-\ell_2,k-k_2} c_{\ell_2,k_2} 
%\\
%&= 
%\sum_{\ell_2 \ge n \atop |k_2| \le m} \ba_{\ell_2-\ell,k-k_2} c_{\ell_2,k_2} 
%\\
%&= \sum_{\ell_2 \ge 0 \atop |k_2| \le m} \alpha^{(\ell,k)}_{\ell_2,k_2} c_{\ell_2,k_2}, 
%\end{align*}
\[
(\ba*c^{(\infty,m)})_{\ell,k} = \sum_{{{\ell_1+\ell_2 = \ell \atop k_1 + k_2 = k} \atop |\ell_1| <n \le |\ell_2|} \atop |k_1| \le N, |k_2| \le m} \ba_{\ell_1,k_1} c_{\ell_2,k_2} = \sum_{\ell_2 \ge 0 \atop |k_2| \le m} \alpha^{(\ell,k)}_{\ell_2,k_2} c_{\ell_2,k_2}, 
\]
where
\[
\alpha^{(\ell,k)}_{\ell_2,k_2} \bydef 
\begin{cases}
0, & \ell_2 = 0,\dots,n-1
\\
\ba_{\ell_2-\ell,k-k_2},& \ell_2 = n,\dots,\ell+n-1.
\end{cases}
\]
For $0<\ell<n$ and $|k| \le m$, the term $(\ba*c^{(\infty,m)})_{\ell,k} \in \C$ can then be thought of as a linear functional acting on $c = (c_{\bfj})_{\bfj \in \cI} \in \cX^{m}_\nu$. Using that representation, we get that for all 
$c \in \cX^{m}_\nu$ with $\|c\|_{\cX^{m}_\nu} \le 1$,
\[
\left| (\ba*c^{(\infty,m)})_{\ell,k} \right| \le \sum_{\bfj \in \cI} \left| \alpha^{(\ell,k)}_{\bfj} \right|  |c_{\bfj}|
= \sum_{\bfj \in \cI} \frac{\left| \alpha^{(\ell,k)}_{\bfj} \right|}{\omega_\bfj}  |c_{\bfj}| \omega_\bfj
\le \Psi^{(\ell,k)}(\ba) \|c\|_{\cX^{m}_\nu} \le  \Psi^{(\ell,k)}(\ba)
\]
where
\begin{equation} \label{eq:Psi_kl}
\Psi^{(\ell,k)}(\ba) \bydef \sup_{\bfj \in \cI} \frac{\left| \alpha^{(\ell,k)}_{\bfj} \right|}{\omega_\bfj}
= \max_{|k_2|\le m \atop \ell_2 = n,\dots,\ell+n-1} \left\{ \frac{\left| \ba_{\ell_2-\ell,k-k_2}\right|}{\nu^{\ell_2}} \right\},
\end{equation}
which can easily be computed using interval arithmetic. For the cases $0\le \ell<n$ and $|k| \le m$, this leads to the bound
\begin{equation} \label{eq:hatz_{ell,k}(c)}
|z_{\ell,k}(\ba,c)| \le h \left| \left( \cT (\ba*c^{(\infty,m)})_k \right)_{\ell} \right|
\le \hat z_{\ell,k} \bydef h \left( |\cT| \Psi^{(\cdot,k)}(\ba) \right)_{\ell},
\end{equation}
for all $c \in B_1(0)$, where $|\cT|$ denotes the operator with component-wise absolute values. We are ready to obtain the bound $Z_1$. Given $c \in B_1(0)$, we use Lemma~\ref{lem:banach_algebra} to conclude that 
{
\begin{align*}
\| A [Df(\bar c)-A^\dagger]c\|_{\cX^{m}_\nu}  
&= \| Az \|_{\cX^{m}_\nu} 
\\
& = \sum_{\bfj \in \cI} |(Az)_{\bfj}| \omega_{\bfj} \\
& = \sum_{\ell = 0,\dots,n-1 \atop |k| \le m} \left( |A^{(n,m)}| \hat z^{(n,m)}\right)_{\ell,k} \nu^\ell 
+ \sum_{\ell \ge n \atop |k| \le m} \frac{1}{2 \ell} \left| \lambda_k (\cT c_{k})_\ell + h \expig \left( \cT (\ba*c^{(m)})_k \right)_{\ell} \right|\nu^\ell
\\
& \le \sum_{\ell = 0,\dots,n-1 \atop |k| \le m} \left( |A^{(n,m)}| \hat z^{(n,m)}\right)_{\ell,k} \nu^\ell 
+ \frac{|\lambda_m|}{2n} \sum_{\ell \ge n \atop |k| \le m} |  -c_{\ell-1,k} + c_{\ell+1,k}|\nu^\ell 
\\
& \quad + \frac{h}{2n} \sum_{\ell \ge n \atop |k| \le m} \left| \left( \cT (\ba*c^{(m)})_k \right)_{\ell} \right|\nu^\ell 
\\
& \le \sum_{\ell = 0,\dots,n-1 \atop |k| \le m} \left( |A^{(n,m)}| \hat z^{(n,m)}\right)_{\ell,k} \nu^\ell
+ \frac{1}{2n} \left( \nu + \frac{1}{\nu} \right) \left( |\lambda_m| \|c\|_{\cX^{m}_\nu}
+ h \| \ba*c\|_{\cX^{m}_\nu} \right)
\\
& \le \sum_{\ell = 0,\dots,n-1 \atop |k| \le m} \left( |A^{(n,m)}| \hat z^{(n,m)}\right)_{\ell,k} \nu^\ell
+ \frac{1}{2n} \left( \nu + \frac{1}{\nu} \right) \left( |\lambda_m|
+ 4 h \| \ba \|_{\cX^{m}_\nu} \right).
\end{align*}
}
Hence, by construction
\begin{equation} \label{eq:Z1_explicit}
Z_1 \bydef \sum_{\ell = 0,\dots,n-1 \atop |k| \le m} \left( |A^{(n,m)}| \hat z^{(n,m)}\right)_{\ell,k} \nu^\ell
+ \frac{1}{2n} \left( \nu + \frac{1}{\nu} \right) \left( |\lambda_m|
+ 4 h \| \ba \|_{\cX^{m}_\nu} \right),
\end{equation}
satisfies \eqref{eq:Z1}.

%\textbf{Conclusion.}
%Mention (or remember) the possibility of using the idea of Maxime and Laurent about inverting the (bounded) tridiagonal elements. That would allow us to consider 
%
%\[
%(A^\dagger c)_{\ell,k} = 
%\begin{cases}
%\left( Df^{(n,m)}(\bar c) c^{(n,m)} \right)_{\ell,k}, & 0 \le \ell < n, |k| \le m 
%\\
%\left( \Lambda c_{k} + \lambda_k \cT c_{k} \right)_\ell, & \ell \ge n, |k| \le m.
%\end{cases}
%\]

%%%%%%%%%%%%%%%%
\subsection{Generation of the evolution operator \boldmath$\bU^{(\infty)}(t,s)$\unboldmath~on~\boldmath$\ell^1$\unboldmath}\label{sec:evolution_operator}
%%%%%%%%%%%%%%%%
Finally, in this section, we verify the existence of the evolution operator $C^{(\infty)}(t,s)$ of the infinite dimensional equation \eqref{eq:linearizedeq_infinite}. This in turns will verify the existence of the evolution operator $\bU^{(\infty)}(t,s)$ generated on $\ell^1$. Moreover, we derive an estimate of $\bU^{(\infty)}(t,s)$ defined in \eqref{eq:bU_infty_definition} using the bounded operator norm on $\ell^1$, which is the hypothesis \eqref{eq:assumption_existence_U_infty} of Theorem~\ref{thm:sol_map}.

%Consider the infinite-dimensional part of \eqref{eq:linearizedeq}, that is the system of differential equations
%%
%\begin{equation}\label{eq:linearized_tail}
%\frac{d}{dt}c_k(t)+\expig\left[k^2\omega^2c_k(t)-2\left(\ba\left(t\right)*c^{(\infty)}(t)\right)_k\right]=0\quad(|k|>m)
%\end{equation}
%%
Let $\ell^1_{\infty}\bydef ({\rm Id}-\pi^{(m)}) \ell^1 = \left\{(a_k)_{k\in\Z }\in\ell^1:a_k=0~(|k|\le m)\right\}\subset\ell^1$ endowed with the norm $\|a\|_{\ell^1_\infty}\bydef \sum_{|k|>m}|a_k|$.
%Note that $\ell^1_\infty$ is also a Banach algebra under discrete convolution, namely
%%
%\begin{equation} \label{eq:banach_algebra_infty}
%\| a * b \|_{\ell^1_\infty} \le \| a \|_{\ell^1_\infty} \| b \|_{\ell^1_\infty}, \quad \text{for all } a,b \in \ell^1_\infty. 
%\end{equation}
Consider the Laplace operator $L$ defined in \eqref{eq:mul_op} whose action is restricted on $\ell^1_{\infty}$, that is
\[
	(La)_k = \begin{cases}
	0, & |k|\le m\\
	-k^2\omega^2a_k, & |k|>m,
	\end{cases}\quad \forall a \in D(L),
\]
where the domain of the operator $L$ is $D(L) = \left\{a\in\ell^1_\infty:La\in\ell^1_\infty\right\}$ in this case.
Denoting $\fL=\expig L$, 
it is easy to see that $\fL:D(L)\subset\ell^1_\infty\to\ell^1_\infty$ is a densely defined closed operator on $\ell^1_\infty$.
The operator $\fL$ then generates the semigroup on $\ell^1_\infty$ (e.g. see \cite{pazy1983semigroups}), which is denoted by $\left\{e^{\fL t}\right\}_{t\ge 0}$. Furthermore, the action of such a semigroup $\left\{e^{\fL t}\right\}_{t\ge 0}$ can be naturally extended to $\ell^1$ by
\begin{equation} \label{eq:semigroup_on_ell1}
	\left(e^{\fL t}\phi\right)_k=\begin{cases}
	0,&|k|\le m\\
	(e^{\fL t}( \phi_k)_{|k| > m})_{k}, &|k|>m,
	\end{cases}\quad \forall \phi\in\ell^1.
\end{equation}
In the following, unless otherwise noted, we consider the semigroup $\left\{e^{\fL t}\right\}_{t\ge 0}$ on $\ell^1$ as defined in \eqref{eq:semigroup_on_ell1}.
We also have the following estimate:
\begin{equation}\label{eq:semigroup_estimate}
\left\|e^{\fL t}\right\|_{B (\ell^1)}\le e^{-\mu_{m+1} t}, \quad\mu_{m+1} = (m+1)^2\omega^2\cos\theta.
\end{equation}
We re-write \eqref{eq:linearizedeq_infinite} as the Cauchy problem in $\ell^1$
\begin{equation}\label{eq:linearized_tail2}
\frac{d}{dt}c^{(\infty)}(t) - \fL c^{(\infty)}(t)  = 2 \expig ({\rm Id} -\pi^{(m)}) \left(\ba\left(t\right)*c^{(\infty)}(t)\right)
\end{equation}
for any initial sequence $c^{(\infty)}(s)=\phi^{(\infty)}$.
Showing the existence of the solution of \eqref{eq:linearized_tail2},
the proof of existence of the evolution operator $\bU^{(\infty)}(t,s)$ is given by the following theorem.
%\[
%c_k(t) = \left(C^{(\infty)}(t,s)  ( \phi_k)_{|k| > m} \right)_k\quad(|k|>m).
%\]

\begin{thm}\label{thm:ev_op}
For the infinite-dimensional system of differential equations \eqref{eq:linearized_tail2}, there exists a unique solution that solves the integral equation in $\ell^1$
\[
c^{(\infty)}(t)=e^{\fL (t-s)}\phi^{(\infty)} + 2\expig\int_s^t e^{\fL(t-\tau)}({\rm Id} -\pi^{(m)})(\ba(\tau)*c^{(\infty)}(\tau)) d \tau.
\]
Furthermore, the evolution operator $\bU^{(\infty)}(t,s)$ exists and the following estimate holds
\[
	\|\bU^{(\infty)}(t,s)\phi^{(\infty)}\|_{\ell^1}\le e^{-\mu_{m+1}(t-s)+2\int_s^t\|\ba(\tau)\|_{\ell^1}d \tau}\|\phi ^{(\infty)}\|_{\ell^1},\quad\forall\phi^{(\infty)}\in\ell^{1}.
\]
\end{thm}

\begin{proof}
	For a fixed $s\ge 0$, let us define a map $\cP: X\to X $ acting on the $c^{(\infty)}(t)$ as
\[
\cP  c^{(\infty)}(t)\bydef e^{\fL (t-s)}\phi ^{(\infty)} +2\expig\int_s^t e^{\fL(t-\tau)}({\rm Id} -\pi^{(m)})(\ba(\tau)*c^{(\infty)}(\tau)) d \tau
\]
and let a function space $\cX_\infty$ be defined by
\[
\cX_\infty \bydef\left\{c^{(\infty)}(t)\in\ell^1:
\sup_{0 \le s\le t\le h}
e^{\mu_{m+1} (t-s)}\left\| c^{(\infty)}(t)\right\|_{\ell^1}<\infty
\right\}
\]
with the distance
\[
\mathbf{d}\left(c_1^{(\infty)},c_2^{(\infty)}\right)\bydef\sup_{0 \le s\le t\le h}\left(e^{
	\mu_{m+1}(t-s)-2\beta\int_s^t\|\ba(\tau)\|_{\ell^1} d \tau
}\left\|
c_1^{(\infty)}(t)-c_2^{(\infty)}(t)
\right\|_{\ell^1}\right),\quad \beta>1.
\]
We prove that the map $\cP $ becomes a contraction mapping under the distance $\mathbf{d}$ on $\cX_\infty$.
For $c_1^{(\infty)},~c_2^{(\infty)}\in \cX_\infty$, we have using \eqref{eq:semigroup_estimate} and the property \eqref{eq:banach_algebra}
\begin{align*}
&e^{\mu_{m+1}(t-s)-2\beta\int_s^t\|\ba(\tau)\|_{\ell^1} d \tau}\left\| \cP  c_1^{(\infty)}(t)-\cP  c_2^{(\infty)}(t)\right\|_{\ell^1}\\
&\le e^{\mu_{m+1}(t-s)-2\beta\int_s^t\|\ba(\tau)\|_{\ell^1}d \tau}2\int_s^te^{-\mu_{m+1}(t-\tau)}\|\ba(\tau)\|_{\ell^1}\|c_1^{(\infty)}(\tau)-c_2^{(\infty)}(\tau)\|_{\ell^1}d \tau\\
&\le e^{\mu_{m+1}(t-s)-2\beta\int_s^t\|\ba(\tau)\|_{\ell^1}d \tau}\mathbf{d}\left(c_1^{(\infty)},c_2^{(\infty)}\right)2\int_s^te^{-\mu_{m+1}(t-\tau)}\|\ba(\tau)\|_{\ell^1}e^{
	-\mu_{m+1}(\tau-s)+2\beta\int_s^\tau\|\ba(\sigma)\|_{\ell^1}d \sigma}d \tau\\
&=e^{-2\beta\int_s^t\|\ba(\tau)\|_{\ell^1}d \tau}\mathbf{d} \left(c_1^{(\infty)},c_2^{(\infty)}\right) 2\int_s^t\|\ba(\tau)\|_{\ell^1}e^{2\beta\int_s^\tau\|\ba(\sigma)\|_{\ell^1}d \sigma}d \tau\\
&=e^{-2\beta\int_s^t\|\ba(r)\|_{\ell^1}dr}\mathbf{d} \left(c_1^{(\infty)},c_2^{(\infty)}\right)\left[\frac{1}{\beta}e^{2\beta\int_s^r\|\ba(\tau)\|_{\ell^1}d\tau}\right]_{r=s}^{r=t}\\
&\le\frac{1}{\beta}\mathbf{d}\left(c_1^{(\infty)},c_2^{(\infty)}\right).
\end{align*}
Since $\beta>1$, $\cP $ becomes a contraction mapping on $\cX_\infty$.
This yields that the solution of \eqref{eq:linearized_tail2} uniquely exists in $\cX_\infty$, which satisfies
\begin{equation}\label{eq:integral_form}
%c^{(\infty)} (t)=e^{\fL (t-s)}\phi ^{(\infty)} +2\expig\int_s^t e^{\fL(t-r)}(\ba(r)* c^{(\infty)}(r))dr.
c^{(\infty)}(t)=e^{\fL (t-s)}\phi^{(\infty)} + 2\expig\int_s^t e^{\fL(t-\tau)}({\rm Id} -\pi^{(m)})(\ba(\tau)*c^{(\infty)}(\tau)) d \tau.
\end{equation}
Moreover, letting $y(t)\bydef e^{\mu_{m+1}(t-s)}\| c^{(\infty)}(t)\|_{\ell^1}$, it follows from \eqref{eq:integral_form} using \eqref{eq:semigroup_estimate} and the property \eqref{eq:banach_algebra}, that 
\begin{align*}
y(t)\le\|\phi^{(\infty)}\|_{\ell^1}+2\int_s^te^{\mu_{m+1} (\tau-s)}\|\ba(\tau)\|_{\ell^1}\| c^{(\infty)}(\tau)\|_{\ell^1}d \tau=\|\phi^{(\infty)}\|_{\ell^1}+2\int_s^t\|\ba(r)\|_{\ell^1}y(\tau)d\tau.
\end{align*}
Gr\"onwall's inequality yields
\[
y(t)\le\|\phi^{(\infty)}\|_{\ell^1}e^{2\int_s^t\|\ba(r)\|_{\ell^1}}dr.
\]
Then, we conclude that the following inequality holds
\begin{equation}\label{eq:ev_op_estimate2}
\left\|\bU^{(\infty)}(t,s)\phi ^{(\infty)}\right\|_{\ell^1}\le\|\phi ^{(\infty)}\|_{\ell^1}e^{-\mu_{m+1}(t-s)+\int_s^t\|\ba(\tau)\|_{\ell^1}d \tau}=W^{(\infty)}(t,s)\|\phi ^{(\infty)}\|_{\ell^1}
\end{equation}
for any $\phi ^{(\infty)}\in\ell^1$,
where $W^{(\infty)}(t,s)$ is defined in \eqref{eq:assumption_existence_U_infty}.
%This completes the proof.
\end{proof}

\subsection{Rigorous construction of the solution map operator}\label{sec:solution_map_operator_construction}

Recall that the definition of the solution map operator $\mathscr{A}$ in \eqref{eq:definition_of_scriptA} requires proving the existence of the evolution operator $U(t,s)$, which is done by verifying the main three hypotheses of Theorem~\ref{thm:sol_map}, namely \eqref{eq:bound_W_m}, \eqref{eq:assumption_existence_U_infty} and \eqref{eq:kappa_condition}. Hypothesis \eqref{eq:bound_W_m} is verified in practice by using the theory of Section~\ref{sec:fundamental_solution} to compute rigorously (with Chebyshev series) the fundamental matrix solutions $\Phi(t)$ and $\Psi(s)$ of problems \eqref{eq:finite_dim_variational_problem} and \eqref{eq:finite_dim_adjoint_variational_problem}, respectively, and then using \eqref{eq:evaluating_W_m}.
Hypothesis \eqref{eq:assumption_existence_U_infty} has been verified in full generality in Section~\ref{sec:evolution_operator} in Theorem~\ref{thm:ev_op}. Hypothesis \eqref{eq:kappa_condition} is computational and requires computing the constant $\kappa$ with interval arithmetic. Once all the hypotheses are verified successfully, the solution map operator $\mathscr{A}$ exists, is defined as in \eqref{eq:definition_of_scriptA} and can be used to study the fixed point of the simplified Newton operator $T$ defined in \eqref{eq:simp_Newton_op}. The fixed point of $T$ is obtained using the approach of Section~\ref{sec:Localinclusion}. The fixed point so obtained yields the unique solution of $F=0$ (with $F$ defined in \eqref{eq:F=0_Cauchy_problem}) and by construction, this yields the local rigorous inclusion of the solution of the Cauchy problem on the time interval $[0,h]$. 

%Theorem \ref{thm:ev_op} proves that the evolution operator $\bU^{(\infty)}(t,s)$ exists and it yields the form $c ^{(\infty)}(t)=\bU^{(\infty)}(t,s)\phi^{(\infty)}$.
%From this representation, we have the following corollary:
%\begin{cor}
%	For the evolution operator $U^{(\infty)}(t,s)$,
%		\begin{equation}\label{eq:ev_op_estimate2}
%	\left\|U^{(\infty)}(t,s)\phi ^{(\infty)}\right\|_{\ell^1}\le\|\phi ^{(\infty)}\|_{\ell^1}e^{-\mu_{m+1}(t-s)+\int_s^t\|\ba(r)\|_{\ell^1}dr}=W^{(\infty)}(t,s)\|\phi ^{(\infty)}\|_{\ell^1}
%	\end{equation}
%holds for any $\phi ^{(\infty)}\in\ell^1_\infty$.
%\end{cor}

%%%%%%%%%%%%%%%%
\section{Local inclusion in a time interval}\label{sec:Localinclusion}
%%%%%%%%%%%%%%%%
%\corrc Notation: I use $\alpha\mapsto\varrho$ for consistency of my previous work. Akitoshi <<>>
%\corrc I am here. Akitoshi <<>>

In this section, we present a sufficient condition whether a fixed point of the simplified Newton operator \eqref{eq:simp_Newton_op} exists and is unique in $B_{J}(\ba,\varrho)$ defined by \eqref{eq:Ball}.
Such a sufficient condition can be rigorously checked by numerical computations based on interval arithmetic.
\begin{thm}\label{thm:main_theorem}
	Consider the Cauchy problem \eqref{eq:CGL_ode}.
	For a given initial sequence $a(0)$ and its approximation $\ba(0)$, assume that there exists $\varepsilon\ge 0$ such that $\|a(0)-\ba(0)\|_{\ell^1}\le\varepsilon$.
	Assume also that $\ba\in C^1(J;D(L))$ and any $a\in B_J\left(\ba,\varrho\right)$ satisfies
	\begin{align*}
	\displaystyle\sup_{t\in J}\sum_{k\in\Z }\left|\left[ T(a)(t)-\ba(t)\right]_k\right|\le f_{\varepsilon}\left(\varrho\right),
	\end{align*}
	where $f_{\varepsilon}(\varrho)$ is defined by
	\begin{align}\label{eq:f}
	f_{\varepsilon}\left(\varrho\right)\bydef \bm{W_h}\left[\varepsilon+h\left(2\varrho^2+\delta\right)\right].
	\end{align}
	Here, $\bm{W_h}>0$ and $\delta> 0$ satisfy
$\sup_{(t,s)\in\mathcal{S}_h}\left\|U(t,s)\right\|_{B (\ell^1)}\le \bm{W_h}$ and
$\left\|F(\ba)\right\|_{X}\le\delta$, respectively.
	If
	\[
	f_{\varepsilon}\left(\varrho\right)\le\varrho
	\]
	holds, then the Fourier coefficients $\ta$ of the solution of \eqref{eq:CGL} are rigorously included in $B_J\left(\ba,\varrho\right)$ and are unique in $B_J\left(\ba,\varrho\right)$.
\end{thm}
The proof of this theorem is based on Banach fixed point theorem.
Before proving Theorem~\ref{thm:main_theorem}, we prepare the following lemma:

\begin{lem}\label{lem:mean-value-form}
For a given sequence $\ba$, the discrete convolution follows for two bi-infinite sequences $a=(a_k)_{k\in \Z }$ and $b=(b_k)_{k\in\Z }$
\begin{align*}
a*a-b*b-2\ba*(a-b)=2\left\{\int_0^1\left[\eta\left(a-\ba\right)+(1-\eta)\left(b-\ba\right)\right]d\eta\right\}*(a-b).
\end{align*}
\end{lem}
\begin{proof}
From the fundamental theorem of calculus, it is easy to see 
\begin{align*}
a*a-b*b-2\ba*(a-b)
&=\int_0^1\frac{d}{d\eta}\left[\left(\eta a+(1-\eta)b\right)*\left(\eta a+(1-\eta)b\right)-2\eta\ba*(a-b)\right]d\eta\\
&=2\left\{\int_0^1\left(\eta a+(1-\eta)b-\ba\right)d\eta\right\}*(a-b). \qedhere
\end{align*}
\end{proof}

\begin{proof}[Proof of Theorem \ref{thm:main_theorem}]
%Let $B_J(\ba,\varrho)$ be the ball defined in \eqref{eq:Ball}.
On the basis of Banach fixed point theorem, we prove that the simplified Newton operator $T$ defined by \eqref{eq:simp_Newton_op} becomes a contraction mapping on $B_J(\ba,\varrho)$.
It is sufficient to show that the following two conditions hold:
\begin{enumerate}
	\item $T(a)\in B_J(\ba,\varrho) $ for any $a\in B_J(\ba,\varrho)$,
	\item there exists $\kappa\in [0,1)$ such that $\mathbf{d}(T(a_1),T(a_2))\le \kappa \mathbf{d}(a_1,a_2)$ for $a_1,a_2\in B_J(\ba,\varrho) $ with a distance $\mathbf{d}$ in $B_J(\ba,\varrho)$.
\end{enumerate}

Firstly, for a sequence $a\in B_J(\ba,\varrho) $ , we have using \eqref{eq:simp_Newton_op}
\begin{align}\label{eq:Ta-ta}
T(a)-\ba&=T(a)-T(\ba)+T(\ba)-\ba\nonumber\\
&= \mathscr{A}\left[\expig(a*a-\ba*\ba-2 \ba*(a-\ba))\right]-\mathscr{A}F\left(\ba\right).
\end{align}
Let $z\bydef a-\ba$.
Using Lemma \ref{lem:mean-value-form} with $b=\ba$, the first term of \eqref{eq:Ta-ta} is given by 
\begin{align*}%\label{eq:Ta-Tta_finite}
\mathscr{A}\left[\expig(a*a-\ba*\ba-2 \ba*(a-\ba))\right]= \mathscr{A}\left[2\expig\int_0^1\eta d\eta(z*z)\right].
\end{align*}
Thus, \eqref{eq:Ta-ta} is represented by
\begin{align*}
T(a)-\ba
&= \mathscr{A}\left(2\expig\int_0^1\eta d\eta(z*z)_k-F_k(\ba)\right)_{k\in\Z },
\end{align*}
where $\mathscr{A}$ is the solution map operator defined in Section \ref{sec:solution_operator} and
\begin{align}\label{eq:defect}
F_k(\ba)=
\begin{cases}
\frac{d}{dt}\ba_k-\expig\left((L\ba)_k+(\ba*\ba)_k\right),&|k|\le N\\[1mm]
-\expig (\ba*\ba)_k,&|k|>N.
\end{cases}
\end{align}
Taking $\ell^1$ norm of $T(a)-\ba$, we have
\begin{align}\label{eq:Ta-ta_norm}
\|T(a)-\ba\|_{\ell^1}&=\sum_{k\in\Z }\left|(T(a)-\ba)_k\right|\nonumber\\
&=\left\| U(t,0)z(0)+\int_0^tU(t,s)g(s)ds\right\|_{\ell^1}\nonumber\\
&\le\left\| U(t,0)z(0)\right\|_{\ell^1}+\int_0^t\left\|U(t,s)g(s)\right\|_{\ell^1}ds,
\end{align}
where
\begin{equation}\label{eq:g(s)}
g(s)\bydef2\expig\int_0^1\eta d\eta(z(s)*z(s))-\left(F(\ba)\right)(s).
\end{equation}
Taking $\ell^1$-norm of $g$, we have using the property \eqref{eq:banach_algebra}
\begin{align}\label{eq:g(s)_norm}\nonumber
\left\|g(s)\right\|_{\ell^1}&\le \sum_{k\in\Z }\left|2\expig(z(s)*z(s))_k\right|+\left\|\left(F(\ba)\right)(s)\right\|_{\ell^1}\\
&\le 2\|z(s)\|_{\ell^1}^2+\delta,
\end{align}
where $\delta$ satisfies $\sup_{s\in J}\left\|\left(F(\ba)\right)(s)\right\|_{\ell^1}\le\delta$.
Since  $a\in B_J(\ba,\varrho)$, $\|z\|_{X}\le\varrho$ holds.
Finally, \eqref{eq:Ta-ta_norm} is bounded by using the uniform bound $\bm{W_h}$ \eqref{eq:W_h_constant} discussed in the previous section as
\begin{align*}
\sup_{t\in J}\sum_{k\in\Z }\left|\left((T(a))(t)-\ba(t)\right)_k\right|&\le \sup_{t\in J}\left\| U(t,0)z(0)\right\|_{\ell^1}+ \sup_{t\in J}\int_0^t\left\|U(t,s)g(s)\right\|_{\ell^1}ds\\
&\le \bm{W_h} \left[\varepsilon+h\left(2\varrho^2+\delta\right)\right]=f_{\varepsilon}\left(\varrho\right),
\end{align*}
where $\varepsilon$ is the upper bound of the initial error such that $\left\|z(0)\right\|_{\ell^1}\le \varepsilon$.
From the assumption $f_{\varepsilon}(\varrho)\le\varrho$, $T(a)\in B_J(\ba,\varrho) $ holds for any $a\in B_J(\ba,\varrho) $.

Secondly, we will show the contraction property of $T$.
For sequences $a_1,a_2\in B_J(\ba,\varrho)$, we define the distance in $B_J(\ba,\varrho)$ as
\[
	\mathbf{d}(a_1,a_2)\bydef \left\|a_1-a_2\right\|_X.
\]
The analogous discussion above yields
\begin{align}\label{eq:Ta-Tb}
T(a_1)-T(a_2) = \mathscr{A}\left[\expig\left(a_1*a_1-a_2*a_2-2\ba*(a_1-a_2)\right)\right].
\end{align}
Let $\zeta\bydef a_1-a_2$ with $\zeta(0)=0$. It follows
\begin{align}\label{eq:Ta-Tb_norm}
\left\|T(a_1)-T(a_2)\right\|_{\ell^1}=\sum_{k\in\Z }\left|\left(T(a_1)-T(a_2)\right)_k\right|
\le\int_0^t\left\|U(t,s)\tilde{g}(s)\right\|_{\ell^1}ds,
\end{align}
where $\tilde{g}$ is defined using Lemma \ref{lem:mean-value-form} with $a=a_1$ and $b=a_2$ by
\begin{align*}
\tilde{g}(s)&\bydef2\expig\left(\int_0^1\left[\eta (a_1-\ba)+(1-\eta)(a_2-\ba)\right]d\eta*\zeta\right).
\end{align*}
Since $B_J(\ba,\varrho)$ is convex, $\eta (a_1-\ba)+(1-\eta)(a_2-\ba)\in B_J(0,\varrho)$ holds for any $\eta\in(0,1)$.
Then, from \eqref{eq:Ta-Tb} and \eqref{eq:Ta-Tb_norm}, the distance is estimated by
\begin{align*}%\label{eq:contraction}
\mathbf{d}\left(T(a_1),T(a_2)\right)\le\left(2\bm{W_h} h \varrho\right)\mathbf{d}(a_1,a_2).
\end{align*}
Taking $\kappa = 2\bm{W_h}h\varrho$,
it follows $\kappa <f_{\varepsilon}(\varrho)/\varrho\le1$ from the assumption of theorem.
%Then, from \eqref{eq:contraction}, 
It is proved that the simplified Newton operator $T$ becomes the contraction mapping on $B_J(\ba,\varrho)$.
\end{proof}

\begin{rem}
Our task in the practical implementation is to rigorously compute the minimum values $\varrho$ 
such that
$f_{\varepsilon}\left(\varrho\right)\le\varrho$
by using interval arithmetic.
\end{rem}

\subsection{The bound \boldmath$\varepsilon$\unboldmath}
We show how we get the $\varepsilon$ bound such that $\|a(0)-\ba(0)\|_{\ell^1}\le\varepsilon$.
From \eqref{eq:app_sol}
\begin{align*}
\ba_k(0)=\ba_{0,k} + 2\sum_{\ell=1}^{n-1}\ba_{\ell,k}T_{\ell}(0)=\ba_{0,k}-2\ba_{1,k}+2\ba_{2,k}-\dots+(-1)^{n-1}2\ba_{n-1,k}\quad(|k|\le N),
\end{align*}
where we used the fact $T_{\ell}(0)=(-1)^{\ell}$.
Then, using interval arithmetic, $\varepsilon$ is given by
\begin{align*}
\varepsilon\bydef\sum_{|k| \le N}\left|\varphi_k-\left(\ba_{0,k}-2\ba_{1,k}+2\ba_{2,k}-\dots+(-1)^{n-1}2\ba_{n-1,k}\right)\right|.
\end{align*}

\subsection{The bound \boldmath$\delta$\unboldmath}\label{sec:defect_bound}
We also show how we get the defect bound of $F$ at the approximate solution $\ba$.
From \eqref{eq:defect}, we recall
\[
F_k(\ba)=
\begin{cases}
\frac{d}{dt}\ba_k-\expig\left(-k^2\omega^2\ba_k+(\ba*\ba)_k\right),&|k|\le N,\\[1mm]
-\expig (\ba*\ba)_k,&|k|>N.
\end{cases}
\]
Here, we suppose that the first derivative of $\ba$ is expressed by
\[
\frac{d}{dt}\ba_k(t)=\sum_{\ell=0}^{n-2} \ba_{\ell,k}^{(1)}T_{\ell}(t),
\]
where $\ba_{\ell,k}^{(1)}\in\mathbb{C}$ can be computed by an recursive algorithm (see, e.g., \cite[page 34]{bib:mason2002}).
For $|k|\le N$, we have
\begin{align*}
	F_k(\ba) &= \frac{d}{dt}\ba_k-\expig\left(-k^2\omega^2\ba_k+(\ba*\ba)_k\right)\\
	&= \sum_{\ell=0}^{n-2} \ba_{\ell,k}^{(1)}T_{\ell}(t)+\sum_{\ell=0}^{n-1}\left(\expig k^2\omega^2\mathfrak{w}_{\ell}\ba_{\ell,k}-\sum_{{\ell_1+\ell_2 = \pm\ell \atop k_1 + k_2 = k} \atop |\ell_i| <n , |k_i| \le N}\ba_{|\ell_1|,k_1}\ba_{|\ell_2|,k_2}\right)T_{\ell}(t)\\
	&\hphantom{=}\quad - \sum_{\ell\ge n}\left(\sum_{{\ell_1+\ell_2 = \pm\ell \atop k_1 + k_2 = k} \atop |\ell_i| <n , |k_i| \le N}\ba_{|\ell_1|,k_1}\ba_{|\ell_2|,k_2}\right)T_{\ell}(t),
\end{align*}
where $\mathfrak{w}_{\ell}$ denotes a weight for the Chebyshev coefficients such that
\[
	\mathfrak{w}_{\ell}=\begin{cases}
	1, & \ell=0\\
	2, & \ell> 0.
	\end{cases}
\]
It then follows
\begin{align*}
&\left|\frac{d}{dt}\ba_k-\expig\left(-k^2\omega^2\ba_k+(\ba*\ba)_k\right) \right|\\
&\le \sum_{\ell=0}^{n-2} \left|\ba_{\ell,k}^{(1)}+\expig k^2\omega^2\mathfrak{w}_{\ell}\ba_{\ell,k}-\sum_{{\ell_1+\ell_2 = \pm\ell \atop k_1 + k_2 = k} \atop |\ell_i| <n , |k_i| \le N}\ba_{|\ell_1|,k_1}\ba_{|\ell_2|,k_2}\right|+\left|2\expig k^2\omega^2\ba_{n-1,k}-\sum_{{\ell_1+\ell_2 = \pm(n-1) \atop k_1 + k_2 = k} \atop |\ell_i| <n , |k_i| \le N}\ba_{|\ell_1|,k_1}\ba_{|\ell_2|,k_2}\right|\\
&\hphantom{\le}\quad + \sum_{\ell\ge n}\left|\sum_{{\ell_1+\ell_2 = \pm\ell \atop k_1 + k_2 = k} \atop |\ell_i| <n , |k_i| \le N}\ba_{|\ell_1|,k_1}\ba_{|\ell_2|,k_2}\right|.
\end{align*}

Furthermore, for $|k|>N$, the tail part is given by
\[
	\left|(\ba*\ba)_k\right| \le\sum_{\ell\ge 0}\left| \sum_{{\ell_1+\ell_2 = \pm\ell \atop k_1 + k_2 = k} \atop |\ell_i| <n , |k_i| \le N}\ba_{|\ell_1|,k_1}\ba_{|\ell_2|,k_2}\right|.
\]

Finally, the defect bound $\delta$ is given by
\begin{align*}
&\sup_{t\in J}\left\|(F(\ba))(t)\right\|_{\ell^1}\\
&=\sup_{t\in J}\left(\sum_{|k| \le N}\left|F_k(\ba)\right| + \sum_{|k|>N}\left|F_k(\ba)\right|\right)\\
&\le \sum_{|k| \le N}\left(
\sum_{\ell=0}^{n-2} \left|\ba_{\ell,k}^{(1)}+\expig k^2\omega^2\mathfrak{w}_{\ell}\ba_{\ell,k}-\sum_{{\ell_1+\ell_2 = \pm\ell \atop k_1 + k_2 = k} \atop |\ell_i| <n , |k_i| \le N}\ba_{|\ell_1|,k_1}\ba_{|\ell_2|,k_2}\right|\right.\\
&\hphantom{\sum_{|k| \le N}\Bigg(}\left.\quad+\left|2\expig k^2\omega^2\ba_{n-1,k}-\sum_{{\ell_1+\ell_2 = \pm(n-1) \atop k_1 + k_2 = k} \atop |\ell_i| <n , |k_i| \le N}\ba_{|\ell_1|,k_1}\ba_{|\ell_2|,k_2}\right|+ \sum_{\ell\ge n}\left|\sum_{{\ell_1+\ell_2 = \pm\ell \atop k_1 + k_2 = k} \atop |\ell_i| <n , |k_i| \le N}\ba_{|\ell_1|,k_1}\ba_{|\ell_2|,k_2}\right|
\right)\\
&\quad + \sum_{|k|>N}\sum_{\ell\ge 0}\left| \sum_{{\ell_1+\ell_2 = \pm\ell \atop k_1 + k_2 = k} \atop |\ell_i| <n , |k_i| \le N}\ba_{|\ell_1|,k_1}\ba_{|\ell_2|,k_2}\right|\bydef\delta.
\end{align*}
This seems to be infinite sum but thanks to the finiteness of nonzero elements of $\ba_{\ell,k}$.
It becomes finite sum and can be rigorously computed based on interval arithmetic.
For the rigorous computing of the discrete convolution, one can consult a FFT based algorithm by, e.g., \cite{Lessard2018}.

%%%%%%%%%%%%%%%%
\section{Time stepping scheme}\label{sec:timestepping}
%%%%%%%%%%%%%%%%

%\corrc I am here. Akitoshi<<>>

Once the rigorous inclusion of Fourier coefficients is obtained, we consider extending the time interval, say \emph{time step}, in which the existence of the solution is verified.
For this purpose the initial sequence in the next time step is replaced by a sequence at the endpoint of the current time step (e.g., $a(h)$ for the first time step).
Replacing $J=(h,2h)$, we apply Theorem \ref{thm:main_theorem} for the initial-boundary value problem on the next time step and recursively repeat this process several times.
We introduce such a time stepping scheme in this section.

For $K\in\N$, let $0=t_0<t_1<\dots<t_K<\infty$ be grid points of the time variable.
We call $J_{i}=(t_{i-1},t_{i})$ the $i$\,th time step and let $h_i=t_{i}-t_{i-1}$ ($i=1,2,\dots,K$) be the step size.
Now we assume that the solution $a(t)=(a_k(t))_{k\in\Z}$ of \eqref{eq:CGL_ode} is rigorously included until $J_K$, i.e., 
\[a(t)\in B_{J_i}\left(\ba,\varrho_i\right)\bydef\left\{a:\|a-\ba\|_{C(J_i;\ell^1)} \le \varrho_i~\mbox{with initial data}~a(t_{i-1})\right\} \quad (t\in J_i)
\]
holds for some $\varrho_i>0$.
In the following, we derive the error bound $\varepsilon_{i}$ ($i=1,2,\dots,K$) at each endpoint of time interval, which satisfies $\|a(t_{i})-\ba(t_{i})\|_{\ell^1}\le\varepsilon_i$.
We call such an error estimate the \emph{point-wise error estimate}.

Let us show the case of first time step.
Recall that $z(t)= a(t)-\ba(t)$ ($t\in J_1$).
At the endpoint of the time step $t= t_1$, it follows from \eqref{eq:Ta-ta_norm}
\begin{equation}\label{eq:err_step1}
	\|z(t_1)\|_{\ell^1}\le\left\|U(t_1,t_0)z(t_0)\right\|_{\ell^1}+\int_{t_0}^{t_1}\left\|U(t_1,s)g(s)\right\|_{\ell^1}ds,
\end{equation}
where $g(s)$ is defined in \eqref{eq:g(s)} and $\left\|g(s)\right\|_{\ell^1}\le 2\varrho_1^2+\delta_1$ holds from \eqref{eq:g(s)_norm}.
Here, the positive constant $\delta_1$ satisfies $\sup_{s\in J_1}\left\|\left(F(\ba)\right)(s)\right\|_{\ell^1} \le \delta_1$, which is given in Section \ref{sec:defect_bound}.
We assume that the initial error is bounded by $\|z(t_0)\|_{\ell^1}\le\varepsilon_0$, which is equal to the hypothesis of Theorem \ref{thm:main_theorem} in Section \ref{sec:Localinclusion}.
Then, to derive the bound $\varepsilon_{1}$, we need two positive constants ${W_{J_1}}>0$ and ${W_{t_1}}>0$ such that
$\sup_{s\in J_1}\|U(t_1,s)\|_{B(\ell^1)}\le W_{J_1}$ and $\|U(t_1,t_0)\|_{B(\ell^1)}\le W_{t_1}$, respectively.

We reconsider the linearized problem \eqref{eq:linearized_problem} and represent the solution of  \eqref{eq:linearized_problem} as $b_s(t)\equiv b(t)$.
%To get the point-wise error estimate, we derive
%\[
%	\sup_{s\in J_1}\|b_s(t_1)\|_{\ell^1}=\sup_{s\in J_1}\|U(t_1,s)\phi\|_{\ell^1}\le W_{J_1}\|\phi\|_{\ell^1},\quad\forall\phi\in\ell^1
%\]
%and setting the initial time $s=t_0$, then
%\[
%	\|b_{t_0}(t_1)\|_{\ell^1}=\|U(t_1,t_0)\phi\|_{\ell^1}\le W_{t_1}\|\phi\|_{\ell^1},\quad\forall\phi\in\ell^1.
%\]
As the analogous discussion in Section \ref{sec:solution_operator}, we split the solution $b_s(t)$ into the finite mode $b^{(m)}_s$ and the tail mode $b^{(\infty)}_s$.
The finite mode is given by plugging $t=t_1$ in \eqref{eq:b^{m}_integral_equation}
\[
	b^{(m)}_s(t_1)=\bU^{(m)}(t_1, s) \phi^{(m)}+2 \expig\int_{s}^{t_1} \bU^{(m)}(t_1, \tau) \pi^{(m)}\left(\ba(\tau) * b^{(\infty)}_s(\tau)\right) d \tau.
\]
%As shown in , $C^{(m)}(t_1, s)=\Phi(t_1)\Psi(s)$ holds.
By the definition of $\bU^{(m)}(t,s)$ given in \eqref{eq:bU_m_definition}, $\|\bU^{(m)}(t,s)\|_{B(\ell^1)}=\|C^{(m)}(t,s)\|_1$ holds for $(t,s)\in\mathcal{S}_h$.
Furthermore, using the form $C^{(m)}(t_1, s)=\Phi(t_1)\Psi(s)$ defined in \eqref{eq:Cm_decomposition}, the property \eqref{eq:banach_algebra} and \eqref{eq:b_inf_norm}, we have
\begin{align*}
\left\|b_s^{(m)}(t_1)\right\|_{\ell^1}
&\leq \left\|\Phi(t_1)\right\|_{1}\left(\|\Psi(s)\|_1\left\|\phi^{(m)}\right\|_{\ell^1}+2 \int_{s}^{t_1}\|\Psi(\tau)\|_1\left\|\ba^{(\bm{s})}(\tau)\right\|_{\ell^1}\left\|b^{(\infty)}_s(\tau)\right\|_{\ell^1} d \tau\right)\\
&\leq \left\|\Phi(t_1)\right\|_{1}\left(\|\Psi(s)\|_1\left\|\phi^{(m)}\right\|_{\ell^1}+2h_1 \sup_{\tau\in J_1}\|\Psi(\tau)\|_1\left\|\ba^{(\bm{s})}\right\|_{X}\left\|b^{(\infty)}_s\right\|_{X}\right)\\
&\leq \left\|\Phi(t_1)\right\|_{1}\Bigg\{\|\Psi(s)\|_1\left\|\phi^{(m)}\right\|_{\ell^1}+2h_1 \sup_{\tau\in J_1}\|\Psi(\tau)\|_1\left\|\ba^{(\bm{s})}\right\|_{X}\cdot\\
&\hphantom{\leq}\quad\left[2W_mW_\infty\|\ba^{(\bm{s})}\|_X\kappa^{-1}\left\|\phi^{(m)}\right\|_{\ell^1}+\left(W_{\infty}^{\sup}+4 W_mW_\infty^{2}\|\ba^{(\bm{s})}\|_{X}^{2}\kappa^{-1}\right)\left\|\phi^{(\infty)}\right\|_{\ell^{1}}\right]\Bigg\}.
\end{align*}
Taking the supremum norm with respect to $s$, it follows
\begin{align}\label{eq:bs_finite}\nonumber
\sup_{s\in J_1}\left\|b_s^{(m)}(t_1)\right\|_{\ell^1}
&\leq \left\|\Phi(t_1)\right\|_{1}\sup_{s\in J_1}\|\Psi(s)\|_1\left(1+4h_1W_mW_\infty\left\|\ba^{(\bm{s})}\right\|^2_{X}\kappa^{-1}\right)\left\|\phi^{(m)}\right\|_{\ell^1}\\
&\hphantom{\leq}\quad +2\left\|\Phi(t_1)\right\|_{1}h_1 \sup_{s\in J_1}\|\Psi(s)\|_1 \left\|\ba^{(\bm{s})}\right\|_{X}\left(W_{\infty}^{\sup}+4 W_mW_\infty^{2}\|\ba^{(\bm{s})}\|_{X}^{2}\kappa^{-1}\right)\left\|\phi^{(\infty)}\right\|_{\ell^{1}}.
\end{align}
Let
\[
	W_{J_1}^{(1,1)}\bydef \left\|\Phi(t_1)\right\|_{1}\sup_{s\in J_1}\|\Psi(s)\|_1\left(1+4h_1W_mW_\infty\left\|\ba^{(\bm{s})}\right\|^2_{X}\kappa^{-1}\right)
\]
and
\[
	W_{J_1}^{(1,2)}\bydef2\left\|\Phi(t_1)\right\|_{1}h_1 \sup_{s\in J_1}\|\Psi(s)\|_1 \left\|\ba^{(\bm{s})}\right\|_{X}\left(W_{\infty}^{\sup}+4 W_mW_\infty^{2}\|\ba^{(\bm{s})}\|_{X}^{2}\kappa^{-1}\right).
\]

Next, the tail mode is given by plugging $t=t_1$ in \eqref{eq:b^{infty}_integral_equation}
\[
b^{(\infty)}_s(t_1)=\bU^{(\infty)}(t_1, s) \phi^{(\infty)}+2 \expig \int_{s}^{t_1} \bU^{(\infty)}(t_1, \tau) ({\rm Id} -\pi^{(m)})\left(\ba(\tau)*b^{(m)}_s(\tau)\right) d \tau.
\]
Using the bound \eqref{eq:ev_op_estimate2}, the property \eqref{eq:banach_algebra} and \eqref{eq:b_m_norm}, we also have
\begin{align*}
&\left\|b^{(\infty)}_s(t_1)\right\|_{\ell^1}\\
&\le W^{(\infty)}(t_1,s)\left\|\phi^{(\infty)}\right\|_{\ell^1}+2\int_{s}^{t_1} W^{(\infty)}(t_1, \tau)\left\|\ba^{(\bm{s})}(\tau)\right\|_{\ell^1}\left\|b^{(m)}_s(\tau)\right\|_{\ell^1} d \tau\\
&\le W^{(\infty)}(t_1,s)\left\|\phi^{(\infty)}\right\|_{\ell^1}+2W_{\infty}\left\|\ba^{(\bm{s})}\right\|_{X}\left\|b^{(m)}_s\right\|_{X}\\
&\le W^{(\infty)}(t_1,s)\left\|\phi^{(\infty)}\right\|_{\ell^1}+2W_{\infty}\left\|\ba^{(\bm{s})}\right\|_{X}\left(W_m\kappa^{-1}\left\|\phi^{(m)}\right\|_{\ell^1}+2 W_mW_\infty\|\ba^{(\bm{s})}\|_{X}\kappa^{-1}\left\|\phi^{(\infty)}\right\|_{\ell^1}\right).
\end{align*}
Taking the supremum norm with respect to $s$, it follows
\begin{align}\label{eq:bs_tail}
\sup_{s\in J_1}\left\|b_s^{(\infty)}(t_1)\right\|_{\ell^1}\leq \left(2W_mW_{\infty}\left\|\ba^{(\bm{s})}\right\|_{X}\kappa^{-1} \right)\left\|\phi^{(m)}\right\|_{\ell^1}+\left(W_{\infty}^{\sup}+4W_mW_{\infty}^2\left\|\ba^{(\bm{s})}\right\|_{X}^2\kappa^{-1}\right)\left\|\phi^{(\infty)}\right\|_{\ell^{1}}.
\end{align}
Moreover, let us define
\[
	W_{J_1}^{(2,1)}\bydef 2W_mW_{\infty}\left\|\ba^{(\bm{s})}\right\|_{X}\kappa^{-1}\quad
	\mbox{and}\quad
	W_{J_1}^{(2,2)}\bydef W_{\infty}^{\sup}+4W_mW_{\infty}^2\left\|\ba^{(\bm{s})}\right\|_{X}^2\kappa^{-1}.
\]

Summing up \eqref{eq:bs_finite} and \eqref{eq:bs_tail}, we get the $W_{J_1}$ bound
\begin{align*}
\sup_{s\in J_1}\|U(t_1,s)\phi\|_{\ell^1}
&\le\sup_{s\in J_1}\left\|b_s^{(m)}(t_1)\right\|_{\ell^1}+\sup_{s\in J_1}\left\|b_s^{(\infty)}(t_1)\right\|_{\ell^1}\\
&=\left\|\begin{bmatrix}
\sup_{s\in J_1}\left\|b_s^{(m)}(t_1)\right\|_{\ell^1}\\[2mm]
\sup_{s\in J_1}\left\|b_s^{(\infty)}(t_1)\right\|_{\ell^1}
\end{bmatrix}\right\|_1\\
&\le\left\|\begin{bmatrix}
W_{J_1}^{(1,1)} & W_{J_1}^{(1,2)}\\[1mm]
W_{J_1}^{(2,1)} & W_{J_1}^{(2,2)}
\end{bmatrix}\begin{bmatrix}
\|\phi^{(m)}\|_{\ell^1}\\\|\phi^{(\infty)}\|_{\ell^1}
\end{bmatrix}\right\|_1\le W_{J_1}\|\phi\|_{\ell^1},\quad\forall\phi\in\ell^1,
\end{align*}
where
\[
	W_{J_1}\bydef\left\|\begin{bmatrix}
	W_{J_1}^{(1,1)} & W_{J_1}^{(1,2)}\\[1mm]
	W_{J_1}^{(2,1)} & W_{J_1}^{(2,2)}
	\end{bmatrix}\right\|_1.
\]

Setting $s=t_1$, \eqref{eq:bs_finite} and \eqref{eq:bs_tail} are changed by
\begin{align}\label{eq:bt1_finite}\nonumber
\left\|b_{t_0}^{(m)}(t_1)\right\|_{\ell^1}
&\leq \left\|\Phi(t_1)\right\|_{1}\left(1+4h_1\sup_{s\in J_1}\|\Psi(s)\|_1W_mW_\infty\left\|\ba^{(\bm{s})}\right\|^2_{X}\kappa^{-1}\right)\left\|\phi^{(m)}\right\|_{\ell^1}\\
&\hphantom{\leq}\quad +2\left\|\Phi(t_1)\right\|_{1} h_1\left\|\ba^{(\bm{s})}\right\|_{X}\left(W_{\infty}^{\sup}+4 W_mW_\infty^{2}\|\ba^{(\bm{s})}\|_{X}^{2}\kappa^{-1}\right)\left\|\phi^{(\infty)}\right\|_{\ell^{1}}.
\end{align}
and
\begin{align}\label{eq:bt1_tail}
\left\|b_{t_0}^{(\infty)}(t_1)\right\|_{\ell^1}\leq \left(2W_mW_{\infty}\left\|\ba^{(\bm{s})}\right\|_{X}\kappa^{-1} \right)\left\|\phi^{(m)}\right\|_{\ell^1}+\left(W^{(\infty)}(t_1,t_0)+4W_mW_{\infty}^2\left\|\ba^{(\bm{s})}\right\|_{X}^2\kappa^{-1}\right)\left\|\phi^{(\infty)}\right\|_{\ell^{1}},
\end{align}
respectively.
The analogous discussion yields the $W_{t_1}$ bound by
\[
W_{t_1}\bydef\left\|\begin{bmatrix}
W_{t_1}^{(1,1)} & W_{t_1}^{(1,2)}\\[1mm]
W_{t_1}^{(2,1)} & W_{t_1}^{(2,2)}
\end{bmatrix}\right\|_1,
\]
where
\[
	W_{t_1}^{(1,1)}\bydef \left\|\Phi(t_1)\right\|_{1}\left(1+4h_1\sup_{s\in J_1}\|\Psi(s)\|_1W_mW_\infty\left\|\ba^{(\bm{s})}\right\|^2_{X}\kappa^{-1}\right),
\]
\[
	W_{t_1}^{(1,2)}\bydef 2\left\|\Phi(t_1)\right\|_{1} h_1\left\|\ba^{(\bm{s})}\right\|_{X}\left(W_{\infty}^{\sup}+4 W_mW_\infty^{2}\|\ba^{(\bm{s})}\|_{X}^{2}\kappa^{-1}\right),
\]
\[
	W_{t_1}^{(2,1)}\bydef 2W_mW_{\infty}\left\|\ba^{(\bm{s})}\right\|_{X}\kappa^{-1},\quad\mbox{and}\quad W_{t_1}^{(2,2)}\bydef W^{(\infty)}(t_1,t_0)+4W_mW_{\infty}^2\left\|\ba^{(\bm{s})}\right\|_{X}^2\kappa^{-1}.
\]

Consequently, the point-wise error is estimated by
\begin{align*}
	\|z(t_1)\|_{\ell^1}&\le\left\|U(t_1,t_0)z(t_0)\right\|_{\ell^1}+\int_{t_0}^{t_1}\left\|U(t_1,s)g(s)\right\|_{\ell^1}ds\\
	&\le W_{t_1}\varepsilon_0+W_{J_1}h_1\left(2\varrho_1+\delta_1\right)= \varepsilon_{1}.
\end{align*}

For the next time step, updating $\varepsilon\equiv\varepsilon_{1}$, we apply Theorem \ref{thm:main_theorem} on $J_2$ and derive the point-wise error $\varepsilon_2$ at the endpoint recursively.
By repeating the time stepping, we have $\varepsilon_{i}= W_{t_i}\varepsilon_{i-1}+W_{J_i}h_i\left(2\varrho_i+\delta_i\right)$ ($i=1,2,\dots,K$).

\begin{rem}
	The point-wise error $\varepsilon_{i}$ can be smaller than the previous one, i.e., $\varepsilon_{i}\le\varepsilon_{i-1}$ holds when $W_{t_i}<1$ with sufficiently small $\varrho_i$ and $\delta_i$.
	This implies that the diffusion property of solution makes the point-wise error small.
	
	Furthermore, we note that the tiny numerical error may include at the end point.
	That is, the value of approximate solution at $t=t_1$ in $J_1$, say $\ba^{J_1}(t_1)$, and that in $J_2$, say $\ba^{J_2}(t_1)$, may be different.
	It is because the Chebyshev polynomial approximates a function globally in each time interval.
	In such a case, we should take care of the numerical error by using the following form:
	\begin{align*}
	\epsilon_1 &\bydef\left\|\ba^{J_1}(t_1)-\ba^{J_2}(t_1)\right\|_{\ell^1}\\
	&= \sum_{|k| \le N}\left|\ba^{J_1}_{0,k}+2\ba^{J_1}_{1,k}+2\ba^{J_1}_{2,k}+\dots+2\ba^{J_1}_{n-1,k} - (\ba^{J_2}_{0,k}-2\ba^{J_2}_{1,k}+2\ba^{J_2}_{2,k}-\dots+(-1)^{n-1}2\ba^{J_2}_{n-1,k})\right|.
	\end{align*}
	Hence, we should add $\epsilon_1$ in the point-wise error, i.e., $\varepsilon_{1}= W_{t_1}\varepsilon_0+W_{J_1}h_1\left(2\varrho_1+\delta_1\right) + \epsilon_1$.
\end{rem}

%%%%%%%%%%%%%%%%
\section{Global existence in time} \label{sec:center-stable-manifold}
%%%%%%%%%%%%%%%%
After several time steppings, we prove global existence of the solution.
That is related to calculation of a center-stable manifold arising from the Cauchy problem \eqref{eq:CGL_ode}.
\subsection{Calculating part of a center-stable manifold}
Starting from the PDE \eqref{eq:CGL}
%\[
%u_t = e^{i \theta} (  u_{xx} + u^2 ) ,
%\]
we consider the system of differential equations in $ \ell^1$ given by \eqref{eq:CGL_ode}
%\begin{align}
%\dot{a}_k = - e^{i \theta } k^2 a_k + e^{i \theta } a *_k a ,
%\label{eq:OriginalDE}
%\end{align}
%where we define the discrete convolution 
%\[
%a *_k b := \sum_{\substack{k_1 + k_2 = k \\ k_1,k_2 \in \Z}} a_{k_1} b_{k_2}.
%\]
%This product makes $ \ell^1$ a Banach algebra, that is to say $ | a * b | \leq | a | \cdot |b|$ for $ a,b\in \ell^1$. 
In this section, we use the Lyapunov-Perron method for computing a foliation of portion of the center-stable manifold of the equilibrium at $ a \equiv 0$.  
A good reference for this method in ODEs is 
\cite{chicone2006ODE}, and for PDEs see \cite{sell2002dynamics}.  
In \cite{BergPREPApproximating} this method is applied to give computer-assisted proofs of the stable manifold theorem in the Swift-Hohenberg PDE.

Let us define subspaces 
\begin{align*}
X_c &\bydef \{a \in \ell^1 | a_k = 0 \; \forall k \neq 0 \}   \\ %\label{eq:DE_Center}
X_s &\bydef \{a \in \ell^1 | a_0 = 0 \} 	.
%	\label{eq:DE_Stable} 
\end{align*}
We may then rewrite \eqref{eq:CGL_ode} as the following system: 
\begin{align}
\label{eq:CenterProjectedEquation}	
\dot{\xx}_c & = \cN_c( \xx_c , \xx_s )   \\ 
\dot{\xx}_s & = \fL \xx_s  + \cN_s( \xx_c , \xx_s ) ,
\label{eq:StableProjectedEquation}	
\end{align}
where for $ a = ( \xx_c , \xx_s ) $ we define:
\begin{align*}
(\fL a)_k &\bydef - e^{i \theta } k^2\omega^2 a_k ,
&
\cN_c(a_c,a_s) &\bydef e^{i \theta } \sum_{k=0}^\infty a_k a_{-k},
&
\left(\cN_s(a_c,a_s)\right)_k &\bydef e^{i \theta } \sum_{\substack{k_1 + k_2 = k \\ k_1,k_2 \in \Z}} a_{k_1} a_{k_2} .
\end{align*}
Note $ \|e^{\fL t}\|_{B(\ell^1)} \leq e^{- \omega^2 \cos \theta \;t}$, $|\cN_c(\xx_c,\xx_s)| \leq |\xx_c|^2 + \|\xx_s\|_{\ell^1}^2$ and $\|\cN_s(\xx_c,\xx_s)\|_{\ell^1} \leq 2 |\xx_c| \|\xx_s\|_{\ell^1} + \|\xx_s\|_{\ell^1}^2$.
To abbreviate, let us define $ \mu = \omega^2\cos \theta$.

For the equilibrium at zero, the center manifold is precisely $ X_c$. 
We can solve \eqref{eq:CenterProjectedEquation} restricted to this subspace: 
\begin{align}
\dot{\xx}_c = e^{i \theta } \xx_c^2 , \label{eq:CenterRestrictedDE}	
\end{align}
using the fact that the differential equation is separable.  
For an initial condition  $ \phi \in \C $ then the solution of  \eqref{eq:CenterRestrictedDE} is given by: 
\[
\Phi(t,\phi) \bydef \frac{\phi}{1 - \phi t e^{i \theta}}.
\]
Note that if $\mathrm{Re}( \phi e^{i \theta}) >0$ and  $\mathrm{Im}( \phi e^{i \theta}) =0$, then the solution $\Phi(t,\phi)$ blows up in finite time.

For  $ r_c , r_s >0$ let us define the following sets 
\begin{align*}
B_c(r_c) &= \{ \xx_c \in X_c : | \xx_c | \leq r_c, \mathrm{Re}(e^{i \theta} \xx_c)  \leq 0  \} \\
B_s(r_s) &= \{ \xx_s \in X_s : \| \xx_s \|_{\ell^1} \leq r_s \}.
\end{align*}
Note that if $ \phi \in B_c(r_c)$ then  $ \Phi(t) \in B_c(r_c)$ for all $ t \geq 0 $, and additionally $ | \Phi(t,\phi)| \leq r_c$. 
For a fixed $\rho \in \R_+$ we define the following set of functions:
\begin{align*}
\cB &= \left\{ \alpha \in \mbox{Lip}\left( B_c(r_c) \times B_s(r_s) , X_c \right) : \alpha (\xx_c,0)=\xx_c, \; | \alpha (\xx_c,\xx_1) - \alpha (\xx_c,\xx_2)| \leq \rho | \xx_1 - \xx_2|  \right\}.
\end{align*}

Continuing with the Lyapunov-Perron method, for a fixed $ \alpha \in \cB, \phi \in B_c(r_c), \xi \in B_s(r_s)$, we define $ x(t,\phi,\xi,\alpha)
$ as a solution with initial conditions $(\phi , \xi)$ of the  differential equation below:
\begin{align} 
\dot{\xx}_s &= \fL \xx_s + \cN_s \left( \alpha ( \Phi(t,\phi) ,\xx_s) , \xx_s \right) . 
\label{eq:ProjectedSystem}
\end{align}

Now we define the Lyapunov-Perron Operator  for $ \alpha \in \cB$ as follows:
\begin{align}
\Psi [\alpha ] ( \phi , \xi) &= - \int_0^\infty \cN_c \left( \alpha \big( \Phi(t,\phi)  , x(t,\phi,\xi,\alpha) \big) ,x(t,\phi,\xi , \alpha ) \right) dt .
\end{align}
In Proposition \ref{prop:Endomorphism}, we show that $ \Psi : \cB\to \cB$ is a well defined operator. 
As is the case for  Lyapunov-Perron operators, from the variation of constants formula, if $ \Psi[\alpha]=\alpha$, then for all $ (\phi,\xi) \in B_c\times B_s$, the trajectory $\alpha( \Phi(t,\phi),x(t,\phi,\xi,\alpha)
)$ satisfies  \eqref{eq:CenterProjectedEquation}. 
Hence $(\alpha( \Phi(t,\phi),x(t,\phi,\xi,\alpha)
), x(t,\phi,\xi,\alpha)
)$ satisfies our original equation \eqref{eq:CGL_ode}. 

Since $ \Phi(t,\phi)$ limits to zero, such a fixed point $ \alpha = \Psi[\alpha]$ gives us a foliation of the center stable manifold over $ B_c(r)$. 
By Corollary \ref{prop:AlphaImage} we obtain an explicit neighborhood within which all points limit to the zero equilibrium. 
To prove the existence of a fixed point to our Lyapunov-Perron operator, and obtain explicit bounds, we prove the following theorem.  
\begin{thm}
	\label{prop:MainGlobalTheorem}
	Fix $ r_c ,r_s, \rho >0$ and define the following constants
	\begin{align*}
	\delta_1 &\bydef 2 r_c+ (1+ 2\rho ) r_s, \\
	\delta_2 &\bydef 2 r_c+ 2(1+ \rho ) r_s, \\
	\delta_3 &\bydef  2 (\rho (r_c + \rho r_s) +   r_s ),\\
	\delta_4 &\bydef 2 (r_c + 2 \rho r_s + r_s),\\
	\lambda &\bydef \frac{ 2 (  \rho(r_c+ \rho r_s) + r_s )    2 r_s	}{(\mu-\delta_1)(\mu-\delta_4)}
	+ \frac{2( r_c + \rho r_s)    }{\mu-\delta_1} .
	\end{align*}
	If the following inequalities are satisfied 
	\begin{align}
	\delta_1, \delta_2 ,\delta_4&< \mu, &
	\frac{\delta_3}{\mu - \delta_2} & < \rho ,
	&
	\lambda&<1,
	\end{align}
	then there is a unique map $ \alpha \in \cB$ such that $ \Psi[\alpha] = \alpha$. 
	
\end{thm}

The proof is given in Section \ref{sec:GlobalExistenceProof}.  

\begin{cor} \label{prop:TrackingEstimate}
	Suppose the hypothesis of Theorem \ref{prop:MainGlobalTheorem} is satisfied and $ \hat{\alpha} \in \cB$ is such that $ \Psi(\hat{\alpha}) = \hat{\alpha}$.  
	Fix  $(\phi,\xi) \in B_c(r_c) \times B_s(r_s)$. 
	Let $ X(t)$ denote the solution to \eqref{eq:CGL_ode} with initial conditions $( \hat{\alpha}(\phi,\xi) ,\xi)$.  
	Then $\lim_{t\to \infty} X(t) =0$ and moreover
	\[
	\| X(t) - (\Phi(t,\phi),0) \| \leq  e^{-(\mu-\delta_1)t}  (1 + \rho ) | \xi|.
	\]
	
\end{cor}
The proof largely follows from Proposition \ref{prop:NormBound}.

\begin{cor}
	\label{prop:AlphaImage}
	Fix  $\rho, r_c,r_s >0$ and write $ B_c = B_c(r_c)$ and $B_s= B_s(r_s)$. Define the set $U \subseteq B_c  \times B_s  $ as follows:
	\[
	U \bydef \{ (\xx_c,\xx_s) \in  B_c \times B_s  : \rho \|\xx_s\|_{\ell^1}  < dist(\xx_c, \partial B_c) \},
	\]
	where $\partial B_c$ denotes the boundary of $ B_c$.  
	Suppose the hypothesis of Theorem \ref{prop:MainGlobalTheorem} is satisfied and there exists some $\alpha \in \cB$ such that $ \Psi[\alpha ] = \alpha$. Then  $U$ is contained in the $\alpha$-skew image of $B_c \times B_s$; that is to say   $U \subseteq  \{  ( \alpha(\xx_c,\xx_s),\xx_s) : (\xx_c,\xx_s) \in B_c \times B_s\}$. 
	Furthermore, $\lim_{t \to \infty} a(t) =0$ under the differential equation  \eqref{eq:CGL_ode} for all $ a(0) \in U$. 
\end{cor}

\begin{proof}
	Fix $( \phi_0 , \xi) \in U \subseteq B_c \times B_s$. 
	We wish to show there exists some $ \phi \in B_c$ such that $ \alpha ( \phi, \xi ) = \phi_0$. Let us define 
	\[
	\tau_0 = \sup \{  \tau \in [0,1] :  \emptyset \neq \alpha^{-1} ( \phi_0 , \tau \xi ) \in \mbox{int}(B_c)  \} .
	\] 
%%%%	$\phi_0  \notin \alpha(B_c \times B_s )$. 
	To show  $U$ is contained in the $\alpha$-skew image of $B_c \times B_s$, it suffices to show that $ \tau_0=1$. 
	Since $ \alpha(\phi,0) = \phi$ for all $ \phi \in B_c$, then $ \tau_0 >0$.
	
	From Corollary \ref{prop:TrackingEstimate} it follows that for each fixed  $ \xi \in B_s $ the map $ \alpha ( \cdot , \xi) : B_c \to B_c$ is injective. By the invariance of domain theorem, it follows that $\alpha ( \cdot , \xi)$ is a homeomorphism from $ \mbox{int} ( B_c )$ onto its image. 
%	Hence $ \Gamma = \{ \alpha^{-1} ( \phi_0 , \tau \xi ) \in B_c: \tau \in ( 0, \tau_0) \}$ is a continuous curve, and we may write $ \Gamma : (0,\tau_0) \to \mbox{int}(B_c)$ as a function.  
We may then define a homotopy $H: \mbox{int}(B_c) \times [0,1] \to \C \times [0,1]$ by $ H(\phi,\tau) = ( \alpha(\phi,\tau \xi),\tau)$, which is also a homeomorphism from $ \mbox{int} (B_c) \times [0,1]$ onto its image. 	
Hence $ H^{-1}\big(\phi_0,(0,\tau_0)\big)= \{ \big ( \alpha^{-1} ( \phi_0 , \tau \xi ) , \tau \big) \in \mbox{int}(B_c) \times [0,1]: \tau \in ( 0, \tau_0) \}$ is a continuous curve, and we may define a function $ \Gamma : (0,\tau_0) \to \mbox{int}(B_c)$ by  $ \Gamma(\tau) \bydef H^{-1}(\phi_0,\tau \xi )$. 
%	Hence $ \alpha ( \Gamma(\tau ), \tau \xi) = \phi_0 $ for $ \tau \in (0,\tau_0)$. 
Fix $ \phi_1 \bydef \lim_{\tau \to \tau_0} \Gamma(\tau)    \in B_c$. 
	
%	By way of contradiction, let us suppose that $ \tau_0 < 1$.  Necessarily $ \phi_0 $ is in the  boundary of $ \alpha( \cdot , \tau_0 \xi )$ image of $ B_c$. 
%	Write $ \phi_1 \bydef \lim_{\tau \to \tau_0} \Gamma(\tau) $.
%	Since homeomorphisms map open sets to open sets, it follows that $ \phi_1 \in \partial B_c$. 
%	But this is contradicted by our definition of $U$, as seen by the following calculation: 

		By way of contradiction, let us suppose that $ \tau_0 < 1$. 
	If $\phi_1 \in \mbox{int}( B_c)$, then $H$ maps an open set about $ (\phi_1, \tau_0)$ onto an open set about $(\phi_0,\tau_0)$. 
	Hence the initial choice of $ \tau_0$ as the supremum above was incorrect, and $\tau_0$ could have been made even larger.   	 
	Otherwise suppose  $ \phi_1 \in \partial B_c$. 
	But this is contradicted by our definition of $U$, as seen by the following calculation: 
	\begin{align*}
	0= | \phi_0 - \alpha(\phi_1, \tau_0 \xi) |\geq 
	\inf_{\phi \in \partial B_c} \left| \phi_0 - \alpha(\phi,\tau_0  \xi)  \right| 
	\geq  \mbox {dist} ( \phi_0 , \partial B_c ) - \rho \tau_0 \|\xi\|_{\ell^1} 
	>0.
	\end{align*}
%	Hence, $U$ is contained in the $\alpha$-skew image of $B_c \times B_s$.
	Hence $ \tau_0 =1$, and there is some $ \phi \in B_c$ such that $ \alpha(\phi,\xi) = \phi_0$. 
	Thereby $U$ is contained in the $\alpha$-skew image of $B_c \times B_s$.
	
	Furthermore, if $a(0) \in U$, then $a(0)$ is in the $\alpha$-skew image of $B_c \times B_s$. 
	That is to say $a(0)$ is in the center stable manifold of the equilibrium $0$; $\lim_{t \to \infty} a(t) =0$.

\end{proof}

\begin{rem}
	\label{rem:ChooseRho}
	By isolating the $ \rho$ terms in the inequality $\frac{\delta_3}{\mu - \delta_2} $, we obtain the following: 
	\begin{align*}
	\rho^2  + \rho \left( \frac{- \mu +  4 r_c + 2 r_s}{4 r_s} \right)   +2 r_s  &<0 .
	\end{align*}
	Hence, we need at a minimum $ \mu > 4 r_c +2 r_s$ in order to satisfy  the hypotheses of Theorem \ref{prop:MainGlobalTheorem}. 
	Using the quadratic formula, we obtain a necessary inequality,
	\begin{align*}
	0 < \frac{ \mu -  4 r_c - 2 r_s }{8 r_s} - \frac{1}{2}   \sqrt{\left(\frac{\mu -  4 r_c - 2 r_s  }{4 r_s} \right)^2- 8 r_s }  < \rho.
	\end{align*}
	We can use this to define a nearly optimal $\rho$ in terms of $ r_c$, $ r_s$ and $ \mu$. 
\end{rem}

\begin{rem}
	For a given $ \theta$, there is something to be said for how to select $r_c$ and $r_s$ so that the hypothesis of Theorem \ref{prop:MainGlobalTheorem} is satisfied. 
	For the best bounds, $\rho$ should be taken as small as possible, and nearly optimal bounds are explicitly given in Remark \ref{rem:ChooseRho}. Furthermore, we necessarily need to take $  r_c < \tfrac{1}{4} \omega^2\cos \theta$. 
	Below are some constants which will satisfy Theorem \ref{prop:MainGlobalTheorem} and try to maximize $r_c$. 
	%	\begin{center}
	\[
	\begin{array}{c|ccccc}
	\theta 	& \tfrac{1}{4} \omega^2 \cos \theta &r_c 	& r_s 	&\rho& \lambda  \\
	\hline 
	0 		&9.87	& 9.77 & 0.01 	& 0.06 & 0.99\\
	\pi/8 	&9.12	& 9.02& 0.01 	& 0.06 & 0.98\\
	\pi / 4	&6.98	& 6.88 & 0.01 	& 0.06 & 0.98\\
	3 \pi /8&3.78	& 3.68 & 0.01 	& 0.06 & 0.96
	\end{array}
	\]
\end{rem}
%\corrc $\mu$ is changed from $\cos\theta$ to $\omega^2\cos\theta$. So the list of constants should be modified. I will be improved from the previous one.  \hfill Akitoshi<<>> 
%
%\corrc I have updated the table.  \hfill - Jonathan <<>>

%%%%%%%%%%%%%%%%%%%%%
%%%%%%%%%%%%%%%%%%%%
%%%%%%%%%%%%%%%%%%%%
\subsection{Proof of Theorem \ref{prop:MainGlobalTheorem}}
\label{sec:GlobalExistenceProof}
The proof of Theorem \ref{prop:MainGlobalTheorem} is organized as follows. 
First in Propositions \ref{prop:NormBound}--\ref{prop:ExponentialTrackingCenter} we derive bounds on the norm of solutions to \eqref{eq:ProjectedSystem} and their dependence on initial conditions. 
In Proposition \ref{prop:Endomorphism}, we show that $\Psi : \cB\to \cB$ is a well defined endomorphism. 
In Definition \ref{def:WeakerNorm}, we define a norm $ \| \cdot \| _{\cE}$ on $\cB$. 
In Proposition \ref{prop:DifferentMapTracking}, we show that if two maps $ \alpha , \beta \in \cB$ are close in the $\| \cdot \|_\cE$ norm, then their solutions to \eqref{eq:ProjectedSystem}, taken with the same initial conditions but different maps $\alpha$ and $\beta$, will also remain close. 
Finally, we finish the proof of Theorem \ref{prop:MainGlobalTheorem} by bounding the Lipschitz constant of $\Psi$ in this new norm. If this Lipschitz constant is less than $1$, then $ \Psi$ is a contraction mapping, and hence it will have a unique fixed point. 
Throughout, we will often make use of the following Gr\"onwall type lemma below~\cite{BergPREPApproximating}. 
\begin{lem}
	\label{lem:GronwallType}
	Fix constants $c_0,c_1,c_2 ,c_3 \in \R$ with $c_1,c_2 ,c_3\geq 0$. 
	Define $ \mu_0 = c_0 + c_2$ and fix $ \mu_1 \in \R$ such that $\mu_0 \neq \mu_1$.   If we have the inequality: 
	\begin{equation*}
	e^{-c_0 t} u_0 (t) \leq \left( c_1 + \int_0^t e^{-c_0 s} e^{\mu_1 s} c_3  \, ds \right) + \int_0^t c_2 e^{-c_0 s} u_0(s) ds,
	\end{equation*}
	then:  
	\begin{equation}
	u_0(t) \leq c_1 e^{\mu_0 t }  +  c_3 \frac{e^{\mu_1 t} - e^{\mu_0 t}}{\mu_1 -\mu_0}  .
	%	\label{eq:Fundamental}
	\end{equation}
\end{lem}

\begin{prop}
	\label{prop:NormBound}
	Fix $ \alpha \in \cB, \phi \in B_c(r_c), \xi \in B_s(r_s)$, and  define 
	\begin{align*}
	\delta_1 = 2 r_c+ (1+ 2\rho ) r_s.
	\end{align*}
	If $ \delta_1 <\mu $ then: 
	\begin{align*}
	\|x(t,\phi,\xi,\alpha) \|_{\ell^1} \leq e^{-(\mu-\delta_1) t} \| \xi\|_{\ell^1} .
	\end{align*}
	Furthermore, $x(t,\phi,\xi,\alpha) \in B_s(r_s)$ for all $ t \geq 0$. 
	
\end{prop}
\begin{proof}
	Let us abbreviate $x_s = x_s(t) = x(t,\phi,\xi,\alpha)
	$. 
	By variation of constants we have 
	\begin{align*}
	\|x_s(t )\|_{\ell^1}  &\leq e^{-\mu t} |\xi| + \int_0^t e^{-\mu (t-\tau)} \|\cN_s\left( \alpha ( \Phi(\tau,\phi) ,x_s(\tau)) , x_s(\tau) \right) \|_{\ell^1} d \tau .
	\end{align*}
	Since $|  \Phi(t,\phi)| \leq r_c$, then it follows that $ | \alpha ( \Phi(t,\phi) ,x_s)| \leq r_c + \rho r_s$. Hence   
	\begin{align*}
	\|\cN_s\left( \alpha ( \Phi(t,\phi) ,x_s) , x_s \right) \|_{\ell^1} & \leq 2(r_c + \rho r_s) \|x_s\|_{\ell^1}  + \|x_s\|_{\ell^1}^2 \\
	& \leq \delta_1 \|x_s\|_{\ell^1}.
	\end{align*}
	Returning to our variation of constants formula, we have:
	\begin{align*}
	e^{\mu t} x_s(t )  &\leq \|\xi\|_{\ell^1} + \int_0^t e^{\mu \tau} \delta_1 \|x_s(\tau)\|_{\ell^1} d \tau .
	\end{align*}
	The proposition follows from Gr\"onwall's inequality. 
\end{proof}

\begin{prop} \label{prop:ExponentialTrackingStable}
	Fix $ \alpha \in \cB, \phi \in B_c(r_c)$, and $ \xi, \zeta \in B_s(r_s)$, and  define 
	\begin{align*}
	\delta_2 = 2 r_c+ 2(1+ \rho ) r_s.
	\end{align*}
	If $ \delta_2 <\mu $ then 
	\begin{align*}
	\|x(t,\phi,\xi,\alpha) -x(t,\phi, \zeta , \alpha) \|_{\ell^1} \leq e^{-(\mu -\delta_2) t} \| \xi - \zeta \|_{\ell^1} .
	\end{align*}

\end{prop}
\begin{proof}
	Let us abbreviate $x_s = x_s(t) = x(t,\phi,\xi,\alpha)
	$ and $z_s = z_s(t) = x(t,\phi, \zeta , \alpha)$. Additionally, let us abbreviate $ x_c = \alpha ( \Phi(t,\phi),x_s(t)) $ and $ z_c = \alpha ( \Phi(t,\phi),z_s(t)) $.
	By variation of constants we have 
	\begin{align*}
	\|x_s(t ) - z_s(t)\|_{\ell^1} &\leq e^{- \mu t} \|\xi - \zeta\|_{\ell^1} + \int_0^t e^{-\mu (t-\tau)} \|\cN_s\left( x_c, x_s \right) -\cN_s\left( z_c , z_s \right)\|_{\ell^1} d \tau .
	\end{align*}
	We calculate
	\begin{align*}
	\|\cN_s\left( x_c, x_s \right) -\cN_s\left( z_c , z_s \right)\|_{\ell^1} & \leq  \|\cN_s\left( x_c, x_s \right) -\cN_s\left( x_c, z_s \right) \|_{\ell^1} + \|\cN_s\left( x_c, z_s \right)  -\cN_s\left( z_c , z_s \right)\|_{\ell^1} \\
	&\leq		2 | x_c| \| x_s - z_s\|_{\ell^1} +\|x_s^2 - z_s^2\|_{\ell^1} + 2 |x_c -z_c| \|z_s\|_{\ell^1} \\
	& \leq \|x_s -z_s\|_{\ell^1} (2 r_c + 2 r_s) + 2 r_s \rho \|x_s - z_s\|_{\ell^1} \\
	&= \delta_2 \|x_s - z_s\|_{\ell^1} .
	\end{align*}
	Returning to our variation of constants formula, we have:
	\begin{align*}
	e^{\mu t}\| x_s(t ) - z_s(t)\|_{\ell^1}  &\leq \|\xi\|_{\ell^1} + \int_0^t e^{\mu \tau} \delta_2 \|x_s(\tau) - z_s (\tau) \|_{\ell^1} d \tau . 
	\end{align*}
	The proposition follows from Gr\"onwall's inequality. 
\end{proof}

\begin{prop} \label{prop:ExponentialTrackingCenter}
	Fix $ \alpha \in \cB$, $\phi_1 , \phi_2 \in B_c(r_c)$ and $ \xi \in B_s (r_s)$.  Fix $ \epsilon >0$  such that $ | \alpha(\phi_1,\zeta) - \alpha(\phi_2 , \zeta)| < \epsilon |\phi_1 - \phi_2|$ for all $ \zeta \in B_s(r_s)$.  Then we have: 
	\[
	\| x(t, \phi_1 , \xi , \alpha) - x(t, \phi_2,\xi,\alpha)\|_{\ell^1} < 2 \epsilon r_s | \phi_1 - \phi_2 | 
	\frac{e^{-(\mu-\delta_1)t} -  e^{-(\mu-\delta_2)t} }{-(\mu-\delta_1) +(\mu-\delta_2)}
	.
	\]
\end{prop}
\begin{proof}
	Let us define $ x_s = x_s(t) = x(t, \phi_1 , \xi , \alpha) $ and $ w_s = w_s(t) = x(t, \phi_2, \xi , \alpha)$. Let us define $ x_c = x_c(t) = \alpha ( \Phi(t,\phi_1),x_s(t))$ and define $ w_c = w_c(t) = \alpha( \Phi(t, \phi_2 ) , w_s(t))$. 
	
	We compute and find 
	\begin{align*}
	\| \cN_s ( x_c , x_s ) - \cN_s(w_c,w_s)\| &\leq 2 | x_c| \|x_s - w_s\|_{\ell^1} + \| x_s^2 - w_s^2\|_{\ell^1} + 2 |x_c - w_c| \| w_s\|_{\ell^1}  \\
	& \leq \|x_s - w_s \|_{\ell^1} ( 2  r _c + 2 r_s) + 2 \| w_s \|_{\ell^1} |x_c - w_c| .
	\end{align*}
	We estimate $| x_c - w_c |$ below: 
	\begin{align*}
	| x_c - w_c | &\leq | \alpha( \Phi(t,\phi_1),x_s(t) ) - \alpha( \Phi(t,\phi_1) , w_s(t) ) | + | \alpha( \Phi(t,\phi_1),w_s(t) ) - \alpha( \Phi(t,\phi_2) , w_s(t) ) | \\
	&\leq \rho | x_s - w_s | + \epsilon | \Phi ( t, \phi_1 ) - \Phi (t, \phi_2) | .
	\end{align*}
	By our assumption  $ \mbox{Re}(e^{i \theta} \phi_1)  \leq 0 $ it follows that $ |1- \phi_1 t e^{i\theta}|^2 \geq  1 + |\phi_1 t|^2$ for $ t \geq 0$. 
	We bound $|\Phi(t,\phi_1) - \Phi(t,\phi_2) | $ below: 
	\begin{align}
	|\Phi(t,\phi_1) - \Phi(t,\phi_2) | 
	&= 
	\frac{|\phi_1 - \phi_2|}{|1- \phi_1 t e^{i\theta} | \cdot |1 - \phi_2 t e^{i\theta} |} \nonumber \\ 
	&\leq 
	\frac{|\phi_1 - \phi_2|}{\sqrt{1+ |\phi_1 t |^2} \sqrt{1+ |\phi_2 t |^2}}  \label{eq:TightPhiDifferenceBound} \\
	&\leq |\phi_1- \phi_2| . \nonumber
	\end{align}
	Hence, we obtain the bound: 
	\begin{align*}
	\| \cN_s ( x_c , x_s ) - \cN_s(w_c,w_s)\|_{\ell^1} &\leq  2 \epsilon \|w_s\|_{\ell^1}  | \phi_1 - \phi_2|  + 2 (r_c + r_s + \rho r_s ) \| x_s - w_s\|_{\ell^1} .
	\end{align*}
	By variation of constants we obtain the following: 
	\begin{align*}
	\| x_s(t) - w_s(t) \|_{\ell^1} &=\int_0^t e^{-\mu (t-\tau)} \|\cN_s\left( x_c, x_s \right) -\cN_s\left( w_c , w_s \right)\|_{\ell^1} d \tau  \\
	&\leq \int_0^t e^{-\mu(t-\tau) } 2 \epsilon r_s e^{-(\mu-\delta_1)t} |\phi_1 - \phi_2| d\tau
	+
	\int_0^t e^{-\mu(t-\tau) } \delta_2 \|x_s(\tau) -  w_s(\tau)\|_{\ell^1}  d\tau .
	\\
	e^{\mu t } \| x_s(t) - w_s(t) \|_{\ell^1} &\leq \int_0^t e^{\mu \tau } 2 \epsilon r_s e^{-(\mu-\delta_1)t} |\phi_1 - \phi_2| d\tau
	+
	\int_0^t e^{\mu \tau } \delta_2 \|x_s(\tau) -  w_s(\tau)\|_{\ell^1}  d\tau .
	\end{align*}
	By Lemma \ref{lem:GronwallType}  we have: 
	\[
	\|x_s(t)  - w_s(t) \|_{\ell^1} \leq 2 \epsilon r_s | \phi_1 - \phi_2|  \frac{e^{-(\mu-\delta_1)t} -  e^{-(\mu-\delta_2)t} }{-(\mu-\delta_1) +(\mu-\delta_2)} .
	\] 
\end{proof}

\begin{prop}
	\label{prop:Endomorphism}
	Define:  
	\begin{align}
	\delta_3 =   2 ( \rho ( r_c + \rho r_s) + r_s).
	\end{align}
	If $ \delta_2  <\mu $ and $\frac{\delta_3}{\mu - \delta_2}< \rho$, then the map $ \Psi: \cB \to \cB$ is a well defined endomorphism.   
\end{prop}

\begin{proof}

	\textbf{Step 1.} 
	Fix some $ \alpha \in \cB$ and  $	\phi \in B_c(r_c)$ and $ \xi ,\zeta \in B_s(r_s)$.
	We show that	$ |  \Psi[\alpha](\phi ,\xi) - \Psi[\alpha](\phi ,\zeta) | \leq \rho \| \xi - \zeta\|_{\ell^1}$.
	Let us abbreviate $x_s = x_s(t) = x(t,\phi,\xi,\alpha)
	$ and $z_s = z_s(t) = x(t,\phi, \zeta , \alpha)$. Additionally, let us abbreviate $ x_c = \alpha ( \Phi(t,\phi),x_s(t)) $ and $ z_c = \alpha ( \Phi(t,\phi),z_s(t)) $. 
	We calculate: 
	\begin{align*}
	|		\cN_c( x_c ,x_s) - \cN_c( z_c , z_s) |  &\leq | x_c^2 - z_c^2 | + \|x_s^2 - z_s^2 \|_{\ell^1} \\ 
	&= 2 (r_c + \rho r_s) |x_c - z_c| + 2 r_s \| x_s - z_s\|_{\ell^1} \\
	&= \delta_3  \|x_s - z_s \|_{\ell^1} .
	\end{align*} 
	We calculate: 
	\begin{align*}
	| \Psi [\alpha ] (\phi , \xi) - \Psi [\alpha ] (\phi , \zeta )| &\leq \int_0^\infty  |		\cN_c( x_c ,x_s) - \cN_c( z_c , z_s) |  dt \\
	&\leq \int_0^\infty \delta_3 \| x_s - z_s \|_{\ell^1} dt \\
	&\leq \int_0^\infty  \delta_3  e^{-(\mu-\delta_2)t} \|\xi - \zeta \|_{\ell^1}  dt  \\
	& = \frac{\delta_3}{\mu - \delta_2 } \|\xi - \zeta\|_{\ell^1}.  
	\end{align*}
	By our assumption that $\frac{\delta_3}{\mu - \delta_2}< \rho$, it follows that we show that	$ |  \Psi[\alpha](\phi ,\xi) - \Psi[\alpha](\phi ,\zeta) | \leq \rho \| \xi - \zeta\|_{\ell^1}$.

	\textbf{Step 2.}  We  show that  the map $ \Psi[\alpha](\cdot , \xi )$ is Lipschitz for fixed $ \xi \in B_s$. 		
	Fix some $ \alpha \in  \cB$ and $ \phi_1 ,\phi_2  \in B_c(r_c)$ and $ \xi \in B_s(r_s)$.  
	As $\alpha$ is Lipschitz, fix  some $ \epsilon > 0$ such that     $ |\alpha( \phi_1 , \zeta ) - \alpha(\phi_2 , \zeta) | \leq \epsilon | \phi_1 - \phi_2|$ for all $ \zeta \in B_s(r_s)$.  
	% To prove the theorem, it suffices to prove that there exists a constant $K$ such that 
	% $ |  \Psi[\alpha](\phi_1 ,\xi) - \Psi[\alpha](\phi_2 ,\xi) | \leq K  | \xi - \zeta|$ for all $ \xi,\zeta \in B_s(r_s)$. 
	Let us define $ x_s = x_s(t) = x(t, \phi_1 , \xi , \alpha) $ and $ w_s = w_s(t) = x(t, \phi_2, \xi , \alpha)$. Let us define $ x_c = x_c(t) = \alpha ( \Phi(t,\phi_1),x_s(t))$ and define $ w_c = w_c(t) = \alpha( \Phi(t, \phi_2 ) , w_s(t))$. 
	We calculate: 
	\begin{align*}
	| \cN_c (x_c , x_s) - \cN_c(w_c ,w_s) | &\leq | x_c^2 - w_c^2 | + \| x_s^2 - w_s^2\|_{\ell^1} \\
	&\leq | \alpha( \Phi(t,\phi_1),x_s)^2 - \alpha( \Phi(t,\phi_2),x_s)^2 | + 2 (r _c  + r_s + \rho r_s) \| x_s - w_s\|_{\ell^1}  .
	\end{align*}
	We can bound the first summand as follows:
	\begin{align*}
	| \alpha( \Phi(t,\phi_1),x_s)^2 - \alpha( \Phi(t,\phi_2),x_s)^2 | &\leq 
	| \alpha( \Phi(t,\phi_1),x_s) + \alpha ( \Phi(t, \phi_2),x_s) | \cdot | \alpha( \Phi(t,\phi_1),x_s) - \alpha ( \Phi(t, \phi_2 ) , x_s) | \\
	&\leq\Big| 2 \rho | x_s | + \Phi(t,\phi_1) + \Phi(t,\phi_2) \Big| \epsilon | \Phi(t,\phi_1) - \Phi(t,\phi_2)| \\
	&\leq \epsilon | \phi_1 - \phi_2| \left( 
	2 \rho| x_s| + \frac{|\phi_1|}{1 + | \phi_1 t|^2} + \frac{|\phi_2|}{1 + | \phi_2 t|^2}
	\right).
	\end{align*}
	In the last inequality, we used our estimate from \eqref{eq:TightPhiDifferenceBound}. 
	Using our calculated bound on $| \cN_c (x_c , x_s) - \cN_c(w_c ,w_s) | $, we estimate $| \Psi[\alpha](\phi_1,\xi) - \Psi[\alpha] (\phi_2,\xi) | $ below: 
	\begin{align*}
	| \Psi[\alpha](\phi_1,\xi) - \Psi[\alpha] (\phi_2,\xi) | &\leq 
	\epsilon |\phi_1 - \phi_2| \int_0^\infty \rho r_s e^{-(\mu-\delta_1)t}  + \frac{|\phi_1|}{1 + | \phi_1 t|^2} + \frac{|\phi_2|}{1 + | \phi_2 t|^2} dt  \\
	& \qquad + 2(r_c + r_s + \rho r_s)\int_0^\infty 2 \epsilon r_s | \phi_1 - \phi_2|  \frac{e^{-(\mu-\delta_1)t} -  e^{-(\mu-\delta_2)t} }{-(\mu-\delta_1) +(\mu-\delta_2)}  dt  .
	\end{align*}
	The second integral can be evaluated as follows:
	\begin{align}
	\int_0^\infty   \frac{e^{-(\mu-\delta_1)t}  - e^{-(\mu-\delta_2)t} }{-(\mu-\delta_1)+(\mu-\delta_2)} dt & =  \frac{1}{-(\mu-\delta_1)+(\mu-\delta_2)} \left( \frac{1}{(\mu-\delta_1)} -\frac{1}{(\mu-\delta_2)} \right) \nonumber \\
	&= \frac{1}{(\mu-\delta_1)(\mu-\delta_2)} \label{eq:FunnyIntegral}.
	\end{align}
	Hence we obtain the final estimate: 
	\begin{align*}
	| \Psi[\alpha](\phi_1,\xi) - \Psi[\alpha] (\phi_2,\xi) | &\leq \epsilon \left (  
	\frac{\rho r_s}{\mu - \delta_1} + \pi + \frac{4 r_s (r_c + r_s + \rho r_s )}{(\mu - \delta_1)(\mu - \delta_2)}
	\right) | \phi_1 - \phi_2|  .
	\end{align*} 
	
	We have thus shown that $ \Psi[\alpha]$ is Lipschitz in each of its coordinates, thus $ \Psi[\alpha] \in \mbox{Lip} ( B_c \times B_s , X_c)$. 
	Then for all $ \alpha \in \cB$ and $ \phi \in X_c$, since $ \Phi(t,\phi) $ solves \eqref{eq:CenterRestrictedDE}, we have  $ \Psi[\alpha ](\phi,0)=\phi$.  Thereby, $\Psi : \cB\to \cB$ is a well defined endomorphism. 	
\end{proof}

\begin{dfn} \label{def:WeakerNorm}
	For $ \alpha \in \cB$ define the norm: 
	\begin{align*}
	\| \alpha \|_{\cE} \bydef \sup_{x_s \in B_s(r_s) ; x_s \neq 0 } \frac{\| \alpha (x_c,x_s)-x_c\|_{\ell^1}}{\|x_s\|_{\ell^1}} .
	\end{align*}
\end{dfn}
This norm makes $ \cB$ a complete metric space. 
We will show that $ \Psi $ is a contraction on a ball of functions in the $ \| \cdot \|_{\cE}$ norm. 
But first we prove a helpful estimate. 
\begin{prop} \label{prop:DifferentMapTracking}
	Fix $ \alpha , \beta \in \cB, \phi \in B_c(r_c), \xi \in B_s(r_s)$.
	Define 
	\begin{align*}
	\delta_4 &= 2 (r_c + 2 \rho r_s + r_s)  .
	\end{align*}
	Then,
	\begin{align*}
	\|x(t,\phi,\xi,\alpha)
	- x(t,\phi,\xi,\beta)\|_{\ell^1} \leq 2 r_s  \| \alpha - \beta \|_{\cE} \| \xi \|_{\ell^1} \left( \frac{e^{-(\mu-\delta_1)t}  - e^{-(\mu-\delta_4)t} }{-(\mu-\delta_1)+(\mu-\delta_4)} \right) .
	\end{align*}
\end{prop}
\begin{proof}
	Let us abbreviate $x_s = x_s(t) = x(t,\phi,\xi,\alpha)
	$ and $y_s = y_s(t) = x(t, \phi, \xi, \beta)$. Additionally, let us abbreviate $ x_c = \alpha ( \Phi(t,\phi),x_s(t)) $ and $ y_c = \beta( \Phi(t,\phi),y_s(t)) $.

	By variations of constants we have the following:
	\begin{align*}
	x_s(t) - y_s(t) & = \int_0^t e^{\fL (t-\tau)} \left(  \cN_s ( x_c,x_s) - \cN_s(y_c,y_s) \right) d \tau .
	\end{align*}
	We bound the difference of the nonlinearities: 
	\begin{align*}
	\|	 \cN_s ( x_c,x_s) - \cN_s(y_c,y_s)\|_{\ell^1} & \leq 2\| x_c x_s - y_c y_s \|_{\ell^1} + \| x_s^2 - y_s^2 \|_{\ell^1} \\
	&\leq 2\|x_c x_s -  x_c y_s \|_{\ell^1} + 2 \|x_c y_s - y_c y_s \|_{\ell^1} + 2 r_s \|x_s - y_s\|_{\ell^1} \\
	&\leq 2|x_c |\| x_s -   y_s \|_{\ell^1} + 2 |x_c - y_c  | \|y_s\|_{\ell^1}+ 2 r_s  \|x_s - y_s\|_{\ell^1}  \\
	&\leq 2 (r_c + \rho r_s + r_s) \|x_s - y_s \|_{\ell^1} + 2 r_s | x_c - y_c|  .
	\end{align*}
	We calculate: 
	\begin{align*}
	|x_c - y_c| &\leq 
	| \alpha ( \Phi(t,\phi),x_s) - \alpha( \Phi(t,\phi),y_s) |
	+ |  \alpha( \Phi(t,\phi),y_s) - \beta( \Phi(t,\phi),y_s) | \\
	&\leq \rho \|x_s - y_s\|_{\ell^1} + \| \alpha - \beta \|_{\cE} \| y_s\|_{\ell^1}  .
	\end{align*}
	Hence: 
	\begin{align*}
	\|	 \cN_s ( x_c,x_s) - \cN_s(y_c,y_s)\|_{\ell^1} 
	&\leq 2 (r_c + 2 \rho r_s + r_s) \|x_s - y_s \|_{\ell^1} + 2 r_s \| \alpha - \beta \|_{\cE} \|y_s\|_{\ell^1}  \\
	&=  \delta_4 \|x_s - y_s \|_{\ell^1} + 2 r_s \| \alpha - \beta \|_{\cE} \|y_s\|_{\ell^1} .
	\end{align*}
	We may continue calculating using variation of constants: 
	\begin{align*}
	e^{\mu t} \| x_s - y_s\|_{\ell^1} &\leq \int_0^t e^{\mu \tau} \left( \delta_4 \|x_s - y_s \|_{\ell^1} + 2 r_s \| \alpha - \beta \|_{\cE} \|y_s\|_{\ell^1} \right) d \tau \\
	& \leq  \int_0^t e^{\mu \tau} 2 r_s \| \alpha - \beta \|_{\cE} e^{-(\mu -\delta_1)\tau} \| \xi\|_{\ell^1}  d \tau + \int_0^t e^{\mu \tau} \delta_4 \|x_s - y_s \|_{\ell^1} d\tau .
	\end{align*}
	By  Lemma \ref{lem:GronwallType}, we have that: 
	\begin{align*}
	\| x_s - y_s\|_{\ell^1} & \leq  2 r_s \| \alpha - \beta \|_{\cE}  \| \xi\|_{\ell^1} \frac{e^{-(\mu-\delta_1)t} - e^{-(\mu-\delta_4)t}}{-(\mu-\delta_1)+(\mu-\delta_4)}.
	\end{align*}
	The proposition follows. 
\end{proof}

We can now show that $ \Psi $ is a contraction map.

\begin{proof}[Proof of Theorem \ref{prop:MainGlobalTheorem}]

	Fix $ \alpha , \beta \in \cB, \phi \in B_c(r_c), \xi \in B_s(r_s)$, and define:
	\begin{align*}
	\lambda &\bydef \frac{ 2 (  \rho(r_c+ \rho r_s) + r_s )    2 r_s  	}{(\mu-\delta_1)(\mu-\delta_4)} 
	+ \frac{2( r_c + \rho r_s)    }{\mu-\delta_1} .
	\end{align*}
	We will show that: 
	\begin{align*}
	\| \Psi[\alpha ] (\phi , \xi) - \Psi[\beta](\phi,\xi) \|_{\cE} \leq \lambda \| \alpha - \beta \|_{\cE} .
	\end{align*}

	Let us abbreviate $x_s = x_s(t) = x(t,\phi,\xi,\alpha)
	$ and $y_s = y_s(t) = x(t,\phi, \xi, \beta)$. Additionally, let's abbreviate $ x_c = \alpha ( \Phi(t,\phi),x_s(t)) $ and $ y_c = \beta ( \Phi(t,\phi),y_s(t)) $. 
	We wish to bound the follow integral:
	\begin{align*}
	| \Psi[\alpha ] (\phi , \xi) - \Psi[\beta](\phi,\xi) | & \leq   \int_0^\infty | \cN_c( x_c,x_s)-\cN_c(y_c,y_s)| dt  .
	\end{align*}
	We calculate: 
	\begin{align*}
	| \cN_c( x_c,x_s)-\cN_c(y_c,y_s)| & \leq |x_c^2 - y_c^2 | + \|x_s^2 - y_s^2\|_{\ell^1} \\
	&\leq 2( r_c + \rho r_s)  | x_c - y_c| + 2 r_s \| x_s - y_s\|_{\ell^1} \\
	&\leq 2 (  \rho(r_c+ \rho r_s) + r_s )\| x_s - y_s\|_{\ell^1} + 2( r_c + \rho r_s)   \| \alpha - \beta \|_{\cE} \| y_s\|_{\ell^1} .
	\end{align*}
	Hence: 
	\begin{align*}
	| \Psi[\alpha ] (\phi , \xi) - \Psi[\beta](\phi,\xi) | & \leq  2 (  \rho(r_c+ \rho r_s) + r_s ) \int_0^\infty \| x_s - y_s\|_{\ell^1} dt   \\
	& \qquad +  2( r_c + \rho r_s)   \| \alpha - \beta \|_{\cE} \int_0^\infty \|y_s\|_{\ell^1} dt \\
	&\leq  2 (  \rho(r_c+ \rho r_s) + r_s ) \int_0^\infty  2 r_s  \| \alpha - \beta \|_{\cE} \| \xi \|_{\ell^1} \left( \frac{e^{-(\mu-\delta_1)t}  - e^{-(\mu-\delta_4)t} }{-(\mu-\delta_1)+(\mu-\delta_4)} \right) dt \\
	& \qquad +  2( r_c + \rho r_s)   \| \alpha - \beta \|_{\cE} \int_0^\infty e^{-(\mu-\delta_1)t} \|\xi\|_{\ell^1}  dt .
	\end{align*}
	The first integral can be solved analogously to \eqref{eq:FunnyIntegral}, and thus we have the final bound:  
	\begin{align*}
	| \Psi[\alpha ] (\phi , \xi) - \Psi[\beta](\phi,\xi) | &\leq 	\frac{ 2 (  \rho(r_c+ \rho r_s) + r_s )    2 r_s  \| \alpha - \beta \|_{\cE} \| \xi \|_{\ell^1}	}{(\mu-\delta_1)(\mu-\delta_4)} 
	+ \frac{2( r_c + \rho r_s)   \| \alpha - \beta \|_{\cE} \|\xi\|_{\ell^1} }{\mu-\delta_1} .
	\end{align*}
	After factoring out $  \| \alpha - \beta \|_{\cE} \|\xi\|_{\ell^1} $, we have our definition of $ \lambda$. 
	
	If $\lambda <1$, as assumed in the Theorem, then $\Psi $ is a contraction mapping. So by the Banach fixed point theorem, there exists a unique map $ \hat{\alpha} \in \cB$ such that $ \Psi [ \hat{\alpha}] = \hat{\alpha}$. 
	
\end{proof}

%\include{numerical_result}
%%%%%%%%%%%%%%%%
\section{Numerical results} \label{sec:Numerical_results}
%%%%%%%%%%%%%%%%
In this section, we show computer-assisted proofs of our main results for the initial-boundary value problem \eqref{eq:CGL}.
All computations are carried out on Windows 10, Intel(R) Core(TM) i7-6700K CPU @ 4.00GHz, and MATLAB 2019a with INTLAB - INTerval LABoratory \cite{Ru99a} version 11 and Chebfun - numerical computing with functions \cite{MR2767023} version 5.7.0.
All codes used to produce the results in this section are freely available from \cite{bib:codes}.

\subsection{Proof of Theorem \ref{thm:branching}}

It is obvious that the solution of \eqref{eq:cnheq} has a symmetry $u(t\expmig,x)=\overline{u(t\expig,x)}$ for a real $t$.
Then, to prove Theorem \ref{thm:branching}, it is sufficient to show that the imaginary part of $u(z,x)$ is a non-zero function at a certain point $z\in\R$ satisfying $z_B<z$, where $z_B$ denotes the blow-up time of \eqref{eq:nheq}.
We take a path $\tilde{\Gamma}_{\theta}$ which bypasses the blow-up point $z_B$ as shown in Fig.~\ref{fig:fig2}.
Here, $z_B\approx0.0119$ holds \cite{Cho2016} under the periodic boundary condition with the initial data $u_0(x)=50(1-\cos(2\pi x))$. 

\begin{figure}[htbp]\em
	\centering
	\includegraphics[width=8cm]{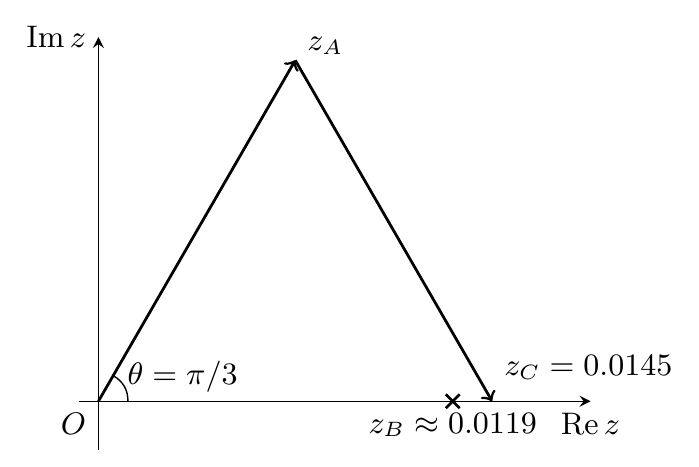}\\
	\caption{%
		The path $\tilde{\Gamma}_{\theta} \bydef \{z: z=t\expig~(O\to z_A),~z=z_A+t\expmig~(z_A\to z_C),~\theta=\pi/3\}$ is plotted, where $z_A = 0.00725(1 +\sqrt{3}\,\im)$ and $z_C=0.0145$.
		We divide each segment $(O\to z_A$ and $z_A\to z_C)$ into 64 steps.
	}
	\label{fig:fig2}
\end{figure}

We divide each segment into 64 steps and, by using our rigorous integrator introduced in Section \ref{sec:setting} -- \ref{sec:timestepping}, analytically continue to the $z_C$ point in Fig.~\ref{fig:fig2}.
From $O$ to $z_A$, we set $\theta=\pi/3$ in \eqref{eq:CGL}.
After that, we change $\theta=-\pi/3$ from $z_A$ to $z_C$.
For each time step $J_i$ ($i=1,2,\dots,128$), we get an approximate solution by using Chebfun \cite{MR2767023} as
\begin{align}
\bar{u}(t,x) &= \sum_{|k| \le N} \left(\ba_{0,k} + 2\sum_{\ell=1}^{n-1} \ba_{\ell,k} T_{\ell}(t)\right) e^{\im k\omega x},~\omega=2\pi,~x\in(0,1),~t\in J_i
\label{eq:ApproximateSolution}
\end{align}
with $N = 25$ and $n = 13$.
We also set $m=0$ in Theorem \ref{thm:sol_map}, which decides the size of variational problem.
The profiles of numerically computed $\mathrm{Re}(\bar{u})$ and $\mathrm{Im}(\bar{u})$ are plotted in Fig.~\ref{fig:fig3}.

\begin{figure}[htbp]\em
	\centering
	\begin{minipage}{0.49\hsize}
		\centering
		\includegraphics[width=\textwidth]{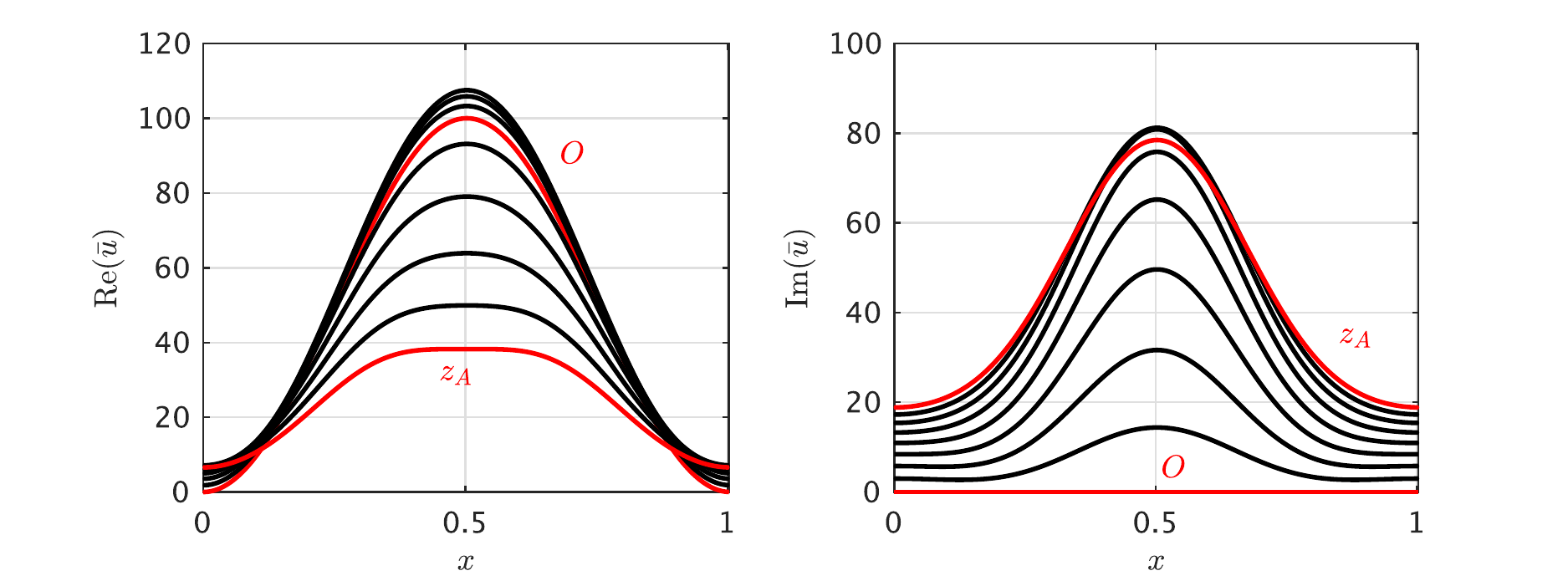}\\
		$(a)$ $O\to z_A$
	\end{minipage}
	\begin{minipage}{0.49\hsize}
		\centering
		\includegraphics[width=\textwidth]{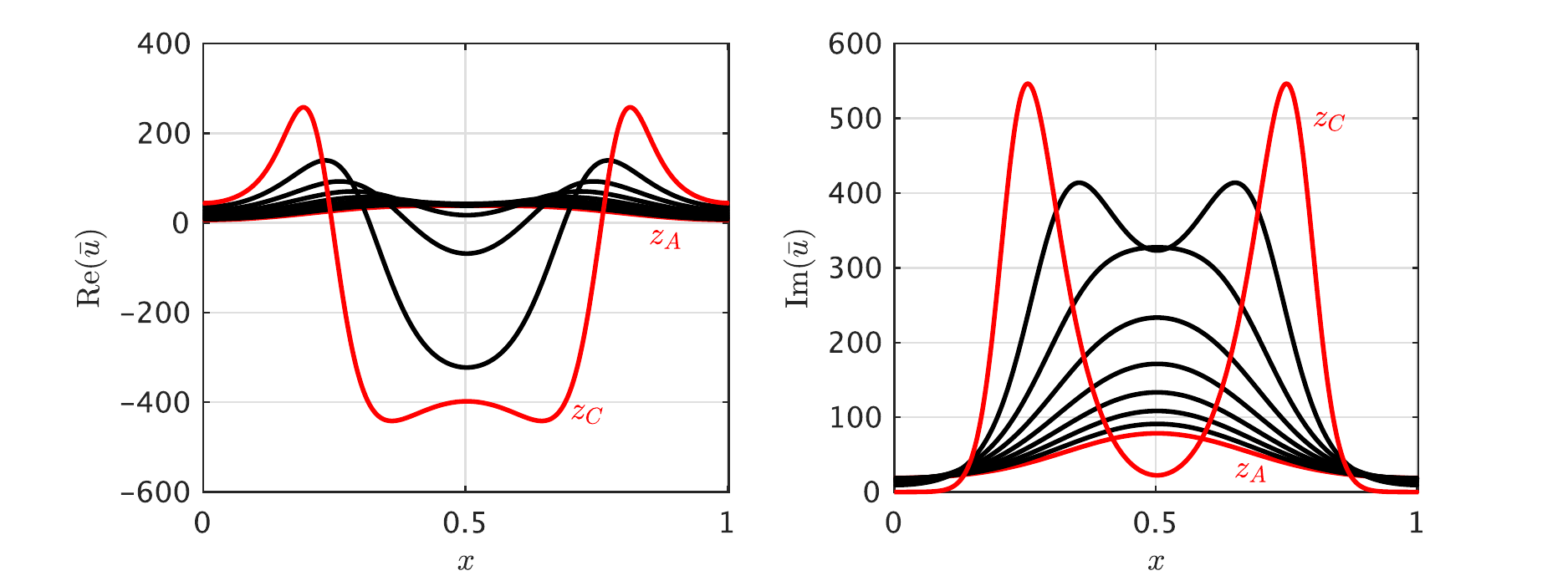}\\
		$(b)$ $z_A\to z_C$
	\end{minipage}
	\caption{%
		Profiles of numerically computed solutions on each segment $((a)~O\to z_A$ and $(b)~z_A\to z_C)$ are plotted.
		At the point $z_C$, the imaginary part of $\bar{u}(z_C,x)$ becomes a non-zero function.
	}
	\label{fig:fig3}
\end{figure}

After 128 time stepping (at $z_C$), we show that the imaginary part of the solution $u(z_C,x)$ is a non-zero function.
To prove this, we use the $\ell^1$ norm of the Fourier coefficients, say
\[
\|v\|\bydef \|c\|_{\ell^1}=\sum_{k\in\Z }|c_k|\quad\mbox{for}\quad v(x) = \sum_{k\in\Z }c_ke^{\im k\omega x}.
\]
Let the solution $u(z_C,x)$ of \eqref{eq:cnheq} at $z_C$ and its numerically computed solution be denoted by
\[
u(z_C,x)=\sum_{k\in\mathbb{Z}}a_k^{z_C}e^{\im k\omega x}\quad\mbox{and}\quad \bar{u}(z_C,x)=\sum_{|k|\le N}\ba_k^{z_C}e^{\im k\omega x},
\]
respectively.
We also denote two bi-infinite complex-valued sequences $a^{z_C}=(a_k^{z_C})_{k\in\mathbb{Z}}$ and $\ba^{z_C}=(\dots,0,\ba_{-N}^{z_C},\dots,\ba_{N}^{z_C},0\dots)$.
Then, the imaginary part of the solution at $z_C$ follows
\begin{align*}
\|\mathrm{Im}\left(u(z_C,\cdot)\right)\|&=\|\mathrm{Im}\left(\bar{u}(z_C,\cdot)\right)+\mathrm{Im}\left(u(z_C,\cdot)\right)-\mathrm{Im}\left(\bar{u}(z_C,\cdot)\right)\|\\
&\ge\|\mathrm{Im}\left(\bar{u}(z_C,\cdot)\right)\|-\|\mathrm{Im}\left(u(z_C,\cdot)-\bar{u}(z_C,\cdot)\right)\|\\
&\ge \|\mathrm{Im}\left(\bar{u}(z_C,\cdot)\right)\|-\left\|a^{z_C}-\ba^{z_C}\right\|_{\ell^1}\\
&\ge \|\mathrm{Im}\left(\bar{u}(z_C,\cdot)\right)\|-\varepsilon^{z_C},
\end{align*}
where $\varepsilon^{z_C}$ is the point-wise error at $z_C$ point and our rigorous computation yields $\varepsilon^{z_C}=0.5765$.
Furthermore, the imaginary part of $\bar{u}(z_C,x)$ is presented by
\begin{equation}\label{eq:imag_tilde_u}
\mathrm{Im}\left(\bar{u}(z_C,x)\right)=\frac{1}{2}\sum_{|k| \le N}\left[\mathrm{Im}\left(\ba_{k}^{z_C}\right)+\mathrm{Im}\left(\ba_{-k}^{z_C}\right)-\im\left(\mathrm{Re}\left(\ba_{k}^{z_C}\right)-\mathrm{Re}\left(\ba_{-k}^{z_C}\right)\right)\right] e^{\im k\omega x}.
\end{equation}
This fact is followed by
\begin{align*}
\mathrm{Im}\left(\ba_{k}^{z_C}e^{\im k\omega x}\right)&=\mathrm{Im}\left(\ba_{k}^{z_C}\right)\cos(k\omega x)+\mathrm{Re}\left(\ba_{k}^{z_C}\right)\sin(k\omega x)\\
&=\mathrm{Im}\left(\ba_{k}^{z_C}\right)\left(\frac{e^{\im k\omega x}+e^{-\im k\omega x}}{2}\right)+\mathrm{Re}\left(\ba_{k}^{z_C}\right)\left(\frac{e^{\im k\omega x}-e^{-\im k\omega x}}{2\im}\right)
\end{align*}
for each $|k|\le N$.
Summing up for $k$, we have
\begin{align*}
\sum_{|k| \le N} \mathrm{Im}\left(\ba_{k}^{z_C}e^{\im k\omega x}\right) &= \frac{1}{2}	\sum_{|k| \le N}\left[ \mathrm{Im}\left(\ba_{k}^{z_C}\right)\left({e^{\im k\omega x}+e^{-\im k\omega x}}\right)-\im\cdot\mathrm{Re}\left(\ba_{k}^{z_C}\right)\left({e^{\im k\omega x}-e^{-\im k\omega x}}\right)\right]\\
&=\frac{1}{2}\sum_{|k| \le N}\left[\mathrm{Im}\left(\ba_{k}^{z_C}\right)+\mathrm{Im}\left(\ba_{-k}^{z_C}\right)-\im\left(\mathrm{Re}\left(\ba_{k}^{z_C}\right)-\mathrm{Re}\left(\ba_{-k}^{z_C}\right)\right)\right] e^{\im k\omega x}.
\end{align*}
We rigorously compute \eqref{eq:imag_tilde_u} based on interval arithmetic and the following is given
% Then, rigorous computation based on interval arithmetic also yields
\[
\|\mathrm{Im}\left(\bar{u}(z_C,\cdot)\right)\|-\varepsilon^{z_C}
%	=\frac{1}{2}\sum_{|k| \le N}\left|\mathrm{Im}\left(\ba_{k}^{z_C}\right)+\mathrm{Im}\left(\ba_{-k}^{z_C}\right)-\im\left(\mathrm{Re}\left(\ba_{k}^{z_C}\right)-\mathrm{Re}\left(\ba_{-k}^{z_C}\right)\right)\right|
\in [660.4935,  660.4936].
\]
This implies $\|\mathrm{Im}\left(u(z_C,\cdot)\right)\|>0$.

Consequently, we prove that the imaginary part of the solution $u(z_C,x)$ is the non-zero function.
Then, there exists a branching singularity on the real line at $t\in (0,z_C)$ with $z_C=0.0145$.
The results of analytical continuation is shown in Fig.~\ref{fig:fig4}.

\begin{figure}[htbp]\em
	\centering
	\begin{minipage}{0.49\hsize}
		\centering
		\includegraphics[width=\textwidth]{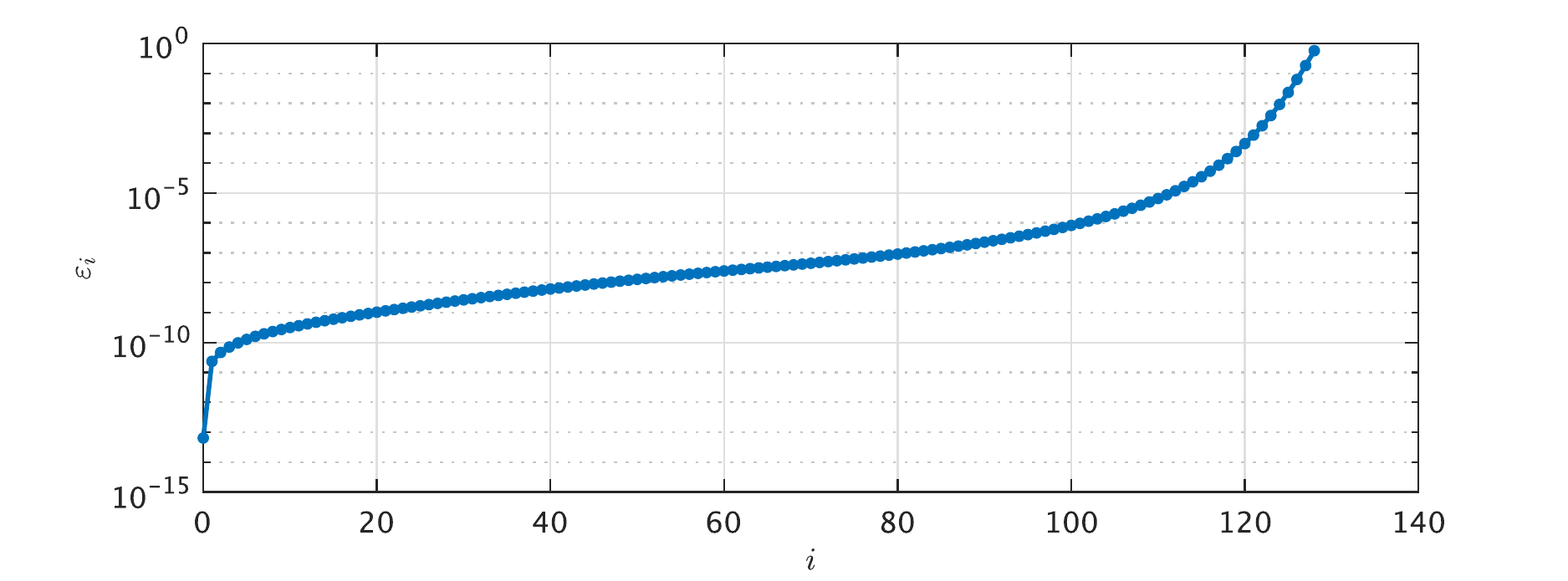}\\
		$(a)$ $\varepsilon_{i}$ bounds at $t_i$
	\end{minipage}
	\begin{minipage}{0.49\hsize}
		\centering
		\includegraphics[width=\textwidth]{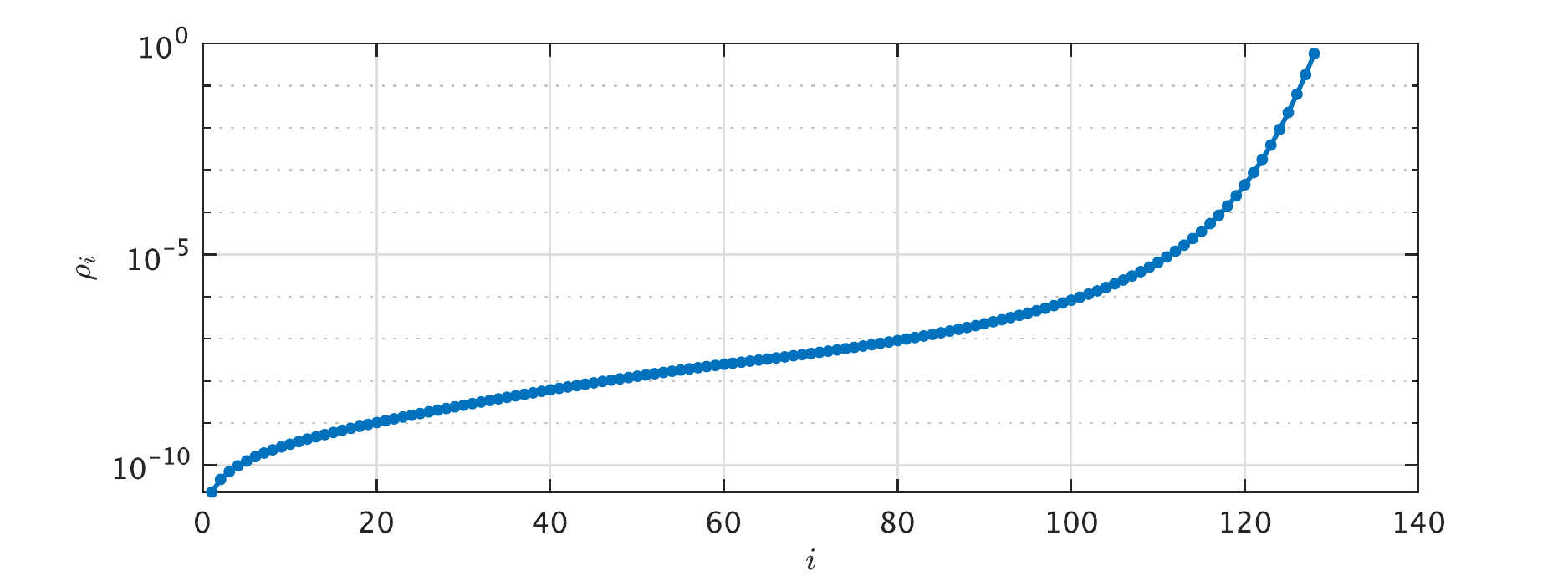}\\
		$(b)$ $\varrho_i$ bounds on $J_i$
	\end{minipage}
	\caption{%
		For each time step $J_i~(i=1,2,\dots,128)$,  results of analytical continuation are plotted. $(a)$ The point-wise error $\varepsilon_{i}$ for $i=0,1,\dots,128$.
		$(b)$ The radius of the neighborhood $B_{J_i}(\ba,\varrho_i)$ in which the exact solution is included for $i=1,2,\dots,128$.
	}
	\label{fig:fig4}
\end{figure}

Lastly, to obtain a lower bound on the location of the branch singularity $z_B$, it suffices to rigorously integrate our initial data $u_0(x)$ in purely real time with $ \theta = 0 $ in \eqref{eq:CGL} for as long as possible. 
Again using Chebfun \cite{MR2767023} we obtain an approximate solution as in \eqref{eq:ApproximateSolution} using computational parameters $N=20$, $n=15$, $m=0$, and adaptive time stepping.
By using our rigorous integrator introduced in Section \ref{sec:setting} -- \ref{sec:timestepping}, we are able to validate our solution until $t= 0.0116$, hence the blowup point must satisfy $ z_B \geq 0.0116$. 
\hfill$\square$
 
 \begin{rem}
 	Generally speaking, the propagation of error estimates makes rigorous integration difficult.
 	As such, both the length of a contour and how close it approaches the blow-up point will    affect the results of our rigorous integrator. We believe the longer contour is a principal reason for explaining why we are able to get $0.0003$ close to the blowup point from below, but only $0.0031$ close from above.  
 	The upper bound could also be improved by adjusting the step size more carefully.
 \end{rem}

\subsection{Proof of Theorem \ref{thm:GE}}
%\corrc I'm here. Akitoshi. <<>>

We show the proof of global existence on the straight path $\Gamma_{\theta}$.
To prove the global existence of the solution of \eqref{eq:cnheq}, we check the hypothesis of Theorem \ref{prop:MainGlobalTheorem} by using rigorous numerics.
More precisely, at $t=t_i$ (after $i$\,th time stepping), we rigorously have
\begin{align*}
&|a_0(t_i)|\le |\ba_0(t_i)| + \varepsilon_{i},
&\|a^{(\bm{s})}(t_i)\|_{\ell^1}\le \|\ba^{(\bm{s})}(t_i)\|_{\ell^1} + \varepsilon_{i},
\end{align*}
where $a^{(\bm{s})}(t) = (a_k(t))_{k\in\Z,~k\neq 0}$ and $\ba^{(\bm{s})}(t) = (\ba_k(t))_{|k|\le N,~k\neq 0}$. 
To construct a trapping region $U$ as in Corollary \ref{prop:AlphaImage} which might contain $a(t_i)$, we fix mildly inflated radii constants
\begin{align*}
 r_c &\bydef \big(  |\ba_0(t_i)| + \varepsilon_{i} \big)  +0.02 \left( \|\ba^{(\bm{s})}(t_i)\|_{\ell^1} + \varepsilon_{i} \right) 
 &
 r_s &\bydef \|\ba^{(\bm{s})}(t_i)\|_{\ell^1} + \varepsilon_{i} .
\end{align*}
Then, for $\mu = \omega^2\cos\theta$, we compute the nearly optimal bounds of $\rho$ explicitly given in Remark \ref{rem:ChooseRho}.
If such a (positive) $\rho$ is not obtained, then we are unable validate the center stable manifold in a region large enough to contain $a(t_i)$.  
In such a case, we continue rigorous integration to the next time step, i.e., $J_{i+1}$.
Fig.~\ref{fig:fig5} shows verified results of our rigorous integration for $\theta=\pi/3$, $\pi/4$, $\pi/6$ and $\pi/12$.

Fig.~\ref{fig:fig5}\,$(a)$ shows results of our rigorous integrator in the case of $\theta=\pi/3$.
After 82 steps of the rigorous integration (at $t=0.205$), the hypothesis of Theorem \ref{prop:MainGlobalTheorem} holds for $r_c=4.9153$, $r_s=0.0081$, $\rho = 0.0086$, and $\lambda=0.9930$, and $a(t)$ is contained in the region $U$ described in Corollary \ref{prop:AlphaImage}. 
We succeed in proving the global existence of solution on $\Gamma_{\pi/3}$.

Fig.~\ref{fig:fig5}\,$(b)$ also shows results of our rigorous integrator in the case of $\theta=\pi/4$.
At $t=0.15$, the hypothesis of Theorem \ref{prop:MainGlobalTheorem} holds for $r_c=6.9290$, $r_s=0.0140$, $\rho = 0.0092$, and $\lambda=0.9868$, and $a(t)$ is contained in the region $U$ described in Corollary \ref{prop:AlphaImage}. 
Number of time stepping is 60.
We also succeed in proving the global existence of solution on $\Gamma_{\pi/4}$.

In the case of $\theta = \pi/6$, our rigorous integrator with the same setting above fails to prove local inclusion because the peak of solution becomes large ($\|\ba\|_X\approx 240$).
So we change the value of $m$ as $m=0$ ($0\le t\le 0.02$) and $m=2$ ($t\ge 0.02$).
Then, as shown in Fig.~\ref{fig:fig5}\,$(c)$, our rigorous integrator succeeds in including the solution until $t = 0.1275$ (51 steps).
At that time, the hypothesis of Theorem \ref{prop:MainGlobalTheorem} holds for $r_c=8.4817$, $r_s=0.0178$, $\rho = 0.0112$, and $\lambda=0.9858$, and $a(t)$ is contained in the region $U$ described in Corollary \ref{prop:AlphaImage}. 
Proof of the global existence on $\Gamma_{\pi/6}$ is complete.

Finally, the case $\theta = \pi/12$ is slightly difficult to complete.
The value of $\|\ba\|_X$ is almost 450 because this path is close to the blow-up point.
We adapt both the step size of time stepping and the value of $m$ as 
\[
h=\begin{cases}
1.25\times 10^{-3} & (0\le t\le 0.01),\\
6.25\times 10^{-4} & (0.01\le t\le 0.02),\\
2.5\times 10^{-3} & (t\ge 0.02),
\end{cases}\quad
m = \begin{cases}
0 & (0\le t\le 0.02),\\
2 & (t\ge 0.02).
\end{cases}
\]
Then, after 64 time stepping, the hypothesis of Theorem \ref{prop:MainGlobalTheorem} holds at $t=0.12$ for $r_c=9.3566$, $r_s=0.0278$, $\rho = 0.0095$, and $\lambda=0.9650$, and $a(t)$ is contained in the region $U$ described in Corollary \ref{prop:AlphaImage}. 
On $\Gamma_{\pi/12}$, we also prove the global existence of solution of \eqref{eq:CGL}.
The results of rigorous integrator is shown in Fig.~\ref{fig:fig5}\,$(d)$.\hfill$\square$

\begin{figure}[htbp]\em
	\begin{minipage}{\hsize}
		\centering
		\includegraphics[width=0.5\textwidth]{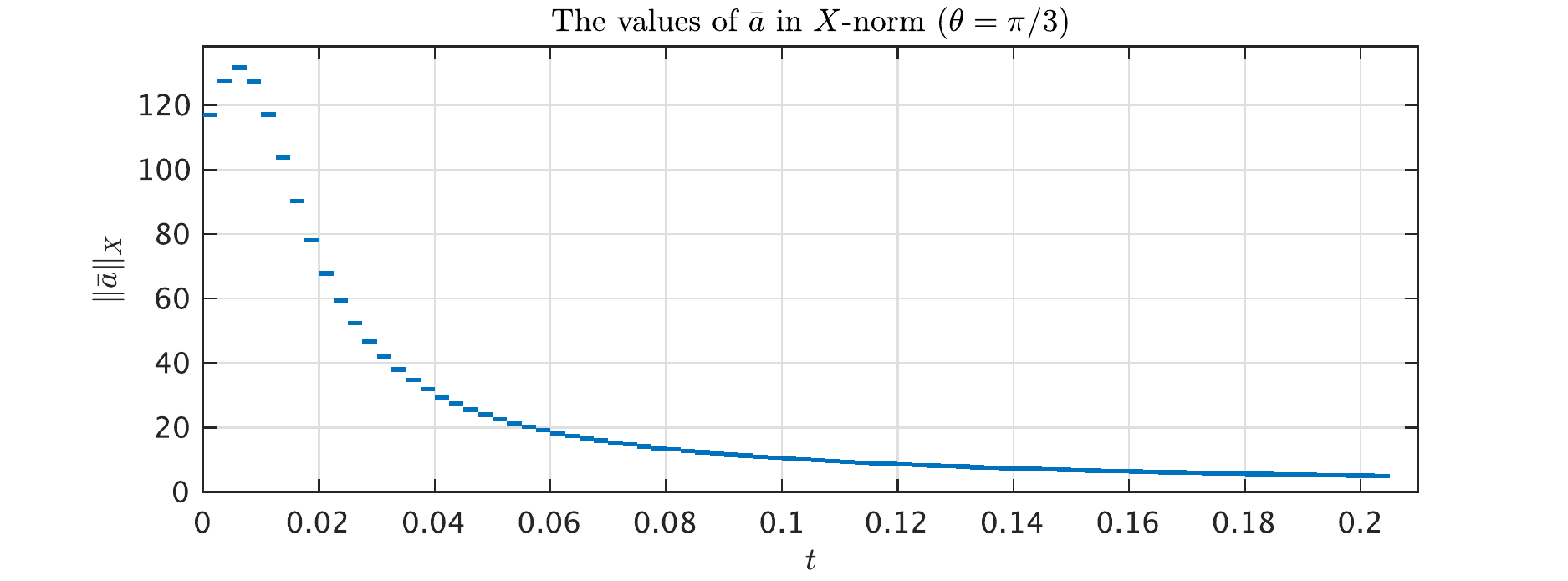}
		\includegraphics[width=0.49\textwidth]{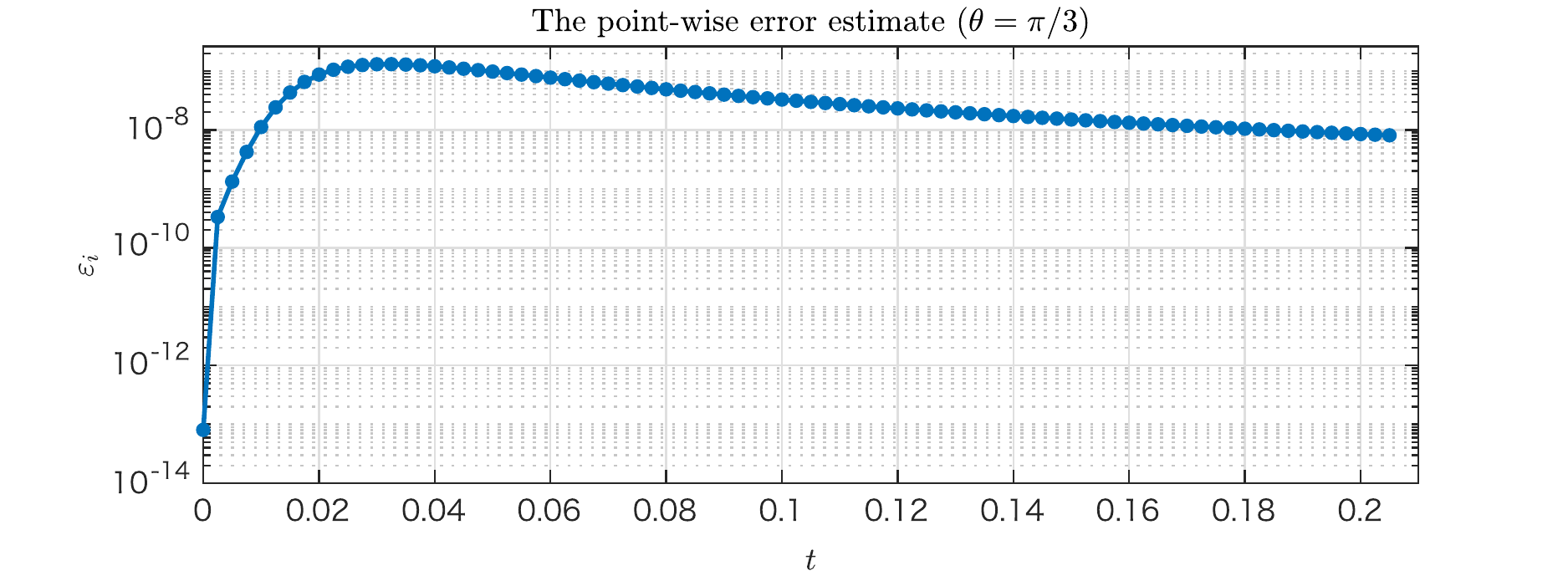}
		$(a)$ $\theta = \pi/3$: The step size is equidistantly taken as $h=2.5\times 10^{-3}$. We  set $N=14$ (maximum wave number of Fourier), $n=13$ (number of Chebyshev basis), and $m=2$.
	\end{minipage}\\[4mm]
	\begin{minipage}{\hsize}
		\centering
		\includegraphics[width=0.5\textwidth]{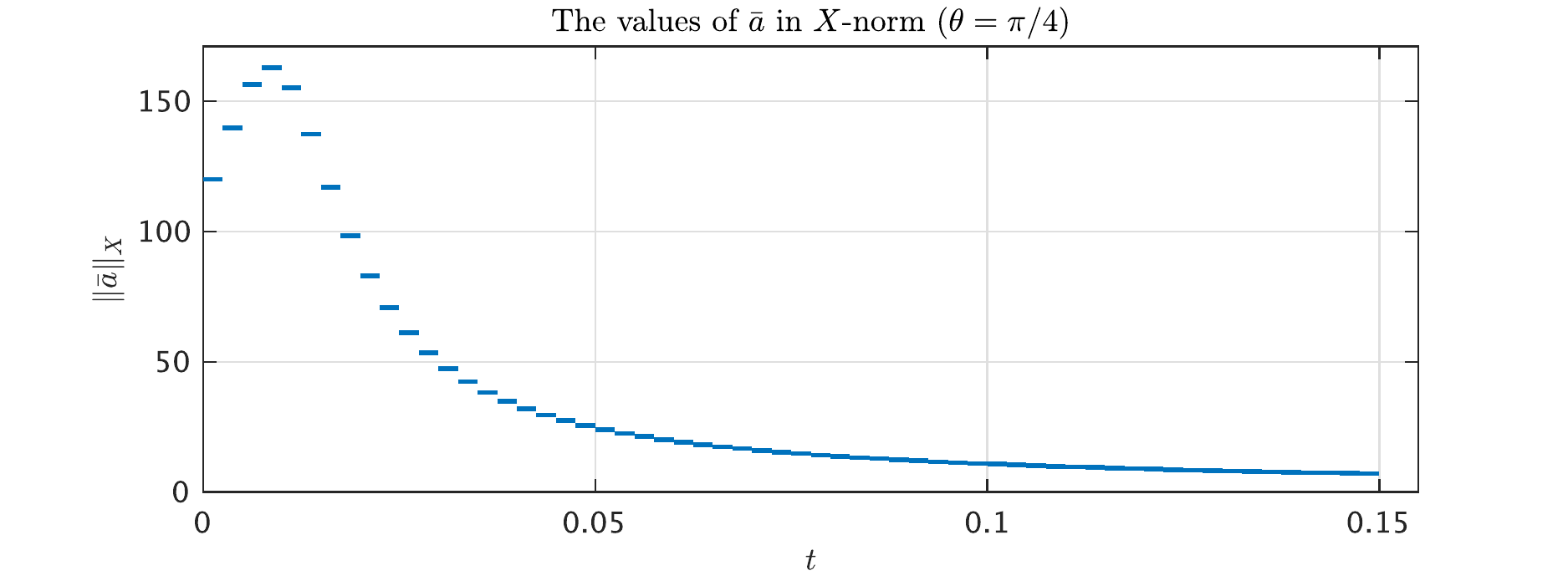}
		\includegraphics[width=0.49\textwidth]{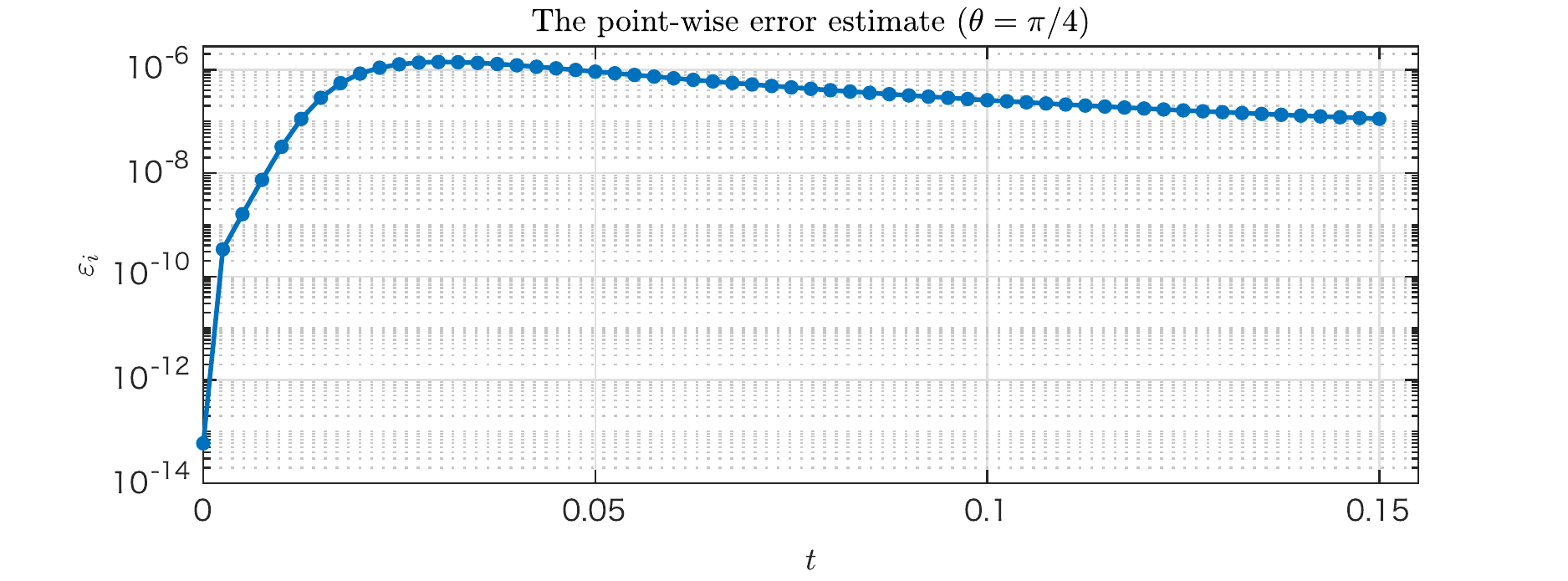}
		$(b)$ $\theta = \pi/4$: The step size is equidistantly taken as $h=2.5\times 10^{-3}$. We set $N=14$, $n=13$, and $m=2$.
	\end{minipage}\\[4mm]
	\begin{minipage}{\hsize}
		\centering
		\includegraphics[width=0.5\textwidth]{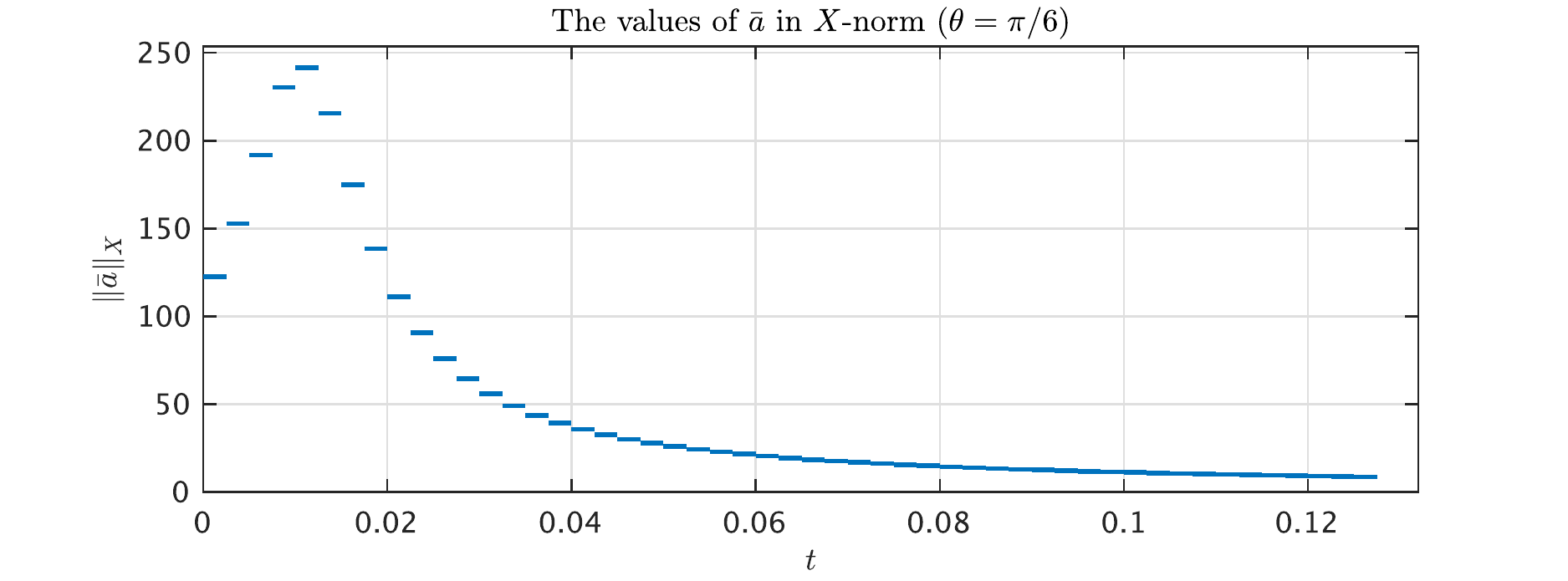}
		\includegraphics[width=0.49\textwidth]{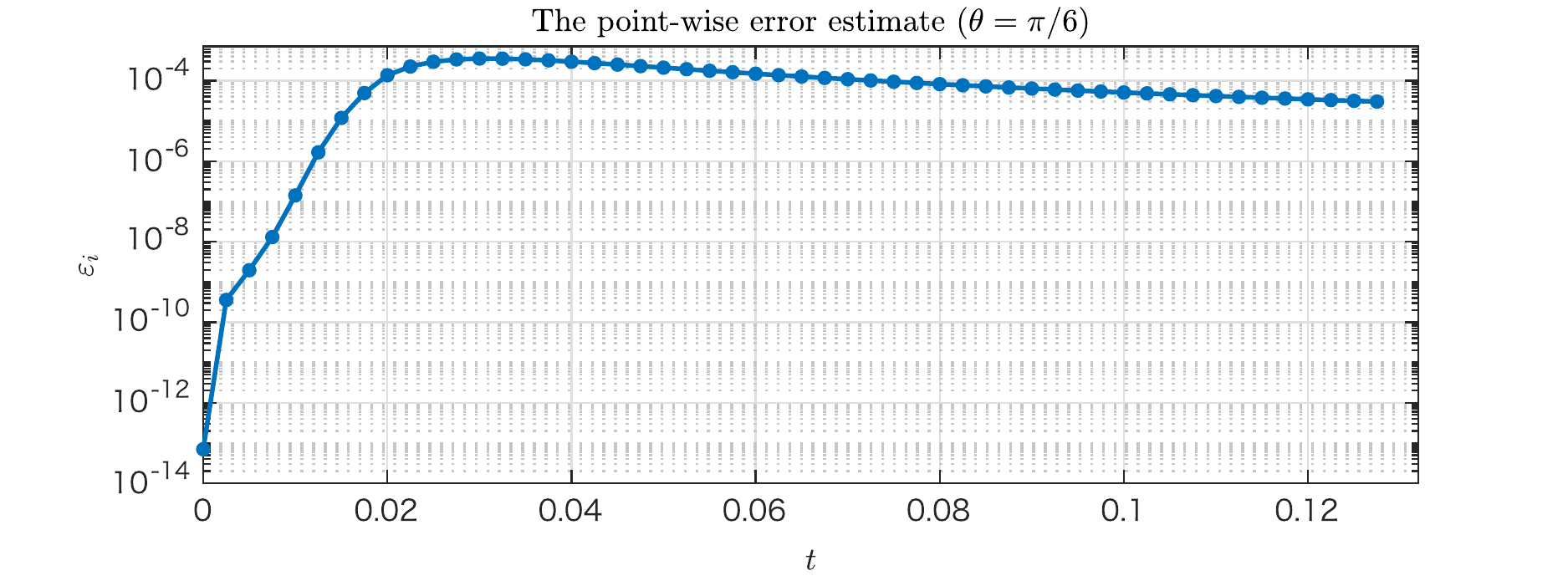}
		$(c)$ $\theta = \pi/6$: The step size is equidistantly taken as $h=2.5\times 10^{-3}$. We set $N=15$ and $n=14$.
		We adapt $m=0$ $(0\le t\le 0.02)$ and $m=2$ $(t\ge 0.02)$ for the success of our proof.
	\end{minipage}\\[4mm]
	\begin{minipage}{\hsize}
		\centering
		\includegraphics[width=0.5\textwidth]{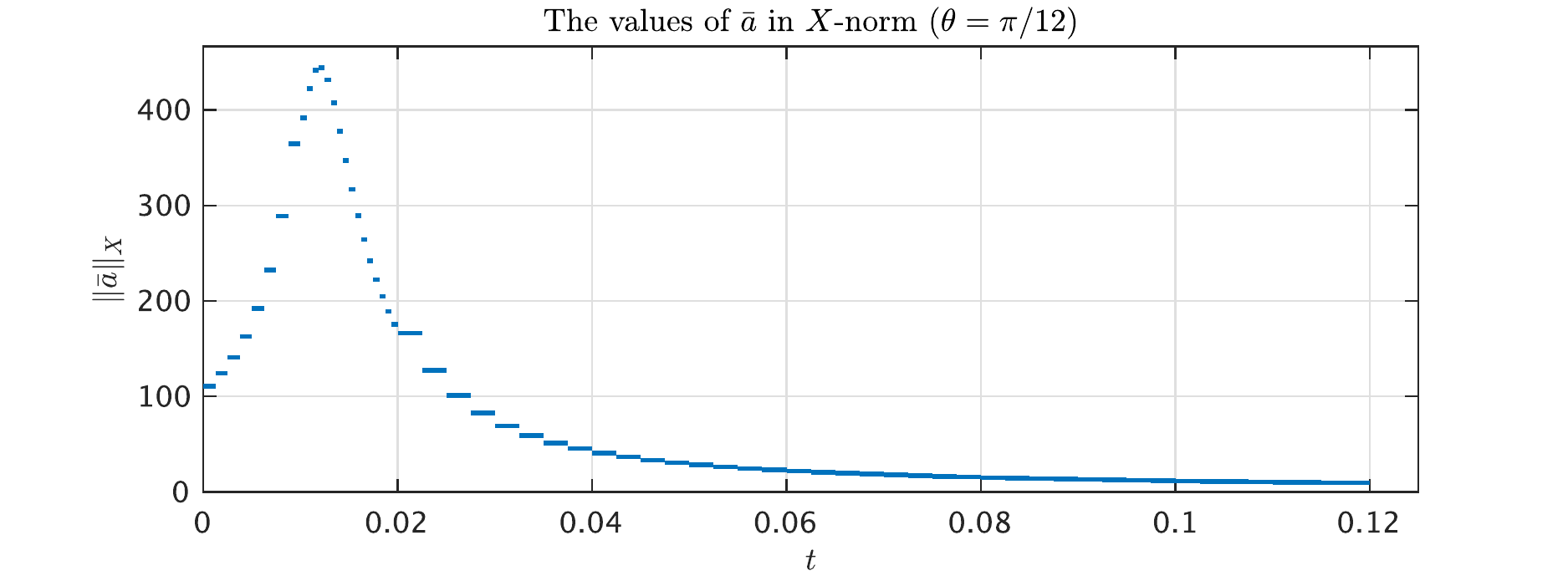}
		\includegraphics[width=0.49\textwidth]{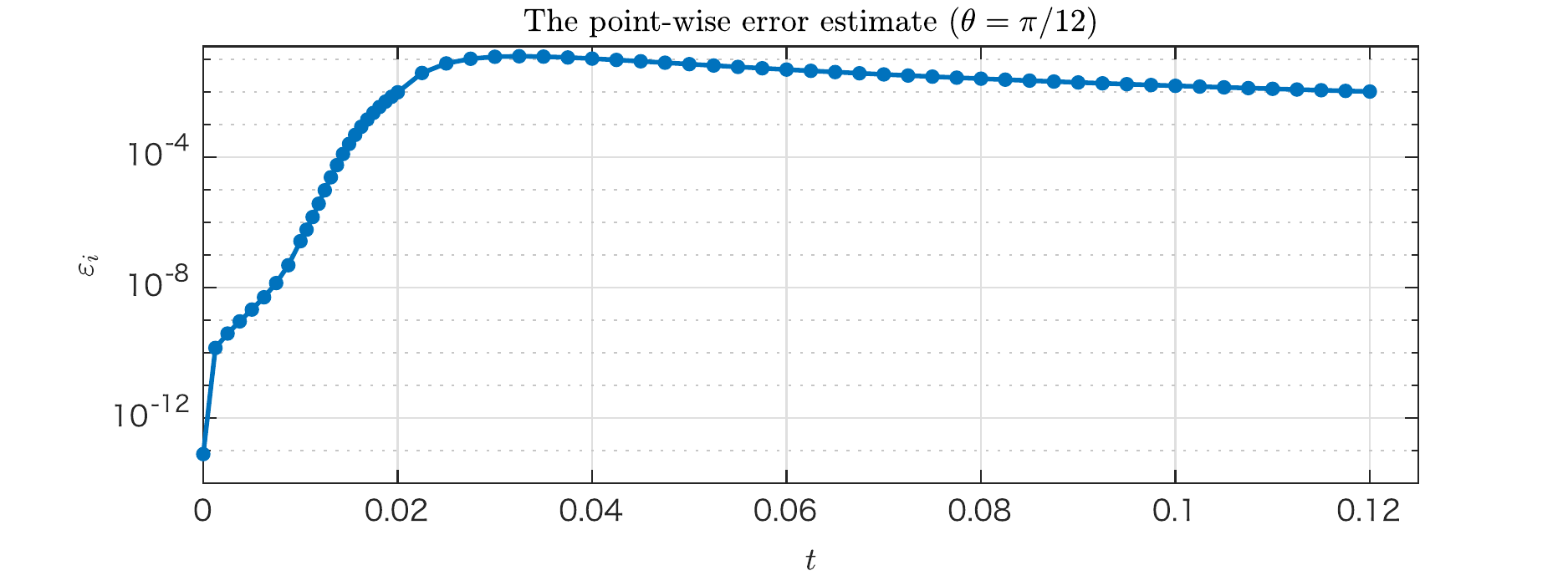}
		$(d)$ $\theta = \pi/12$: The step size is adapted as $h=1.25\times 10^{-3}$ $(0\le t\le 0.01)$, $h=6.25\times 10^{-4}$ $(0.01\le t\le 0.02)$, and $h=2.5\times 10^{-3}$ $(t\ge 0.02)$. We set $N=19$ and $n=15$.
		We also adapt $m=0$ $(0\le t\le 0.02)$ and $m=2$ $(t\ge 0.02)$.
	\end{minipage}\\
	\caption{Results of rigorous integration for $(a)$ $\theta=\pi/3$ , $(b)$ $\pi/4$, $(c)$ $\pi/6$ and $(d)$ $\pi/12$.}
	\label{fig:fig5}
\end{figure}

\section*{Conclusion}
%%%%%%%%%%%%%%%%

In this paper, we introduced a computational method for computing rigorous local inclusions of solutions of the Cauchy problems for the nonlinear heat equation \eqref{eq:cnheq} for complex time values. The proof is constructive and provides explicit bounds for the inclusion of the solutions of the Cauchy problem rewritten as $F(a)=0$ on the Banach space $X=C(J;\ell^1)$. The idea is to show that a simplified Newton operator $T:X \to X$ of the form $T(a)=a - \mathscr{A}F(a)$ has a unique fixed point by the Banach fixed point theorem. The rigorous enclosure of the fixed point yields the solution of the Cauchy problem. The construction of the solution map operator $\mathscr{A}$ is guaranteed with the theory of Section~\ref{sec:solution_operator} by verifying the hypotheses of Theorem~\ref{thm:sol_map}. More explicitly, this requires first computing rigorously with Chebyshev series the solutions of the linearized problems \eqref{eq:finite_dim_variational_problem} and \eqref{eq:finite_dim_adjoint_variational_problem}, and second proving the existence of the evolution operator on the tail. Once the hypotheses are verified successfully, the solution map operator $\mathscr{A}$ exists, is defined as in \eqref{eq:definition_of_scriptA}, and the theory of Section~\ref{sec:Localinclusion} is used to enclose rigorously the fixed point of $T$, and therefore the local inclusion in time of the Cauchy problem. Section~\ref{sec:timestepping} then introduces a method for applying iteratively the approach to compute solutions over long time intervals. This technique is directly used to prove Theorem~\ref{thm:branching} about the existence of a branching singularity in the nonlinear heat equation \eqref{eq:cnheq}. Afterwards, in Section~\ref{sec:center-stable-manifold}, we introduced an approach based on the Lyapunov-Perron method for calculating part of a center-stable manifold, which is then used to prove that an open set of solutions of the given Cauchy problem converge to zero, hence yielding the global existence in time of the solutions. This method is used to prove our second main result, namely Theorem~\ref{thm:GE}. The results are presented in details in Section~\ref{sec:Numerical_results}. 

Finally, we conclude this paper by discussing some potential extensions and remained problems for the complex-valued nonlinear heat equation \eqref{eq:cnheq}. We believe that our methodology of rigorous numerics could be extended to more general time-dependent PDEs. It is interesting to rigorously integrate the solution orbit of high (space) dimensional PDEs.
Furthermore, the case where the time is pure imaginary offers us an interesting problem. The computation in \cite{Cho2016} suggests strongly that the solution exists globally in time and decays slowly toward the zero solution. Since the semigroup generator becomes conservative in that case, our method in this paper, which depends on the existence of dissipation, cannot be used. 
%The reader is invited to challenge this case.

\section*{Acknowledgement}
The work of the first author was supported  in part by JSPS Grant-in-Aid for Early-Career Scientists [18K13453] and 	
Grant-in-Aid for Scientific Research (B) [16H03950].
The third author was supported in part by the grant from NIH [T32 NS007292].

\bibliographystyle{abbrv}
%\bibliography{papers}
\bibliography{papers,BIBInvariantManifold,Bib_complexblowup}
\end{document}